\documentclass[10pt, a4paper]{article}

\usepackage{fullpage}
\usepackage{amsmath}
\usepackage{amsthm}
\usepackage{amsfonts}
\usepackage{tikz-cd}
\usepackage{chngcntr}
\usepackage[utf8x]{inputenc}
\usepackage{enumerate}
\usepackage[english]{babel}
\usepackage{amssymb}
\usepackage{esint}
\usepackage{hyperref}
\usepackage{verbatim}
\usepackage{mathtools}

\numberwithin{equation}{section}

\theoremstyle{plain}
\newtheorem{thm}[equation]{Theorem}
\newtheorem{lemma}[equation]{Lemma}
\newtheorem{cor}[equation]{Corollary}
\newtheorem{prop}[equation]{Proposition}
\theoremstyle{definition}
\newtheorem{dfn}[equation]{Definition}
\theoremstyle{remark}

\newtheorem{rmk}[equation]{Remark}

\newcommand{\circleland}{ 
	\mathbin{
		\mathchoice
		{\buildcircleland{\displaystyle}}
		{\buildcircleland{\textstyle}}
		{\buildcircleland{\scriptstyle}}
		{\buildcircleland{\scriptscriptstyle}}
	} 
}

\newcommand\buildcircleland[1]{%
	\begin{tikzpicture}[baseline=(X.base), inner sep=0, outer sep=0]
	\node[draw,circle] (X)  {$#1\land$};
	\end{tikzpicture}%
}

\title{On the Bach and Einstein equations in presence of a field}
\author{Andrea Anselli\thanks{\url{andrea.anselli@unimi.it}}\\ Dipartimento di Matematica \textquotedblleft Federigo Enriques\textquotedblright, Universit\'{a} degli Studi di Milano,\\
	Via Cesare Saldini 50, 20133 Milano, Italia.\thanks{No longer affiliated}}
\date{}

\begin{document}
	
	\maketitle

	\begin{abstract}
		The aim of this paper is to introduce and justify a possible generalization of the classic Bach field equations on a four dimensional smooth manifold $M$ in presence of field $\varphi$, that in this context is given by a smooth map with source $M$ and target another Riemannian manifold. Those equations are characterized by the vanishing of a two times covariant, symmetric, traceless and conformally invariant tensor field, called $\varphi$-Bach tensor, that in absence of the field $\varphi$ reduces to the classic Bach tensor. We provide a variational characterization for $\varphi$-Bach flat manifolds and we do the same also for harmonic-Einstein manifolds, i.e., solutions of the Einstein field equations with source the conservative field $\varphi$. We take the opportunity to discuss a generalization of some related topics: the Yamabe problem, the image of the scalar curvature map, warped product solutions and static manifolds.
	\end{abstract}

	\tableofcontents
	
	\section{Introduction}
	
	The Einstein field equations on a four dimensional Lorentzian manifolds $(M,g)$ are given by
	\begin{equation*}\label{einst field equation intro}
	G+\Lambda g=\alpha T,
	\end{equation*}
	where
	\begin{itemize}
		\item $G$ is the Einstein tensor of $(M,g)$, that is,
		\begin{equation*}
		G:=\mbox{Ric}-\frac{S}{2}g,
		\end{equation*}
		where $\mbox{Ric}$ and $S$ denote, respectively, the Ricci and the scalar curvature of $(M,g)$;
		\item $T$ is the stress-energy tensor, a symmetric two times covariant tensor, supposed to represent the density of all the energies, momenta and stresses of the sources, that has to be divergence free in order to satisfy the conservation laws, also called equations of motion (geometrically, the fact that $T$ is divergence free follows from the fact that the Einstein tensor is divergence free, since Schur's identity holds).
		\item $\Lambda\in\mathbb{R}$ is the cosmological constant;
		\item $\alpha$ is the positive constant given by
		\begin{equation}\label{alpha per eq campo einstein intro}
		\alpha=\frac{8\pi \mathcal{G}}{c^4},
		\end{equation}
		where $\mathcal{G}$ is Newton's gravitational constant and $c$ is the speed of light in vacuum.
	\end{itemize}
	The study of solutions of the Einstein field equations in vacuum, i.e.,
	\begin{equation}\label{einst equation vacuum}
		G+\Lambda g=0,
	\end{equation}
	led to the definition of Einstein manifolds and, when the cosmological constant is not included, of Ricci flat manifolds. Those definitions have been easily extended to smooth manifolds of any dimension $m\geq 2$ endowed with any semi-Riemannian (mostly, Riemannian) metric and, at least for mathematicians, they are geometrical objects of great interest and we refer to \cite{B} for an overview on the topic. We just mention that Einstein manifolds of dimension $m\geq 3$ are characterized by the vanishing of the traceless part of the Ricci tensor and, as a consequence of Schur's identity (the Einstein tensor $G$ is divergence free), their scalar curvature $S$ is (locally) constant on $M$. Going back to solutions of \eqref{einst equation vacuum} the cosmological constant $\Lambda$ is given by a constant positive multiple of the scalar curvature.
	
	Almost all the equations of mathematical physics admit a variational characterization and the Einstein equations are not an exception. For simplicity, consider a closed and orientable smooth manifold $M$ of dimension $m\geq 3$ and denote by $\mathcal{M}$ the set of all the Riemannian metrics on $M$. Let $\mathcal{M}_1$ be the subset of $\mathcal{M}$ of metrics of total volume equal to $1$. The functional of total scalar curvature $\mathcal{S}:\mathcal{M}\to \mathbb{R}$, introduced by D. Hilbert, is given by, for every $g\in\mathcal{M}$,
	\begin{equation*}
		\mathcal{S}(g):=\int_MS_g\mu_g,
	\end{equation*}
	where $S_g$ and $\mu_g$ are, respectively, the scalar curvature and the volume form of $(M,g)$. Critical metrics of $\mathcal{S}$ on $\mathcal{M}$ coincide with Ricci flat metrics on $M$ while critical metrics of $\mathcal{S}$ on $\mathcal{M}_1$ (or equivalently, critical metrics of the normalized total scalar curvature on $\mathcal{M}$) coincide with Einstein metrics on $M$. For the details see \cite{B} or \cite{S}.
	
	Since Einstein metrics have constant scalar curvature one may ask weather or not also the largest class of metrics with constant scalar curvature admit a variational characterization. The answer is yes, a metric $g\in\mathcal{M}$ has constant scalar curvature if and only if is a critical point of $\mathcal{S}$ in $[g]\cap \mathcal{M}_1$ (see Proposition 4.25 of \cite{B}). This led naturally to the study of the famous Yamabe problem, that consists in finding pointwise conformal metrics with constant scalar curvature on a Riemannian manifold. We refer to \cite{LP} for a detailed proof of its solution.

	Another interesting feature of the scalar curvature is that we are able to characterize its image when we look at it as a map $g\in\mathcal{M}\mapsto S_g$, see \cite{FM} or Section 4.E of \cite{B} for an overview. The key step is to prove the surjectivity of its linearization at $g\in\mathcal{M}$, or equivalently, the injectivity of its adjoint. We have that if $u$ belongs to kernel of the adjoint map then either $(M,g)$ is Ricci flat (and in this case $u$ is constant) or the product $M\times \mathbb{R}$, endowed with the metric $g-u^2dt\otimes dt$, where $t$ is the coordinate on $\mathbb{R}$, is an Einstein manifold outside the zero locus of $u$, where it degenerates. In the latter case, if $(M,g)$ is a three-dimensional Riemannian manifold, then the Lorentzian warped product $M\times_u \mathbb{R}$ is a static spacetime (i.e., it admits a timelike and irrotational Killing vector field) that solves the vacuum Einstein field equations. In \cite{C}, J. Corvino studied the same problem for complete non-compact Riemannian manifolds.\\
	
	In some circumstances one may be interested in field equations for the metric $g$ on $M$ that are conformal invariant, meaning that if $g$ is a solution then every metric that is pointwise conformal equivalent to $g$, i.e., every metric in the conformal equivalence class $[g]$, is a solution too. A disadvantage of the Einstein field equations and, consequently of Einstein manifolds, is that they are not, in general, conformally invariant. Although it is possible to study conformally Einstein manifolds, that are given by semi-Riemannian manifolds that after a pointwise conformal change of metric are Einstein manifolds, in this work we are more interested to a different approach. A century ago, in \cite{Ba}, R. Bach introduced  a two times covariant symmetric tensor $B$ on a semi-Riemannian manifold $(M,g)$, nowadays called Bach tensor. The Bach tensor has some particular properties: it is traceless, quadratic in the Riemann tensor and, furthermore, for four dimensional manifolds it is conformally invariant and divergence free. The Bach field equations in vacuum for a four dimensional Lorentzian manifold $(M,g)$ are given by
	\begin{equation}\label{bach equation intro}
		B=0,
	\end{equation}
	they are the conformally invariant counterpart of \eqref{einst equation vacuum} (since $B$ is traceless we cannot have a cosmological constant). Einstein and conformally-Einstein manifolds solves the Bach equations, hence they admits more solutions than Einstein equations, even with non-constant scalar curvature. Bach flat Riemannian metrics on a four dimensional closed orientable smooth manifold $M$ admit the following variational characterization: they are critical points of the functional
	\begin{equation}\label{Bach functional intro}
		\mathcal{B}:\mathcal{M}\to \mathbb{R}, \quad \mathcal{B}(g):=\int_M|W_g|^2_g\mu_g,
	\end{equation}
	where $W_g$ is the Weyl tensor of $(M,g)$, the \textquotedblleft conformal invariant part of the Riemann tensor\textquotedblright. Since the $(1,3)$ version of the Weyl tensor is a conformal invariant tensor it is easy to realize that the functional $\mathcal{B}$ is conformal invariant, i.e., $\mathcal{B}(\widetilde{g})=\mathcal{B}(g)$ for every $\widetilde{g}\in[g]$. This provide an easy way to see that the Bach tensor, the gradient of $\mathcal{B}$, is a conformal invariant tensor for four dimensional manifolds, at least when they are closed, oriented and Riemannian.\\
	
	We now leave the vacuum case, coming to the description of our contribution. One of the easiest way to allow the presence of energy and matter is to consider the presence of a field $\varphi$. More precisely: let $(M,g)$ be a four dimensional Lorentzian manifold and $\varphi:M\to N$ a smooth map, where the target $(N,\eta)$ is a Riemannian manifold. To the field $\varphi$ we can associate a energy-stress tensor $T=T_{\varphi}$, as P. Baird and J. Eells did in \cite{BaE} (see \eqref{stress energy tensor} for its definition). In order to be an admissible energy-stress tensor it must satisfy the conservation laws. A computation shows that in case $\varphi$ is a wave map (i.e., a harmonic map with source a Lorentzian manifold) then $T_{\varphi}$ satisfies the conservation laws, compare with \hyperref[rmk stress energy tensor har e biharm map]{Remark \ref*{rmk stress energy tensor har e biharm map}}. The converse implication holds, that is, a smooth map that satisfies the conservation laws is harmonic, providing that $\varphi$ is a submersion a.e., although is not true in general.
	
	Now the Einstein fields equations for the four dimensional Lorentzian manifold $(M,g)$ when the presence of field and matter is described by the wave map $\varphi$ takes the form
	\begin{equation}\label{Einst field eq intro con phi}
	G+\Lambda g=\alpha T_{\varphi},
	\end{equation}
	where $G$ is the Einstein tensor of $(M,g)$, $\alpha$ is given by \eqref{alpha per eq campo einstein intro} and $\Lambda\in\mathbb{R}$ is the cosmological constant.

	Now, as did for the vacuum case, we can forget that $(M,g)$ is a four dimensional manifold and that $\alpha$ is given by \eqref{alpha per eq campo einstein intro} and study semi-Riemannian manifolds $(M,g)$ of dimension $m\geq 3$ satisfying \eqref{Einst field eq intro con phi} for some $\Lambda,\alpha\in\mathbb{R}$ and a wave map $\varphi:M\to N$, where the target $(N,\eta)$ is a Riemannian manifold. We call them harmonic-Einstein manifolds (with respect to $\varphi$ and $\alpha$), see \hyperref[def harm einst mani]{Definition \ref*{def harm einst mani}}. Those class of semi-Riemannian manifold includes the Einstein manifolds and, in dimension $m\geq 3$, also its elements are characterized by the vanishing of the traceless part of a symmetric two times covariant tensor (see \hyperref[prop harm einst soddisfano einst equ]{Proposition \ref*{prop harm einst soddisfano einst equ}}). This tensor, that plays the role of Ricci's tensor for harmonic-Einstein manifolds, is the $\varphi$-Ricci tensor, given by
	\begin{equation*}
		\mbox{Ric}^{\varphi}:=\mbox{Ric}-\alpha \varphi^*\eta
	\end{equation*}
	and first used by R. M\"{u}ller in \cite{M} when dealing with the Ricci-harmonic flow, a combination of the Ricci flow with the heat flow of a smooth map. Notice that in order to define the $\varphi$-Ricci tensor we need a semi-Riemannian metric $g$, a smooth map $\varphi$ and a constant $\alpha$. The trace of the $\varphi$-Ricci tensor is called $\varphi$-scalar curvature and it is given by $S^{\varphi}=S-\alpha |d\varphi|^2$, where $|d\varphi|^2=\mbox{tr}(\varphi^*\eta)$ is the Hilbert-Schmdit norm of the differential of $\varphi$. The generalized Schur's identity \eqref{div of phi Ricci} gives that harmonic-Einstein manifolds of dimension $m\geq 3$ have constant $\varphi$-scalar curvature and the one with vanishing $\varphi$-scalar curvature are called $\varphi$-Ricci flat manifolds. Going back to \eqref{Einst field eq intro con phi}, the cosmological constant $\Lambda$ is a constant positive multiple of the $\varphi$-scalar curvature.
	
	The similarities between the theory of harmonic-Einstein manifolds and the classical theory of Einstein manifolds are not over. Harmonic-Einstein manifolds admits a variational characterization too. Let $M$ be a closed and oriented smooth manifold of dimension $m\geq 3$ and $(N,\eta)$ be a fixed target Riemannian manifold. We denote by $\mathcal{F}$ the set of all the smooth maps $\varphi:M\to N$. Moreover, fix $\alpha\in\mathbb{R}\setminus\{0\}$. Then harmonic-Einstein structures on $M$, with respect to $\alpha\in\mathbb{R}$, are critical points of the functional of total $\varphi$-scalar curvature
	\begin{equation*}
		\mathcal{S}:\mathcal{M}_1\times \mathcal{F}\to\mathbb{R}, \quad \mathcal{S}(g,\varphi):=\int_MS^{\varphi}_g\mu_g.
	\end{equation*}
	Considering critical points of $\mathcal{S}$ on $\mathcal{M}\times \mathbb{R}$ we obtain $\varphi$-Ricci flat manifolds. Those results are stated and proved in \hyperref[section Variational derivation of the harmonic-Einstein equations]{Section \ref*{section Variational derivation of the harmonic-Einstein equations}}, where we also discuss the what happens for surfaces, see \hyperref[remark per superfici]{Remark \ref*{remark per superfici}}.
	
	Considering, for a fixed $\varphi\in\mathcal{F}$, the restriction of the total $\varphi$-scalar curvature to $\mathcal{M}_1\cap[g]$ its critical points are given by Riemannian metrics of constant $\varphi$-scalar curvature, see \hyperref[prop phi curv scalar const punto critico in classe confome]{Proposition \ref*{prop phi curv scalar const punto critico in classe confome}}. This led us to formulate the $\varphi$-Yamabe problem: on a compact Riemannian manifold $(M,g)$ of dimension $m\geq 3$ there exists $\widetilde{g}\in [g]$ such that $S^{\varphi}_{\widetilde{g}}$ is constant? In \hyperref[section scalar curvature restricted to conformal classes of metrics]{Section \ref*{section scalar curvature restricted to conformal classes of metrics}} we just give the definition of the $\varphi$-Yamabe constant, the first step in the solution of this problem. We expect the solution of the $\varphi$-Yamabe problem to be very similar and very close to the one of the classic Yamabe problem but we postpone it to some future work.
	
	Proceeding with the similarities, in \hyperref[section The linearization of the]{Section \ref*{section The linearization of the}} we study the adjoint of the linearization of the $\varphi$-scalar curvature map $(g,\varphi)\in\mathcal{M}\times \mathcal{F}\mapsto S^{\varphi}_g$ on a smooth manifold $M$. We show that a function $u$ belongs to its kernel if either $(M,g)$ is $\varphi$-Ricci flat (and in this case $u$ is constant) or the product $M\times \mathbb{R}$, endowed with the metric $g-u^2dt\otimes dt$ is a harmonic-Einstein manifold with respect to the extension $\bar{\varphi}$ of $\varphi$ to $M\times \mathbb{R}$, outside the zero locus of $u$, where it degenerates (see \hyperref[prop aggiunto iniettivo]{Proposition \ref*{prop aggiunto iniettivo}}). This is the first step in the study of the image of the $\varphi$-scalar curvature map, that we also postpone to some future work. \hyperref[section Warped products]{Section \ref*{section Warped products}} is devoted to the study harmonic-Einstein manifolds that arise as semi-Riemannian warped product metric with respect to the lifting $\bar{\varphi}$ of a smooth map $\varphi$ with source the base of the warped product, see \hyperref[thm harm einst warp prod]{Theorem \ref*{thm harm einst warp prod}}. An interesting example of the above situation is given $\varphi$-static harmonic-Einstein manifold, a concept related to the one of static spacetime: see \hyperref[def phi static]{Definition \ref*{def phi static}} and the remarks below.
	
	Finally, we now come to conformal invariant theories in presence of the field $\varphi$. At this point the definition of conformally harmonic-Einstein is straightforward: we refer to \cite{A} for it and for results on regarding those manifolds. In this article we do not want to focus on them, instead we are interested in a theory similar to the one of R. Bach mentioned above. To obtain this goal we need to define on a semi-Riemannian manifold $(M,g)$ endowed with a smooth map $\varphi$ a two times covariant symmetric tensor that, at least for four dimensional Lorentzian manifolds, is conformally invariant and that when $\varphi$ is constant reduces to the Bach tensor $B$. This tensor, that appears for the first time in \cite{A}, is the $\varphi$-Bach tensor $B^{\varphi}$, see \hyperref[section phi bach]{Section \ref*{section phi bach}} for its definition. The proof of its conformal invariance is contained in \hyperref[cor phi bach invariante conforme per m uguale a 4]{Corollary \ref*{cor phi bach invariante conforme per m uguale a 4}} and it is the main result of \hyperref[section Transformation laws of curvature under a conformal change of metric]{Section \ref*{section Transformation laws of curvature under a conformal change of metric}}. Furthermore, the $\varphi$-Bach tensor is traceless for four dimensional manifolds (for higher dimensions see \eqref{traccia phi bach}) and, analogously to the $\varphi$-Ricci tensor, it depends on a scale factor $\alpha\in\mathbb{R}$. The equation
	\begin{equation*}
		B^{\varphi}=0
	\end{equation*}  
	seems a natural generalization of \eqref{bach equation intro}, indeed four dimensional harmonic-Einstein and conformally harmonic-Einstein manifolds are solutions of the above.
	
	To justify that the above equation is indeed the generalization of the Bach field equation in the presence of the field $\varphi$ the only thing that remains is to provide a variational characterization. In \hyperref[section phi curvature]{Section \ref*{section phi curvature}}, where we briefly recall the definitions and the properties of the $\varphi$-curvatures (whose detailed proof is contained in \cite{A}), one may find the definition of the $\varphi$-Weyl tensor $W^{\varphi}$, i.e., a generalization of the Weyl tensor in presence of the field $\varphi$. Assuming that the $\varphi$-Weyl tensor is the \textquoteleft right\textquotedblright generalization of the Weyl tensor, one may think that $\varphi$-Bach flat Riemannian metrics on a four dimensional closed orientable smooth manifold $M$ are critical points of the functional
	\begin{equation*}
	\mathcal{M}\times \mathcal{F}\to \mathbb{R}, \quad (g,\varphi)\mapsto\int_M|W^{\varphi}_g|^2_g\mu_g.
	\end{equation*}
	This is not true: the appropriate functional is given by
	\begin{equation}\label{phi bach functional intro}
		\mathcal{B}:\mathcal{M}\times \mathcal{F}\to \mathbb{R}, \quad \mathcal{B}(g,\varphi)=\int_M\left(S_2(A^{\varphi}_g)-\frac{\alpha}{2}|\tau^g(\varphi)|^2\right)\mu_g,
	\end{equation}
	where $\tau^g(\varphi)$ denotes the tension field of $\varphi$ and $S_2(A^{\varphi}_g)$ denotes the second elementary symmetric polynomial in the eigenvalues of $A^{\varphi}_g$, the $\varphi$-Schouten tensor of $(M,g)$ (see  \hyperref[section phi curvature]{Section \ref*{section phi curvature}} for the definition of the $\varphi$-Schouten tensor and \hyperref[remark per S 2 schouten]{Remark \ref*{remark per S 2 schouten}} the one of $S^2(A^{\varphi}_g)$). By choosing as $\varphi$ any constant map, for every $g\in\mathcal{M}$ the above yields
	\begin{equation*}
		\mathcal{B}(g,\varphi)=\int_MS_2(A_g)\mu_g
	\end{equation*}
	and, in view of the Chern-Gauss-Bonnet formula for four dimensional manifolds, its critical points are the ones of the standard functional \eqref{Bach functional intro}.
	
	The aim of \hyperref[section Variational characterization of four dimensional]{Section \ref*{section Variational characterization of four dimensional}} is to prove \hyperref[theom phi bach flat punti critici bach]{Theorem \ref*{theom phi bach flat punti critici bach}}: on a four dimensional closed and orientable manifold $M$ the pair $(g,\varphi)$ is critical for $\mathcal{B}$ if and only if the $\varphi$-Bach tensor together with another tensor, denoted by $J$, vanish. The tensor $J$, defined in \eqref{def of J 4}, is strictly connected to the bi-tension of $\varphi$ (we can say, in a certain sense, that is a conformal invariant bi-tension for four dimensional manifold). The fact that the equation $B^{\varphi}=0$ is coupled with another one, involving the map $\varphi$, is not a surprise: the same happens for harmonic-Einstein manifolds where the equation for $G+\Lambda g=\alpha T_{\varphi}$ is coupled with $\tau(\varphi)=0$. Instead, the surprising fact is that $J$ is strictly related to the divergence of $\varphi$-Bach and when $\varphi$ is a submersion a.e. the vanishing of $J$ is equivalent to the vanishing of $\mbox{div}(B^{\varphi})$, see \hyperref[rmk phi bach privo di div per m uguale a 4]{Remark \ref*{rmk phi bach privo di div per m uguale a 4}}. Since in the definition \eqref{phi bach functional intro} the total bi-energy of the map $\varphi$, that is,
	\begin{equation*}
	E_2^g(\varphi)=\frac{1}{2}\int_M|\tau^g(\varphi)|\mu_g,
	\end{equation*}
	is involved, the tensor $J$ is related to the bi-tension of $\varphi$.
	
	\section{Conventions, notations and preliminaries}
	
	Let $M$ be a smooth, connected manifold without boundary and of dimension $m\geq 2$. Let $g$ be a pseudo-Riemannian metric on $M$ and $\nabla$ the Levi-Civita connection of $(M,g)$. For the Riemann tensor $\mbox{Riem}$ of $(M,g)$ we use the sign conventions
	\begin{equation*}
		R(X,Y)Z=\nabla_X(\nabla_YZ)-\nabla_Y(\nabla_XZ)-\nabla_{[X,Y]}Z \quad \mbox{ for every } X,Y,Z\in\mathfrak{X}(M),
	\end{equation*}
	where $\mathfrak{X}(M)$ is the $\mathcal{C}^{\infty}(M)$-module of vector fields on $M$ and $[\,,\,]$ denotes the Lie bracket, and
	\begin{equation*}
		\mbox{Riem}(W,Z,X,Y)=g(R(X,Y)Z,W) \quad \mbox{ for every } X,Y,Z,W\in\mathfrak{X}(M).
	\end{equation*}
	
	Let $(N,\eta)$ be a Riemannian manifold of dimension $n$ and $\varphi:M\to N$ a smooth map.
	
 	In order to carry on computations we will mostly use (with some exception in \hyperref[section variations]{Section \ref*{section variations}}) the moving frame formalism introduced by E. Cartan. For a crash course in the Riemannian case see Chapter 1 of \cite{AMR} while for the semi-Riemannian case we refer to Chapter 5 of \cite{CCL}.
 	
 	We fix the indexes ranges
 	\begin{equation*}
 	1\leq i,j,k,t,\ldots \leq m, \quad 1\leq a,b,c,d\ldots \leq n
 	\end{equation*}
 	and from now on we adopt the Einstein summation convention over repeated indexes. In a neighborhood of each point of $M$ we can write
		\begin{equation*}
			g=\eta_{ij}\theta^i\otimes \theta^j
		\end{equation*}
		where $\{\theta^i\}$ is a local $g$-orthonormal coframe and
		\begin{equation}\label{def eta i j}
			\eta_{ij}=\begin{cases}
			0 \quad &\mbox{ if } i\neq j\\
			1 \quad &\mbox{ if } i=j\leq k\\
			-1 \quad &\mbox{ if } i=j>k,
			\end{cases}
		\end{equation}
		where the signature of the metric is $(k,m-k)$. When $g$ is a Riemannian metric then $\eta_{ij}=\delta_{ij}$, the Kronecker delta. Denoting by $\{e_i\}$ the local frame dual to $\{\theta^i\}$ it is easy to see that $\{e_i\}$ is a local $g$-orthonormal frame, i.e., $g(e_i,e_j)=\eta_{ij}$.
		
		The Levi-Civita connection forms $\{\theta^i_j\}$ are characterized by the validity of the first structure equations
		\begin{equation}\label{first struct equation}
			d\theta^i+\theta^i_j\wedge \theta^j=0
		\end{equation}
		and by the relation
		\begin{equation}\label{rel simmetria con gli eta per forme levi civita}
			\eta_{ik}\theta^k_j+\eta_{kj}\theta^k_i=0.
		\end{equation}
		In the following we will use the metric to raise and lower the indexes, for instance
		\begin{equation*}
		\theta_i:=\eta_{ij}\theta^j, \quad \theta_{ij}:=\eta_{ik}\theta_j^k.
		\end{equation*}
		Then the relation \eqref{rel simmetria con gli eta per forme levi civita} is equivalent to the skew-symmetry
		\begin{equation}\label{rel simmetria per forme levi civita}
			\theta_{ij}+\theta_{ji}=0.
		\end{equation}
		
		The curvature forms $\{\Theta^i_j\}$ are given by
		\begin{equation}\label{comp of curv forms are components of Riem}
		\Theta^i_j=\frac{1}{2}R^i_{jkt}\theta^k\wedge \theta^t,
		\end{equation}
		where $R_{ijkt}$ are the components of the $(0,4)$-version of the Riemann tensor $\mbox{Riem}$ of $(M,g)$,
		\begin{equation*}
			\mbox{Riem}=R_{ijkt}\theta^i\otimes \theta^j\otimes \theta^k\otimes \theta^t,
		\end{equation*}
		and they satisfy the second structure equations
		\begin{equation*}
		\Theta^j_i=d\theta^j_i+\theta_{ki}\wedge\theta^{kj},
		\end{equation*}
		where we denoted
		\begin{equation*}
			\theta^{ij}:=\eta^{jk}\theta^i_k, \quad \theta_{ij}=\eta_{ik}\theta_j^k.
		\end{equation*}
		
		The symmetries of the Riemann tensor are given by
		\begin{equation*}
		R_{ijkt}+R_{ijtk}=0, \quad R_{ijkt}+R_{jikt}=0, \quad R_{ijkt}=R_{ktij},
		\end{equation*}
		while the first Bianchi identity is reads
		\begin{equation*}
		R_{ijkt}+R_{iktj}+R_{itjk}=0
		\end{equation*}
		and finally the second Bianchi identity is expressed as
		\begin{equation*}
		R_{ijkt,l}+R_{ijtl,k}+R_{ijlk,t}=0,
		\end{equation*}
		where, for an arbitrary tensor field $T$ of type $(r,s)$
		\begin{equation*}
		T=T_{j_1\ldots j_s}^{i_1\ldots i_r}\theta^{j_1}\otimes \ldots \otimes \theta^{j_s}\otimes e_{i_1}\otimes \ldots \otimes e_{i_r},
		\end{equation*}
		its covariant derivative is defined as the tensor field of type $(r,s+1)$
		\begin{equation*}
		\nabla T=T_{j_1\ldots j_s,k}^{i_1\ldots i_r}\theta^k\otimes \theta^{j_1}\otimes \ldots \otimes \theta^{j_s}\otimes e_{i_1}\otimes \ldots \otimes e_{i_r},
		\end{equation*}
		where the following relation holds
		\begin{equation*}
		T_{j_1\ldots j_s,k}^{i_1\ldots i_r}\theta^k=dT_{j_1\ldots j_s}^{i_1\ldots i_r}-\sum_{t=1}^sT_{j_1\ldots j_{t-1} h j_{t+1}\ldots  j_s}^{i_1\ldots i_r}\theta^{h}_{j_t}+\sum_{t=1}^rT_{j_1\ldots j_s}^{i_1\ldots i_{t-1} h i_{t+1} \ldots i_r}\theta^{i_t}_h.
		\end{equation*}
		It is possible to prove that the following commutation relation holds
		\begin{equation}\label{general commutation rule tensor field}
		T_{j_1\ldots j_s,kt}^{i_1\ldots i_r}=T_{j_1\ldots j_s,tk}^{i_1\ldots i_r}+\sum_{t=1}^sR^h_{j_tkt}T_{j_1\ldots j_{t-1} hj_{t+1}\ldots j_s}^{i_1\ldots i_r}-\sum_{t=1}^rR^{i_t}_{hkt}T_{j_1\ldots j_s}^{i_1\ldots i_{t-1} h i_{t+1}\ldots i_r}.
		\end{equation}
		
		The Ricci tensor $\mbox{Ric}_g=\mbox{Ric}$ of $(M,g)$ is given by the trace of the Riemann tensor. In a local $g$-orthonormal coframe $\{\theta^i\}$
		\begin{equation}\label{ricci tensor trace riem}
			\mbox{Ric}=R_{ij}\theta^i\otimes \theta^j, \quad R_{ij}=\eta^{tk}R_{tikj}.
		\end{equation}
		The scalar curvature $S$ of $(M,g)$, denoted also by $S_g$, is given by the trace of the Ricci tensor, locally
		\begin{equation*}
			S=\eta^{ij}R_{ij}.
		\end{equation*}
		The Riemannian volume element of $(M,g)$ is locally given by
		\begin{equation}\label{riem vol elem con moving frame}
			\mu=\theta^1\land \ldots \land \theta^m.
		\end{equation}
		
		Let $\{E_a\}$, $\{\omega^a\}$, $\{\omega^a_b\}$, $\{\Omega^a_b\}$ be a orthonormal frame, coframe, the respectively Levi-Civita connection forms and curvature forms on an open subset $\mathcal{V}$ on $N$ such that $\varphi^{-1}(\mathcal{V})\subseteq \mathcal{U}$. We set
		\begin{equation*}
		\varphi^*\omega^a=\varphi^a_i\theta^i
		\end{equation*}
		so that the differential $d\varphi$ of $\varphi$, a $1$-form on $M$ with values in the pullback bundle $\varphi^{-1}TN$, can be written as
		\begin{equation*}
		d\varphi=\varphi^a_i\theta^i\otimes E_a.
		\end{equation*}
		The energy density $e(\varphi)$, or $e^g(\varphi)$, of the map $\varphi$ is given by
		\begin{equation}\label{def dens en}
		e(\varphi)=\frac{1}{2}|d\varphi|^2,
		\end{equation}
		where $|d\varphi|^2=\mbox{tr}(\varphi^*\eta)$ is the square of the Hilbert-Schmidt norm of $d\varphi$. The generalized second fundamental tensor of the map $\varphi$ is given by $\nabla d\varphi$, locally
		\begin{equation*}
		\nabla d\varphi=\varphi^a_{ij}\theta^j\otimes \theta^i\otimes E_a,
		\end{equation*}
		where its coefficient are defined according to the rule
		\begin{equation*}
		\varphi^a_{ij}\theta^j=d\varphi^a_i-\varphi^a_k\theta^k_i+\varphi^b_i\omega^a_b.
		\end{equation*}
		The tension field $\tau(\varphi)$, or $\tau^g(\varphi)$, of the map $\varphi$ is the section of $\varphi^{-1}TN$ defined by
		\begin{equation}\label{def tension field}
		\tau(\varphi)=\mbox{tr}(\nabla d\varphi),
		\end{equation}
		locally
		\begin{equation*}
		\tau(\varphi)^a=\eta^{ij}\varphi^a_{ij}.
		\end{equation*}
		The bi-energy density $e_2(\varphi)$, or $e_2^g(\varphi)$, of the map $\varphi$ is defined as
		\begin{equation}\label{def dens bi-ener}
		e_2(\varphi)=\frac{1}{2}|\tau(\varphi)|^2,
		\end{equation}
		and the bi-tension field $\tau_2(\varphi)$, or $\tau^g_2(\varphi)$, of the map $\varphi$ is the section of $\varphi^{-1}TN$ locally given by
		\begin{equation}\label{def bi tension}
		\tau_2(\varphi)^a=\varphi^a_{iijj}-{}^NR^a_{bcd}\varphi^b_i\varphi^c_i\varphi^d_{jj}.
		\end{equation}
		
		Once again, it is possible to prove
		\begin{equation*}
			\varphi^a_{i_1\ldots i_r j k}=\varphi^a_{i_1\ldots i_r k j}+\sum_{p=1}^r R^t_{i_pjk}\varphi^a_{i_1\ \ldots i_{p-1} t i_{p+1} \ldots i_r}-{}^NR^a_{bcd}\varphi^b_{i_1\ldots i_r}\varphi^c_j\varphi^d_k.
		\end{equation*}
		In particular
		\begin{equation}\label{comm rule der terza phi}
		\varphi^a_{ijk}=\varphi^a_{ikj}+R^t_{ijk}\varphi^a_t-{}^NR^a_{bcd}\varphi^b_i\varphi^c_j\varphi^d_k
		\end{equation}
		and
		\begin{equation}\label{comm rule der quarta phi}
		\varphi^a_{ijkt}=\varphi^a_{ijtk}+R^s_{ikt}\varphi^a_{sj}+R^s_{jkt}\varphi^a_{is}-{}^NR^a_{bcd}\varphi^b_{ij}\varphi^c_k\varphi^d_t
		\end{equation}
		hold.
		
		We denote by $\Delta$ the Laplace-Beltrami operator acting on functions $u:M\to \mathbb{R}$, defined as $\Delta u=\mbox{tr}(\mbox{Hess}(u))$, where $\mbox{Hess}(u)=u_{ij}\theta^i\otimes \theta^j$, that is, $\Delta u=\eta^{ij}u_{ij}$. Furthermore, if $A$ is a symmetric two times covariant tensor on $(M,g)$ we set
		\begin{equation*}
			\mathring{A}:=A-\frac{\mbox{tr}(A)}{m}g
		\end{equation*}
		and the symmetric two times covariant tensors $A^2$ and $\Delta A$ are locally given by
		\begin{equation*}
		A^2_{ij}=\eta^{kt}A_{ik}A_{kt}, \quad \Delta A_{ij}=\eta^{kt}A_{ij,kt}.
		\end{equation*}

	\subsection{$\varphi$-Curvatures}\label{section phi curvature}
	Let $(M,g)$ be a semi-Riemannian manifold of dimension $m\geq 2$, $(N,\eta)$ a Riemannian manifold and $\alpha\in\mathbb{R}\setminus\{0\}$. In this section we recall the definition of $\varphi$-curvatures and we state their properties that shall be useful later on. Their proofs in the Riemannian setting are contained in Section 1.2 of \cite{A}. Those proofs can be easily extended to the semi-Riemannian setting.
	
	The $\varphi$-Ricci tensor is defined as
	\begin{equation}\label{def phi Ricci}
		\mbox{Ric}^{\varphi}_g\equiv\mbox{Ric}^{\varphi}:=\mbox{Ric}-\alpha \varphi^*\eta
	\end{equation}
	and the $\varphi$-scalar curvature is given by its trace:
	\begin{equation}\label{def of phi scalar curv}
		S^{\varphi}_g\equiv S^{\varphi}:=\mbox{tr}(\mbox{Ric}^{\varphi}).
	\end{equation}
	Starting from them the we define $\varphi$-Schouten tensor
	\begin{equation}\label{def of phi schouten}
		A^{\varphi}_g\equiv A^{\varphi}:=\mbox{Ric}^{\varphi}-\frac{S^{\varphi}}{2(m-1)}g.
	\end{equation}
	The $\varphi$-Cotton tensor measures the failure of the commutation of the covariant derivatives of the $\varphi$-Schouten tensor and in global notation is given by
	\begin{equation}\label{def of phi Cotton}
		C^{\varphi}_g(X,Y,Z)\equiv C^{\varphi}(X,Y,Z):=\nabla_ZA^{\varphi}(X,Y)-\nabla_YA^{\varphi}(X,Z) \quad \mbox{ for every } X,Y,Z\in\mathfrak{X}(M).
	\end{equation}
	In moving frame notation
	\begin{equation*}
		C^{\varphi}_{ijk}=A^{\varphi}_{ij,k}-A^{\varphi}_{ik,j}
	\end{equation*}
	Clearly the $\varphi$-Cotton tensor is skew-symmetric in the last two entries and satisfies the following Bianchi identity
	\begin{equation}\label{bianchi id phi cotton}
		C^{\varphi}(X,Y,Z)+C^{\varphi}(Y,Z,X)+C^{\varphi}(Z,X,Y)=0 \quad \mbox{ for every } X,Y,Z\in\mathfrak{X}(M).
	\end{equation}
	When $m\geq 3$ we define the $\varphi$-Weyl tensor as
	\begin{equation}\label{def of phi Weyl}
		W_g^{\varphi}=W^{\varphi}:=\mbox{Riem}-\frac{1}{m-2}A^{\varphi}\circleland g,
	\end{equation}
	where $\circleland$ denotes the Kulkarni-Nomizu product of two times covariant symmetric tensors and give rise to a $(0,4)$ tensor that has the same symmetries of the Riemann tensor. Recall that, for $V$ and $T$ two times covariant symmetric tensors,
	\begin{equation*}
		(T\circleland V)(X,Y,Z,W)=T(X,Z)V(Y,W)-T(X,W)V(Y,Z)+T(Y,W)V(X,Z)-T(Y,Z)V(X,W),
	\end{equation*}
	in moving frame notation
	\begin{equation*}
		(T\circleland V)_{ijkt}=T_{ik}V_{jt}-T_{it}V_{jk}+T_{jt}V_{ik}-T_{jk}V_{it}.
	\end{equation*}
	Following P. Baird and J. Eells, see \cite{BaE}, we define the stress-energy tensor of $\varphi$ (with a different sign convention) by
	\begin{equation}\label{stress energy tensor}
		T^g\equiv T:=\varphi^*\eta-\frac{|d\varphi|^2}{2}g.
	\end{equation}
	The generalized Schur's identity is given by
	\begin{equation}\label{div of phi Ricci}
		\mbox{div}(\mbox{Ric}^{\varphi})=\frac{1}{2}dS^{\varphi}-\alpha \mbox{div}(T).
	\end{equation}
	The divergence of $\varphi$-Weyl is related to the $\varphi$-Cotton tensor as follows, in a local $g$-orthonormal coframe,
	\begin{equation}\label{div di phi Weyl}
		\eta^{ts}W^{\varphi}_{tijk,s}=\frac{m-3}{m-2}C^{\varphi}_{ikj}+\alpha (\varphi^a_{ij}\varphi^a_k-\varphi^a_{ik}\varphi^a_j)+\frac{\alpha}{m-2}\tau(\varphi)^a(\varphi^a_j\eta_{ik}-\varphi^a_k\eta_{ij}),
	\end{equation}
	The traces of $\varphi$-Cotton (with respect to the first and the second entries) and of $\varphi$-Weyl (with respect to the first and the third entries) are given by, respectively, 
	\begin{equation}\label{traccia phi cotton}
		\mbox{tr}(C^{\varphi})=\alpha \mbox{div}(T), \quad \eta^{jk}C^{\varphi}_{jki}=\alpha \mbox{div}(T)_i
	\end{equation}
	and
	\begin{equation}\label{traccia phi weyl}
		\mbox{tr}(W^{\varphi})=\alpha \varphi^*\eta, \quad \eta^{kt}W^{\varphi}_{kitj}=\alpha\varphi^a_i\varphi^a_j.
	\end{equation}
	
	\subsection{The $\varphi$-Bach tensor}\label{section phi bach}
	
	It remain another $\varphi$-curvature to introduce: the $\varphi$-Bach tensor $B^{\varphi}_g\equiv B^{\varphi}$, whose components are given by, in a local $g$-orthonormal coframe,
	\begin{equation}\label{def phi bach}
	\begin{aligned}
	(m-2)B^{\varphi}_{ij}=&\eta^{kt}C^{\varphi}_{ijk,t}+(R^{\varphi})^{tk}(W^{\varphi}_{tikj}-\alpha \varphi^a_t\varphi^a_i\eta_{jk})\\
	&+\alpha\left(\varphi^a_{ij}\tau(\varphi)^a-\tau(\varphi)^a_j\varphi^a_i-\frac{1}{m-2}|\tau(\varphi)|^2\eta_{ij}\right).
	\end{aligned}
	\end{equation}
	It is not immediate to see but the $\varphi$-Bach tensor is symmetric and this is due to the validity of
	\begin{equation}\label{div of phi cotton}
	\eta^{kt}C^{\varphi}_{kij,t}=\alpha [\varphi^a_k((R^{\varphi})^k_{i}\varphi^a_j-(R^{\varphi})^k_{j}\varphi^a_i)+\tau(\varphi)^a_i\varphi^a_j-\tau(\varphi)^a_j\varphi^a_i].
	\end{equation}
	Furthermore, its trace is given by
	\begin{equation}\label{traccia phi bach}
	(m-2)\mbox{tr}(B^{\varphi})=\alpha\frac{m-4}{m-2}|\tau(\varphi)|^2.
	\end{equation}
	
	In the following Proposition we provide an alternative, but equivalent, definition for $\varphi$-Bach that shall be useful in the proof of \hyperref[lemma variazione bach operator rispetto metrica]{Lemma \ref*{lemma variazione bach operator rispetto metrica}}.
	\begin{prop}
		In a local $g$-orthonormal coframe the components of the $\varphi$-Bach tensor can be written as
		\begin{equation}\label{scritt alt phi bach}
		\begin{aligned}
		(m-2)B^{\varphi}_{ij}=&\Delta R^{\varphi}_{ij}-\frac{m-2}{2(m-1)}S^{\varphi}_{ij}-\frac{m-4}{m-2}(R^{\varphi})^2_{ij}-\frac{m}{(m-1)(m-2)}S^{\varphi}R^{\varphi}_{ij}+2R_{kitj}(R^{\varphi})^{kt}\\
		&+\left(\frac{(S^{\varphi})^2}{(m-1)(m-2)}-\frac{\Delta S^{\varphi}}{2(m-1)}-\frac{1}{m-2}|\mbox{Ric}^{\varphi}|^2\right)\eta_{ij}\\
		&+\alpha\left[2\varphi^a_{ij}\tau(\varphi)^a-\frac{1}{m-2}|\tau(\varphi)|^2\eta_{ij}-((R^{\varphi})^k_i\varphi^a_j+(R^{\varphi})^k_j\varphi^a_i)\varphi^a_k\right].
		\end{aligned}
		\end{equation}
		In particular, if $m=4$,
		\begin{equation}\label{scritt alt phi bach m=4}
		\begin{aligned}
		B^{\varphi}_{ij}=&\frac{1}{2}\Delta R^{\varphi}_{ij}-\frac{1}{6}S^{\varphi}_{ij}-\frac{1}{3}S^{\varphi}R^{\varphi}_{ij}+R_{kitj}(R^{\varphi})^{kt}+\left(\frac{(S^{\varphi})^2}{12}-\frac{\Delta S^{\varphi}}{12}-\frac{1}{4}|\mbox{Ric}^{\varphi}|^2\right)\eta_{ij}\\
		&+\alpha\left[\varphi^a_{ij}\tau(\varphi)^a-\frac{1}{4}|\tau(\varphi)|^2\eta_{ij}-\frac{1}{2}((R^{\varphi})^k_i\varphi^a_j+(R^{\varphi})^k_j\varphi^a_i)\varphi^a_k\right].
		\end{aligned}
		\end{equation}
	\end{prop}
	\begin{proof}
		Using the definitions \eqref{def of phi Cotton} and \eqref{def of phi schouten} we get
		\begin{equation*}
		C^{\varphi}_{ijk,t}=R^{\varphi}_{ij,kt}-R^{\varphi}_{ik,jt}-\frac{S^{\varphi}_{kt}}{2(m-1)}\eta_{ij}+\frac{S^{\varphi}_{jt}}{2(m-1)}\eta_{ik}.
		\end{equation*}
		Then
		\begin{equation}\label{primo pezzo div cotton per riscrivere phi bach}
		\eta^{kt}C^{\varphi}_{ijk,t}=\Delta R^{\varphi}_{ij}-\eta^{kt}R^{\varphi}_{ik,jt}-\frac{\Delta S^{\varphi}}{2(m-1)}\eta_{ij}+\frac{S^{\varphi}_{ij}}{2(m-1)}.
		\end{equation}
		
		The following relation holds
		\begin{equation}\label{pezzo con div ricci per div cotton per riscrivere phi bach}
		\eta^{kt}R^{\varphi}_{ik,jt}=\frac{1}{2}S^{\varphi}_{ij}-R_{ikjt}(R^{\varphi})^{kt}+(R^{\varphi})^2_{ij}+\alpha\left( (R^{\varphi})^s_i\varphi^a_j\varphi^a_s-\tau(\varphi)^a_j\varphi^a_i-\tau(\varphi)^a\varphi^a_{ij}\right).
		\end{equation}
		To prove \eqref{pezzo con div ricci per div cotton per riscrivere phi bach} first notice that, commutating the indexes,
		\begin{equation*}
		R^{\varphi}_{ik,jt}=R^{\varphi}_{ik,tj}+R^s_{ijt}R^{\varphi}_{sk}+R^s_{kjt}R^{\varphi}_{is},
		\end{equation*}
		hence contracting the above
		\begin{equation*}
		\eta^{kt}R^{\varphi}_{ik,jt}=\mbox{div}(\mbox{Ric}^{\varphi})_{i,j}+R^s_{ijt}(R^{\varphi})^t_s+R^s_jR^{\varphi}_{is}.
		\end{equation*}
		Using \eqref{div of phi Ricci} and \eqref{def phi Ricci} from the above we infer
		\begin{equation*}
		\eta^{kt}R^{\varphi}_{ik,jt}=\left(\frac{1}{2}S^{\varphi}_i-\alpha\tau(\varphi)^a\varphi^a_i\right)_j+R^s_{ijt}(R^{\varphi})^t_s+(R^{\varphi})^2_{ij}+\alpha (R^{\varphi})^s_i\varphi^a_j\varphi^a_s,
		\end{equation*}
		that is \eqref{pezzo con div ricci per div cotton per riscrivere phi bach}. By plugging \eqref{pezzo con div ricci per div cotton per riscrivere phi bach} into \eqref{primo pezzo div cotton per riscrivere phi bach} we get
		\begin{equation}\label{primo pezzo def div cotton per riscrivere phi bach}
		\begin{aligned}
		\eta^{kt}C^{\varphi}_{ijk,t}=&\Delta R^{\varphi}_{ij}-\frac{m-2}{2(m-1)}S^{\varphi}_{ij}+R_{kitj}(R^{\varphi})^{kt}-(R^{\varphi})^2_{ij}-\frac{\Delta S^{\varphi}}{2(m-1)}\eta_{ij}\\
		&+\alpha\left( \tau(\varphi)^a_j\varphi^a_i+\tau(\varphi)^a\varphi^a_{ij}-(R^{\varphi})^s_i\varphi^a_j\varphi^a_s\right).
		\end{aligned}
		\end{equation}
		Now, using \eqref{def of phi Weyl} and \eqref{def of phi schouten},
		\begin{equation}\label{primo pezzo def weyl per riscrivere phi bach}
		W^{\varphi}_{ikjt}(R^{\varphi})^{kt}=R_{ikjt}(R^{\varphi})^{kt}-\frac{1}{m-2}\left[\frac{m}{m-1}S^{\varphi}R^{\varphi}_{ij}-2(R^{\varphi})^2_{ij}+\left(|\mbox{Ric}^{\varphi}|^2-\frac{(S^{\varphi})^2}{m-1}\right)\eta_{ij}\right].
		\end{equation}
		By plugging \eqref{primo pezzo def div cotton per riscrivere phi bach} and \eqref{primo pezzo def weyl per riscrivere phi bach} into \eqref{def phi bach}, after some computations, we get \eqref{scritt alt phi bach}.
	\end{proof}
	
	Motivated by the fact that the Bach tensor is divergence free for four dimensional manifolds in the following Proposition we evaluate the divergence of $\varphi$-Bach. The following Proposition will be useful also in the proof of \hyperref[lemma variazione bach operator rispetto mappa]{Lemma \ref*{lemma variazione bach operator rispetto mappa}}.
	\begin{prop}
		In a local $g$-orthonormal coframe the components of the divergence of the $\varphi$-Bach tensor are given by
		\begin{equation}\label{div di phi Bach}
		\begin{aligned}
		\mbox{div}(B^{\varphi})_i&=\frac{m-4}{m-2}(R^{\varphi}_{jk}C^{\varphi}_{jki}+\alpha(\tau(\varphi)^a_i+R^{\varphi}_{ij}\varphi^a_j)\tau(\varphi)^a)\\
		&+\alpha \varphi^a_i\left[\frac{mS^{\varphi}}{(m-1)(m-2)}\tau(\varphi)^a-\frac{m-2}{2(m-1)}S^{\varphi}_j\varphi^a_j-2R^{\varphi}_{jk}\varphi^a_{jk}+2\tau(\varphi)^b\varphi^b_j\varphi^a_j-\tau_2(\varphi)^a\right].
		\end{aligned}
		\end{equation}
		By setting
		\begin{equation*}
		(J_m)^a=J^a:=\frac{mS^{\varphi}}{(m-1)(m-2)}\tau(\varphi)^a-\frac{m-2}{2(m-1)}S^{\varphi}_j\varphi^a_j-2R^{\varphi}_{jk}\varphi^a_{jk}+2\tau(\varphi)^b\varphi^b_j\varphi^a_j-\tau_2(\varphi)^a
		\end{equation*}
		the above can be written as
		\begin{equation*}
		\mbox{div}(B^{\varphi})_i=\frac{m-4}{m-2}(R^{\varphi}_{jk}C^{\varphi}_{jki}+\alpha(\tau(\varphi)^a_i+R^{\varphi}_{ij}\varphi^a_j)\tau(\varphi)^a)+\alpha J_m\varphi^a_i.
		\end{equation*}
	\end{prop}
	\begin{rmk}\label{rmk phi bach privo di div per m uguale a 4}
		In particular, for $m=4$,
		\begin{align*}
		\mbox{div}(B^{\varphi})_i=&\alpha J^a\varphi^a_i,
		\end{align*}
		with
		\begin{equation}\label{def of J 4}
		J^a=(J_4)^a=\frac{2}{3}S^{\varphi}\tau(\varphi)^a-\frac{1}{3}S^{\varphi}_j\varphi^a_j-2R^{\varphi}_{jk}\varphi^a_{jk}+2\tau(\varphi)^b\varphi^b_j\varphi^a_j-\tau_2(\varphi)^a.
		\end{equation}
		Then, if $\varphi$-Bach is divergence free and $\varphi$ is a submersion a.e.,
		\begin{equation*}
		J=0.
		\end{equation*}
	\end{rmk}
	\begin{proof}
		We decompose the $\varphi$-Bach tensor as follows
		\begin{equation*}
		(m-2)B^{\varphi}_{ij}=\alpha L_{ij}+M_{ij}+N_{ij},
		\end{equation*}
		where
		\begin{equation*}
		L_{ij}:=\varphi^a_{kk}\varphi^a_{ij}-\frac{1}{m-2}|\tau(\varphi)|^2\delta_{ij},
		\end{equation*}
		\begin{equation*}
		M_{ij}:=R^{\varphi}_{tk}W^{\varphi}_{tikj}
		\end{equation*}
		and
		\begin{equation*}
		N_{ij}:=C^{\varphi}_{ijk,k}-\alpha (R^{\varphi}_{kj}\varphi^a_k+\varphi^a_{kkj})\varphi^a_i.
		\end{equation*}
		We proceed by evaluating separately the divergences of $L$, $M$ and $N$ and then we combine them all to obtain \eqref{div di phi Bach}.
		
		\begin{itemize}
			\item Using the commutation rule \eqref{comm rule der terza phi} and the definition of $\varphi$-Ricci \eqref{def phi Ricci} we easily get
			\begin{equation}\label{div L}
			\mbox{div}(L)_j=\frac{m-4}{m-2}\varphi^a_{kkj}\varphi^a_{tt}+\varphi^a_{kki}\varphi^a_{ij}+{}^NR^a_{bcd}\varphi^b_i\varphi^c_i\varphi^d_j\varphi^a_{kk}+R^{\varphi}_{tj}\varphi^a_t\varphi^a_{kk}+\alpha \varphi^b_t\varphi^b_j\varphi^a_t\varphi^a_{kk}.
			\end{equation}
			\item Clearly
			\begin{equation*}
			\mbox{div}(M)_j=R^{\varphi}_{tk,i}W^{\varphi}_{tikj}+R^{\varphi}_{tk}W^{\varphi}_{tikj,i}.
			\end{equation*}
			Using the definition of $\varphi$-Schouten and $\varphi$-Cotton and the symmetries of $\varphi$-Weyl and $\varphi$-Schouten we infer
			\begin{equation*}
			R^{\varphi}_{tk,i}W^{\varphi}_{tikj}=C^{\varphi}_{tki}W^{\varphi}_{tikj}+\alpha\frac{S^{\varphi}_i}{2(m-1)}\varphi^a_i\varphi^a_j.
			\end{equation*}
			Using \eqref{div di phi Weyl} we easily get
			\begin{equation*}
			R^{\varphi}_{tk}W^{\varphi}_{tikj,i}=\frac{m-3}{m-2}R^{\varphi}_{ik}C^{\varphi}_{ikj}+\alpha(R^{\varphi}_{ik}\varphi^a_{ij}\varphi^a_k-R^{\varphi}_{ik}\varphi^a_{ik}\varphi^a_j)+\frac{\alpha}{m-2}(S^{\varphi}\varphi^a_j-R^{\varphi}_{kj}\varphi^a_k)\varphi^a_{tt}.
			\end{equation*}
			Combining the equations above we have
			\begin{equation}\label{div di M}
			\begin{aligned}
			\mbox{div}(M)_j=&C^{\varphi}_{tki}W^{\varphi}_{tikj}+\alpha\frac{S^{\varphi}_i}{2(m-1)}\varphi^a_i\varphi^a_j+\frac{m-3}{m-2}R^{\varphi}_{ik}C^{\varphi}_{ikj}\\
			&+\alpha(R^{\varphi}_{ik}\varphi^a_{ij}\varphi^a_k-R^{\varphi}_{ik}\varphi^a_{ik}\varphi^a_j)+\frac{\alpha}{m-2}(S^{\varphi}\varphi^a_j-R^{\varphi}_{kj}\varphi^a_k)\varphi^a_{tt}.
			\end{aligned}
			\end{equation}
			\item Finally
			\begin{equation}\label{div di N primo conto}
			\mbox{div}(N)_j=C^{\varphi}_{ijk,ki}-\alpha (R^{\varphi}_{kj,i}\varphi^a_k+R^{\varphi}_{kj}\varphi^a_{ki}+\varphi^a_{kkji})\varphi^a_i-\alpha (R^{\varphi}_{kj}\varphi^a_k+\varphi^a_{kkj})\varphi^a_{ii}.
			\end{equation}
			Exchanging the covariant derivatives we obtain
			\begin{equation}\label{per calc div di N termine con div di cotton completo}
			C^{\varphi}_{ijk,ki}=(C^{\varphi}_{kji,k})_i-R^t_{jik}C^{\varphi}_{itk}.
			\end{equation}
			Using \eqref{div of phi cotton} we infer
			\begin{equation}\label{per calc div di N termine con div di cotton}
			\begin{aligned}
			(C^{\varphi}_{kji,k})_i=&\alpha[(R^{\varphi}_{jk,i}\varphi^a_i+R^{\varphi}_{jk}\varphi^a_{ii}-R^{\varphi}_{ik,i}\varphi^a_j-R^{\varphi}_{ik}\varphi^a_{ji})\varphi^a_k+(R^{\varphi}_{jk}\varphi^a_i-R^{\varphi}_{ik}\varphi^a_j)\varphi^a_{ki}]\\
			&+\alpha[\varphi^a_{kkji}\varphi^a_i+\varphi^a_{kkj}\varphi^a_{ii}-\varphi^a_{kkii}\varphi^a_j-\varphi^a_{kki}\varphi^a_{ji}].
			\end{aligned}
			\end{equation}
			Plugging \eqref{per calc div di N termine con div di cotton} into \eqref{per calc div di N termine con div di cotton completo} and then plugging it all into \eqref{div di N primo conto}, with the aid of \eqref{div of phi Ricci}, we get
			\begin{equation*}
			\mbox{div}(N)_j=R^t_{jki}C^{\varphi}_{itk}-\alpha\left[\frac{1}{2}S^{\varphi}_k\varphi^a_k\varphi^a_j-\alpha \varphi^b_{ii}\varphi^b_k\varphi^a_j\varphi^a_k+R^{\varphi}_{ik}\varphi^a_{ij}\varphi^a_k+R^{\varphi}_{ik}\varphi^a_{ik}\varphi^a_j+\varphi^a_{kkii}+\varphi^a_{kki}\varphi^a_{ij}\right].
			\end{equation*}
			Using the decomposition \eqref{def of phi Weyl}, the definition of $\varphi$-Schouten, the symmetries of $\varphi$-Cotton and \eqref{traccia phi cotton} we get
			\begin{equation*}
			R^t_{jki}C^{\varphi}_{itk}=W^{\varphi}_{tjki}C^{\varphi}_{itk}-\frac{1}{m-2}\left(R^{\varphi}_{it}C^{\varphi}_{itj}+\alpha R^{\varphi}_{jk}\varphi^a_k\varphi^a_{ii}-\frac{\alpha S^{\varphi}}{m-1}\varphi^a_{ii}\varphi^a_j\right),
			\end{equation*}
			hence, by plugging into the above we finally get
			\begin{equation}\label{div di N}
			\begin{aligned}
			\mbox{div}(N)_j=&W^{\varphi}_{tjki}C^{\varphi}_{itk}-\frac{1}{m-2}R^{\varphi}_{ik}C^{\varphi}_{ikj}+\alpha\left[\frac{S^{\varphi}}{(m-1)(m-2)}\varphi^a_{ii}\varphi^a_j-\frac{1}{m-2}R^{\varphi}_{kj}\varphi^a_k\varphi^a_{ii}\right]\\
			&+\alpha\left[\alpha \varphi^b_{ii}\varphi^b_k\varphi^a_j\varphi^a_k-\frac{1}{2}S^{\varphi}_k\varphi^a_k\varphi^a_j-R^{\varphi}_{ik}\varphi^a_{ij}\varphi^a_k-R^{\varphi}_{ik}\varphi^a_{ik}\varphi^a_j-\varphi^a_{kkii}-\varphi^a_{kki}\varphi^a_{ij}\right].
			\end{aligned}
			\end{equation}
		\end{itemize}
		Combining \eqref{div L}, \eqref{div di M} and \eqref{div di N}, after a computation and rearranging the terms, recalling the definition \eqref{def bi tension} of the bi-tension of $\varphi$, we conclude the validity of \eqref{div di phi Bach}.
	\end{proof}

	\subsection{Einstein fields equations in presence of a field}\label{section sol einst field eq con campo}
	
	\begin{dfn}\label{def harm einst mani}
		Let $(M,g)$ be a semi-Riemannian manifold of dimension $m\geq 3$, $\varphi:M\to N$ a smooth map with target a Riemannian manifold $(N,\eta)$ and $\alpha\in\mathbb{R}\setminus\{0\}$. We say that $(M,g)$ is {\em harmonic-Einstein (with respect to $\varphi$ and $\alpha$)} if
		\begin{equation}\label{harm einst}
			\begin{dcases}
			\mathring{\mbox{Ric}}^{\varphi}=0\\
			\tau(\varphi)=0.
			\end{dcases}
		\end{equation}
		We say that $(M,g)$ is {\em $\varphi$-Ricci flat (with respect to $\alpha$)} if
		\begin{equation}\label{phi ricci flat}
		\begin{dcases}
		\mbox{Ric}^{\varphi}=0\\
		\tau(\varphi)=0.
		\end{dcases}
		\end{equation}
	\end{dfn}
	\begin{rmk}\label{rmk phi curv of harm einst}
		Harmonic-Einstein manifolds have constant $\varphi$-scalar curvature, it follows from \eqref{div of phi Ricci}. Then the $\varphi$-Schouten tensor is parallel and, as a consequence, the $\varphi$-Cotton tensor vanishes. Furthermore, harmonic-Einstein manifolds are $\varphi$-Bach flat.
	\end{rmk}
	\begin{rmk}
		$\varphi$-Ricci flat manifolds (with respect to $\alpha$) are no more than harmonic-Einstein manifolds (with respect to $\varphi$ and $\alpha$) with vanishing $\varphi$-scalar curvature.
	\end{rmk}
	
	\begin{prop}\label{prop harm einst soddisfano einst equ}
		Let $(M,g)$ be a semi-Riemannian manifold of dimension $m\geq 3$, $\varphi:M\to N$ a harmonic map with target a Riemannian manifold $(N,\eta)$ and $\alpha\in\mathbb{R}\setminus\{0\}$. Then $(M,g)$ is harmonic-Einstein (with respect to $\varphi$ and $\alpha$) if and only if
		\begin{equation}\label{einst field equation}
			G+\Lambda g=\alpha T,
		\end{equation}
		where $G$ is the Einstein tensor of $(M,g)$, that is,
		\begin{equation*}
			G:=\mbox{Ric}-\frac{S}{2}g,
		\end{equation*}
		$T$ is the energy-stress tensor of the map $\varphi$ that is given by \eqref{stress energy tensor}. If this the case, the cosmological constant is given by
		\begin{equation}\label{cosm const}
			\Lambda:=\frac{m-2}{2m}S^{\varphi}\in\mathbb{R}.
		\end{equation}
	\end{prop}
	\begin{proof}
		Notice that, since $\varphi$ is harmonic, $(M,g)$ is harmonic-Einstein with respect to $\varphi$ and $\alpha$ if and only if the first equation of \eqref{harm einst} holds, that is,
		\begin{equation}\label{eq harm einst per mostare sol eq einst}
			\mbox{Ric}^{\varphi}=\frac{S^{\varphi}}{m}g.
		\end{equation}
		Assume \eqref{eq harm einst per mostare sol eq einst} holds. Using the definition of $\varphi$-Ricci tensor, the relation $S^{\varphi}=S-\alpha|d\varphi|^2$ and \eqref{eq harm einst per mostare sol eq einst} we obtain
		\begin{align*}
			G=&\mbox{Ric}-\frac{S}{2}g\\
			=&\mbox{Ric}^{\varphi}+\alpha\varphi^*\eta-\frac{S^{\varphi}}{2}g+\alpha\frac{|d\varphi|^2}{2}g\\
			=&\frac{S^{\varphi}}{m}g-\frac{S^{\varphi}}{2}g+\alpha\left(\varphi^*\eta-\frac{|d\varphi|^2}{2}g\right),
		\end{align*}
		that gives \eqref{einst field equation}, by setting $\Lambda$ as in \eqref{cosm const} (that is constant since it is a constant multiple of the $\varphi$-scalar curvature) and recalling the definition \eqref{stress energy tensor} of $T$. The converse implication follows analogously.
	\end{proof}
	\begin{rmk}\label{rmk sol eq Einstein}
		Four dimensional harmonic-Einstein Lorentzian manifolds $(M,g)$ with respect to a smooth map $\varphi:M\to N$, with target a Riemannian manifold $(N,\eta)$, and
		\begin{equation}\label{alpha per eq campo einstein}
			\alpha=\frac{8\pi \mathcal{G}}{c^4},
		\end{equation}
		where $\mathcal{G}$ is Newton's gravitational constant and $c$ is the speed of light in vacuum, are solutions of the Einstein field equations with cosmological constant $\Lambda$ given by \eqref{cosm const} and as field source the wave map (i.e., harmonic map with source a Lorentzian manifold) $\varphi$, see Section 6.5 of \cite{C-B}. In particular, $\varphi$-Ricci flat (with respect to $\alpha$ given by \eqref{alpha per eq campo einstein}) solves the Einstein field equations with the absence of cosmological constant, i.e., with $\Lambda=0$.
	\end{rmk}
	
	\section{Transformation laws of $\varphi$-curvature under a conformal change of metric}\label{section Transformation laws of curvature under a conformal change of metric}
	
	Let $g$ be a pseudo-Riemannian metric on $M$. Let $h\in\mathcal{C}^{\infty}(M)$ and denote
	\begin{equation}\label{conf change metric}
	\widetilde{g}:=e^{-2h}g.
	\end{equation}
	
	In this Section we denote with a tilde the tensors depending on the metric $\widetilde{g}$ and without subscripts the ones depending on $g$, for instance, for the Ricci tensors
	\begin{equation*}
	\widetilde{\mbox{Ric}}\equiv \mbox{Ric}_{\widetilde{g}}, \quad \mbox{Ric}\equiv \mbox{Ric}_g,
	\end{equation*}
	and we do the same for their components.
	
	The following Proposition is well known (compare to 1.159 of \cite{B}, that deals with the Riemannian case).
	\begin{prop}
		Let $\{\theta^i\}$, $\{e_i\}$, $\{\theta^i_j\}$ and $\{\Theta^i_j\}$ be, respectively, a local $g$-orthonormal coframe, the dual frame, the Levi-Civita connection forms and the curvature forms for the pseudo-Riemannian metric $g$ of $M$ on an open subset $\mathcal{U}$. Let $h\in\mathcal{C}^{\infty}(M)$ and set $\widetilde{g}$ as in \eqref{conf change metric}.
		\begin{itemize}
			\item A $\widetilde{g}$-orthonormal coframe $\{\widetilde{\theta}^i\}$ on $\mathcal{U}$ is given by
			\begin{equation}\label{def tilde theta}
			\widetilde{\theta}^i=e^{-h}\theta^i.
			\end{equation}
			\item The dual frame $\{\widetilde{e}_i\}$ of $\{\widetilde{\theta}^i\}$ is given by
			\begin{equation}\label{frame conf change metric}
			\widetilde{e}_i:=e^he_i,
			\end{equation}
			\item The Levi-Civita connection forms $\{\widetilde{\theta}^i_j\}$ for $\widetilde{g}$ on $\mathcal{U}$ are given by
			\begin{equation}\label{conn form change metric}
			\widetilde{\theta}^i_j=\theta^i_j-h_j\theta^i+h^i\theta_j.
			\end{equation}
			\item The components of the Riemann tensor of $(M,\widetilde{g})$ in the local coframe $\{\widetilde{\theta}^i\}$ are given by
			\begin{equation}\label{curv form conf change metric}
			\begin{aligned}
			e^{-2h}\widetilde{R}_{jikt}=&R_{jikt}+(h_{jk}\eta_{it}-h_{tj}\eta_{ik}+h_{it}\eta_{jk}-h_{ik}\eta_{tj})\\
			&+(h_jh_k\eta_{it}-h_jh_t\eta_{ik}+h_ih_t\eta_{jk}-h_ih_k\eta_{jt})-|\nabla h|^2(\eta_{jk}\eta_{it}-\eta_{jt}\eta_{ik}),
			\end{aligned}
			\end{equation}
			where $R_{jikt}$ are the components of the Riemann tensor of $(M,g)$ in the local coframe $\{\theta^i\}$.
			In global notation
			\begin{equation}\label{riem conf change metric}
			e^{2h}\widetilde{\mbox{Riem}}=\mbox{Riem}+\left(\mbox{Hess}(h)+dh\otimes dh-\frac{|\nabla h|^2}{2}g\right)\circleland g.
			\end{equation}
			\item The Riemannian volume element of $\widetilde{g}$ is given by
			\begin{equation}\label{conf change metric vol element}
			\widetilde{\mu}=e^{-mh}\mu.
			\end{equation}
		\end{itemize}
	\end{prop}
	\begin{proof}
		Since on $\mathcal{U}$ we have $g=\eta_{ij}\theta^i\otimes \theta^j$, \eqref{conf change metric} gives $\widetilde{g}=\eta_{ij}\widetilde{\theta}^i\otimes \widetilde{\theta}^j$, where $\{\widetilde{\theta}^i\}$ are given by \eqref{def tilde theta}. This shows that $\{\widetilde{\theta}^i\}$ is a local $\widetilde{g}$-orthonormal coframe. By setting $\{\widetilde{e}_i\}$ as in \eqref{frame conf change metric} it is immediate to check that $\{\widetilde{e}_i\}$ is the dual frame of $\{\widetilde{\theta}^i\}$.
		
		The first structure equations read
		\begin{equation*}
		d\widetilde{\theta}^i+\widetilde{\theta}^i_j\wedge\widetilde{\theta}^j=0.
		\end{equation*}
		Using \eqref{def tilde theta} the above gives
		\begin{equation*}
		e^{-h}(d\theta^i-dh\wedge \theta^i)+e^{-h}\widetilde{\theta}^i_j\wedge\theta^j=0,
		\end{equation*}
		that is,
		\begin{equation*}
		d\theta^i-dh\wedge \theta^i+\widetilde{\theta}^i_j\wedge\theta^j=0.
		\end{equation*}
		With the aid of the first structure equations for the coframe $\{\theta^i\}$ (see \eqref{first struct equation}) we get
		\begin{equation*}
		-\theta^i_j\wedge\theta^j-h_j\theta^j\wedge \theta^i+\widetilde{\theta}^i_j\wedge\theta^j=0,
		\end{equation*}
		that is,
		\begin{equation*}
		(\widetilde{\theta}^i_j-\theta^i_j+h_j\theta^i)\wedge \theta^j=0.
		\end{equation*}
		Assuming that the connection forms $\{\widetilde{\theta}^i_j\}$ are given by \eqref{conn form change metric} we obtain that the above equation is satisfied. Indeed, by plugging \eqref{conn form change metric} into the above we get $h^i\theta_j\wedge \theta^j=0$, that is satisfied. Moreover, using once again \eqref{conn form change metric} we deduce
		\begin{equation*}
		\eta_{ik}\widetilde{\theta}^k_j+\eta_{kj}\widetilde{\theta}^k_i=\eta_{ik}\theta^k_j+\eta_{kj}\theta^k_i,
		\end{equation*}
		hence the skew-symmetry $\widetilde{\theta}_{ij}+\widetilde{\theta}_{ji}=0$ follows immediately from the skew-symmetry $\theta_{ij}+\theta_{ji}=0$. Recalling that the connection forms are characterized by the skew symmetry and the validity of the first structure equations we obtain that $\{\widetilde{\theta}^i_j\}$ given by \eqref{conn form change metric} are the Levi-Civita connection forms for $\widetilde{g}$ on $\mathcal{U}$.
		
		The second structure equations for the metric $\widetilde{g}$ are given by
		\begin{equation*}
			\widetilde{\Theta}^j_i=d\widetilde{\theta}^j_i+\widetilde{\theta}_{ki}\land \widetilde{\theta}^{kj}.
		\end{equation*}
		Using \eqref{conn form change metric} we get
		\begin{align*}
			\widetilde{\Theta}^j_i=d(\theta^j_i-h_i\theta^j+h^j\theta_i)+(\theta_{ki}-h_i\theta_k+h_k\theta_i)\land (\theta^{kj}-h^j\theta^k+h^k\theta^j),
		\end{align*}
		that gives
		\begin{align*}
		\widetilde{\Theta}^j_i=&d\theta^j_i+\theta_{ki}\land \theta^{kj}+(dh^j-h_k\theta^{kj})\land\theta_i-(dh_i-h^k\theta_{ki})\land\theta^j\\
		&-h_i(d\theta^j-\theta^{kj}\land \theta_k)+h^j(d\theta_i-\theta_{ki}\land \theta^k)+h_ih^j\theta_k\land \theta^k\\
		&-h_ih^k\theta_k\land\theta^j-h_kh^j\theta_i\land \theta^k+h_kh^k\theta_i\land \theta^j.
		\end{align*}
		From the definition of covariant derivatives we get $h_{ij}\theta^j=dh_i-h_j\theta^j_i$ and by plugging it together with the first and the second structure equations for $g$ and $\theta_k\land \theta^k=0$ into the above we deduce
		\begin{align*}
		\widetilde{\Theta}^j_i=&\Theta^j_i+h^j_k\theta^k\land\theta_i-h_{ik}\theta^k\land\theta^j\\
		&+h_i(\theta^{jk}+\theta^{kj})\land \theta_k-h^j(\theta_{ik}+\theta_{ki})\land \theta^k\\
		&-h_ih^k\theta_k\land\theta^j-h_kh^j\theta_i\land \theta^k+|\nabla h|^2\theta_i\land \theta^j.
		\end{align*}
		Using the skew symmetry $\theta_{ij}+\theta_{ji}=0$ from the above we conclude
		\begin{align*}
		\eta_{js}\widetilde{\Theta}^j_i=\eta_{js}\Theta^j_i+(h_{sk}\eta_{it}-h_{ik}\eta_{st}-h_ih_k\eta_{st}+h_kh_s\eta_{it}+|\nabla h|^2\eta_{ik}\eta_{st})\theta^k\land\theta^t,
		\end{align*}
		that gives, recalling \eqref{comp of curv forms are components of Riem},
		\begin{align*}
		\frac{1}{2}\widetilde{R}_{jikt}\widetilde{\theta}^k\land \widetilde{\theta}^t=\left(\frac{1}{2}R_{jikt}+h_{jk}\eta_{it}-h_{ik}\eta_{jt}-h_ih_k\eta_{jt}+h_kh_j\eta_{it}+|\nabla h|^2\eta_{ik}\eta_{jt}\right)\theta^k\land\theta^t.
		\end{align*}
		Using \eqref{def tilde theta}, skew-symmetrizing the above we finally get
		\begin{align*}
			e^{-2h}\widetilde{R}_{jikt}=&R_{jikt}+h_{jk}\eta_{it}-h_{ik}\eta_{jt}-h_{jt}\eta_{ik}+h_{it}\eta_{jk}\\
			&-h_ih_k\eta_{jt}+h_kh_j\eta_{it}+h_ih_t\eta_{jk}-h_th_j\eta_{ik}+|\nabla h|^2(\eta_{ik}\eta_{jt}-\eta_{it}\eta_{jk})
		\end{align*}
		that is, \eqref{curv form conf change metric} holds. Since
		\begin{equation*}
		\widetilde{\mbox{Riem}}=\widetilde{R}_{jikt}\widetilde{\theta}^k\otimes\widetilde{\theta}^t\otimes\widetilde{\theta}^i\otimes\widetilde{\theta}^j=e^{-4h}\widetilde{R}_{jikt}\theta^k\otimes\theta^t\otimes\theta^i\otimes\theta^j,
		\end{equation*}
		the validity of \eqref{curv form conf change metric} implies \eqref{riem conf change metric}.
		
		The validity of \eqref{conf change metric vol element} follows easily from \eqref{riem vol elem con moving frame} and \eqref{def tilde theta}.
	\end{proof}
	
	Let $(N,\eta)$ be a Riemannian manifold, we denote by $\{E_a\}$, $\{\omega^a\}$ and $\{\omega^a_b\}$ the local orthonormal frame, coframe and the corresponding Levi-Civita connection forms on an open set $\mathcal{V}$ such that $\varphi^{-1}(\mathcal{V})\subseteq \mathcal{U}$. Clearly $d\varphi$ is independent on the choice of the metric on $M$, it means that
	\begin{equation}\label{coeff d phi conformal change metric}
	\widetilde{\varphi}^a_i=e^{h}\varphi^a_i,
	\end{equation}
	where we used \eqref{def tilde theta} and
	\begin{equation*}
	\varphi^a_i\theta
	^i\otimes E_a=d\varphi=\widetilde{\varphi}^a_i\widetilde{\theta}^i\otimes E_a.
	\end{equation*}
	As an immediate consequence we get
	\begin{equation}\label{conf change energy of phi}
	\widetilde{|d\varphi|}^2=e^{2h}|d\varphi|^2.
	\end{equation}
	By definition
	\begin{equation*}
	\nabla d\varphi=\varphi^a_{ij}\theta^j\otimes \theta^i\otimes E_a, \quad \varphi^a_{ij}\theta^j=d\varphi^a_i-\varphi^a_j\theta^j_i+\varphi^b_i\omega^a_b
	\end{equation*}
	and
	\begin{equation*}
	\widetilde{\nabla} d\varphi=\widetilde{\varphi}^a_{ij}\widetilde{\theta}^j\otimes \widetilde{\theta}^i\otimes E_a, \quad \widetilde{\varphi}^a_{ij}\widetilde{\theta}^j=d\widetilde{\varphi}^a_i-\widetilde{\varphi}^a_j\widetilde{\theta}^j_i+\widetilde{\varphi}^b_i\omega^a_b.
	\end{equation*}
	We denote by $\widetilde{\tau}(\varphi)$ the tension of the map
	\begin{equation*}
	\varphi:(M,\widetilde{g})\to (N,\eta),
	\end{equation*}
	in components
	\begin{equation*}
	\widetilde{\tau}(\varphi)^a=\eta^{ij}\widetilde{\varphi}^a_{ij}.
	\end{equation*}
	In the next Proposition we determine the transformation laws for the quantities of our interest related to the smooth map $\varphi$, under the conformal change of the metric \eqref{conf change metric}.
	\begin{prop}
		In the notations above, in a local orthonormal coframe,
		\begin{equation}\label{coeff nabla d phi conf change metric}
		\widetilde{\varphi}^a_{ij}=e^{2h}\left(\varphi^a_{ij}+\varphi^a_ih_j+\varphi^a_jh_i-\varphi^a_kh^k\eta_{ij}\right),
		\end{equation}
		in particular
		\begin{equation}\label{tension phi conf change metric}
		\widetilde{\tau}(\varphi)=e^{2h}[\tau(\varphi)-(m-2)d\varphi(\nabla h)].
		\end{equation}
		Moreover, in a local orthonormal coframe,
		\begin{equation}\label{conf chan third der of d phi}
		\widetilde{\tau}(\varphi)^a_k=e^{3h}\left[\tau(\varphi)^a_k-(m-2)\varphi^a_{ik}h^i-(m-2)\varphi^a_ih^i_k+2\tau^g(\varphi)^ah_k-2(m-2)\varphi^a_ih^ih_k\right].
		\end{equation}
	\end{prop}
	\begin{proof}
		The validity of \eqref{coeff nabla d phi conf change metric} follows easily using \eqref{def tilde theta}, the definition of $\widetilde{\varphi}^a_{ij}$, \eqref{coeff d phi conformal change metric}, \eqref{conn form change metric} and the definition of $\varphi^a_{ij}$ as follows:
		\begin{align*}
		\widetilde{\varphi}^a_{ij}e^{-h}\theta^j=&\widetilde{\varphi}^a_{ij}\widetilde{\theta}^j\\
		=&d\widetilde{\varphi}^a_i-\widetilde{\varphi}^a_j\widetilde{\theta}^j_i+\widetilde{\varphi}^b_i\omega^a_b\\
		=&d(e^{h}\varphi^a_i)-e^{h}\varphi^a_j\left(\theta^j_i-h_i\theta^j+h^j\theta_i\right)+e^{h}\varphi^b_i\omega^a_b\\
		=&e^{h}(d\varphi^a_i-\varphi^a_j\theta^j_i+\varphi^b_i\omega^a_b)+e^{h}\varphi^a_i dh+e^{h}\varphi^a_j(h_i\theta^j-h^j\theta_i)\\
		=&e^{h}\left[\varphi^a_{ij}+\varphi^a_ih_j+\varphi^a_jh_i-\varphi^a_kh^k\eta_{ij}\right]\theta^j.
		\end{align*}
		Applying $\eta^{ij}$ to \eqref{coeff nabla d phi conf change metric} we immediately get \eqref{tension phi conf change metric}. For convenience we denote
		\begin{equation*}
		T^a_{ij}=\varphi^a_{ij}+\varphi^a_ih_j+\varphi^a_jh_i-\varphi^a_th^t\eta_{ij}.
		\end{equation*}
		Using the definition of covariant derivative, with the aid of \eqref{conn form change metric} we get
		\begin{align*}
		\widetilde{\varphi}^a_{ijk}\widetilde{\theta}^k=&d\widetilde{\varphi}^a_{ij}-\widetilde{\varphi}^a_{kj}\widetilde{\theta}^k_i-\widetilde{\varphi}^a_{ik}\widetilde{\theta}^k_j+\widetilde{\varphi}^b_{ij}\omega^a_b\\
		=&d(e^{2h}T^a_{ij})-e^{2h}T^a_{kj}(\theta^k_i+h^k\theta_i-h_i\theta^k)-e^{2h}T^a_{ik}(\theta^k_j+h^k\theta_j-h_j\theta^k)+e^{2h}T^b_{ij}\omega^a_b.
		\end{align*}
		Thus, by plugging \eqref{def tilde theta} into the above and doing some calculations,
		\begin{align*}
		e^{-3h}\widetilde{\varphi}^a_{ijk}\theta^k=&2T^a_{ij}h_k\theta^k+T^a_{ij,k}\theta^k-T^a_{kj}(h^k\theta_i-h_i\theta^k)-T^a_{ik}(h^k\theta_j-h_j\theta^k)\\
		=&(T^a_{ij,k}+2T^a_{ij}h_k+T^a_{kj}h_i-T^a_{tj}h^t\eta_{ik}+T^a_{ik}h_j-T^a_{it}h^t\eta_{jk})\theta^k,
		\end{align*}
		that gives,
		\begin{align*}
		e^{-3h}\widetilde{\varphi}^a_{ijk}=&T^a_{ij,k}+2T^a_{ij}h_k+T^a_{kj}h_i-T^a_{tj}h^t\eta_{ik}+T^a_{ik}h_j-T^a_{it}h^t\eta_{jk}.
		\end{align*}
		Applying $\eta_{ij}$ and using the relations
		\begin{equation*}
		\eta^{ij}T^a_{ij}=\tau^g(\varphi)^a-(m-2)\varphi^a_ih^i, \quad \eta^{ij}T^a_{ij,k}=\tau(\varphi)^a_{k}-(m-2)\varphi^a_{ik}h^i-(m-2)\varphi^a_ih^i_k
		\end{equation*}
		(the first follows immediately from the definition of $T^a_{ij}$ while the second is obtained taking covariant derivative of the first), we get from the above
		\begin{align*}
		e^{-3h}\eta^{ij}\widetilde{\varphi}^a_{ijk}=&\eta^{ij}T^a_{ij,k}+2\eta^{ij}T^a_{ij}f_k+T^a_{ki}f^i-T^a_{tk}f^t+T^a_{ik}f^i-T^a_{kt}f^t\\
		=&\eta^{ij}T^a_{ij,k}+2\eta^{ij}T^a_{ij}h_k\\
		=&\tau(\varphi)^a_{k}-(m-2)\varphi^a_{ik}h^i-(m-2)\varphi^a_ih^i_k+2h_k(\tau^g(\varphi)^a-(m-2)\varphi^a_ih^i),
		\end{align*}
		that is \eqref{conf chan third der of d phi}.
	\end{proof}

	\begin{prop}
		The components of the $\varphi$-Ricci tensor of $(M,\widetilde{g})$ in the coframe $\{\widetilde{\theta}^i\}$ are given by
		\begin{equation}\label{conformal change metric ricci in local coordinates}
		e^{-2h}\widetilde{R}^{\varphi}_{ij}=R^{\varphi}_{jk}+(m-2)h_{ij}+(m-2)h_ih_j+[\Delta h-(m-2)|\nabla h|^2]\eta_{ij},
		\end{equation}
		where $R^{\varphi}_{jk}$ are the components of the $\varphi$-Ricci tensor of $(M,g)$ in the coframe $\{\theta^i\}$. In global notation
		\begin{equation}\label{transf law ricci phi}
		\widetilde{\mbox{Ric}}^{\varphi}=\mbox{Ric}^{\varphi}+(m-2)\mbox{Hess}(h)+(m-2)dh\otimes dh+[\Delta h-(m-2)|\nabla h|^2]g.
		\end{equation}
		Moreover the $\varphi$-scalar curvature of $(M,\widetilde{g})$ is given by
		\begin{equation}\label{conformal change metric scalar curvature}
		e^{-2h}\widetilde{S}^{\varphi}=S^{\varphi}+(m-1)[2\Delta h-(m-2)|\nabla h|^2],
		\end{equation}
		where $S^{\varphi}$ is the $\varphi$-scalar curvature of $(M,g)$.
	\end{prop}
	\begin{proof}
		First of all we prove
		\begin{equation}\label{comp ricci conf change metric}
		e^{-2h}\widetilde{R}_{jk}=R_{jk}+(m-2)h_{jk}+(m-2)h_jh_k+[\Delta h-(m-2)|\nabla h|^2]\eta_{jk}.
		\end{equation}
		From \eqref{ricci tensor trace riem}, using \eqref{curv form conf change metric} we get
		\begin{equation*}
		e^{-2h}\widetilde{R}_{jk}=R_{jk}+(mh_{jk}-h_{kj}+\Delta h\eta_{jk}-h_{jk})+(mh_jh_k-h_jh_k+|\nabla h|^2\eta_{jk}-h_jh_k)-|\nabla h|^2(m\eta_{jk}-\eta_{jk}),
		\end{equation*}
		or equivalently, \eqref{comp ricci conf change metric}.
		
		The validity of \eqref{conformal change metric ricci in local coordinates} follows immediately from \eqref{comp ricci conf change metric} and \eqref{coeff d phi conformal change metric}. Using \eqref{transf law ricci phi} and \eqref{def tilde theta} we deduce that \eqref{conformal change metric ricci in local coordinates} holds.
		
		Applying $\eta^{ij}$ to \eqref{conformal change metric ricci in local coordinates} we immediately get \eqref{conformal change metric scalar curvature}.
	\end{proof}

	From now on we assume $m\geq 3$ and we set
	\begin{equation}\label{def f}
	h:=\frac{1}{m-2}f,
	\end{equation}
	so that \eqref{conf change metric} reads
	\begin{equation*}
		\widetilde{g}=e^{-\frac{2f}{m-2}}g.
	\end{equation*}
	Then \eqref{transf law ricci phi}, \eqref{conformal change metric scalar curvature}, \eqref{tension phi conf change metric} and \eqref{conf chan third der of d phi} reads, respectively,
	\begin{equation}\label{transf law ricci phi con f}
	\widetilde{\mbox{Ric}}^{\varphi}=\mbox{Ric}^{\varphi}+\mbox{Hess}(f)+\frac{1}{m-2}(df\otimes df+\Delta_ff g),
	\end{equation}
	\begin{equation}\label{transf law phi scalar con f}
	e^{-\frac{2}{m-2}f}\widetilde{S}^{\varphi}=S^{\varphi}+\frac{m-1}{m-2}[2\Delta f-|\nabla f|^2],
	\end{equation}
	\begin{equation}\label{tension phi conf change metric con f}
	\widetilde{\tau}(\varphi)=e^{\frac{2}{m-2}f}[\tau(\varphi)-d\varphi(\nabla f)]
	\end{equation}
	and
	\begin{equation}\label{conf chan third der of d phi con f}
	\eta^{ij}\widetilde{\varphi}^a_{ijk}=e^{\frac{3}{m-2}f}\left[\eta^{ij}\varphi^a_{ijk}-\varphi^a_{ik}f^i-\varphi^a_if^i_k+\frac{2}{m-2}(\tau^g(\varphi)^a-\varphi^a_if^i)f_k\right].
	\end{equation}

	Then it is easy to obtain, using the definition of $\varphi$-Schouten and \eqref{transf law ricci phi con f} and \eqref{transf law phi scalar con f},
	\begin{equation}\label{conformal change metric schouten tensor}
	\widetilde{A}^{\varphi}=A^{\varphi}+\mbox{Hess}(f)+\frac{1}{m-2}\left(df\otimes df-\frac{|\nabla f|^2}{2}g\right),
	\end{equation}
	that locally reads
	\begin{equation}\label{conformal change metric schouten tensor local coordinates}
	e^{-\frac{2}{m-2}f}\widetilde{A}^{\varphi}_{ij}=A^{\varphi}_{ij}+f_{ij}+\frac{1}{m-2}\left(f_if_j-\frac{|\nabla f|^2}{2}\eta_{ij}\right).
	\end{equation}
	Moreover, from the definition \eqref{def of phi Weyl} of $\varphi$-Weyl, using \eqref{riem conf change metric} and \eqref{conformal change metric schouten tensor}, we immediately get
	\begin{equation}\label{conf invariance phi weyl}
		e^{\frac{2f}{m-2}}\widetilde{W}^{\varphi}=W^{\varphi},
	\end{equation}
	that is, the $(1,3)$ version of the $\varphi$-Weyl is conformal invariant.
	
	In the next Proposition we deal with the transformation laws for the $\varphi$-Cotton tensor.
	\begin{prop}
		In the notations above, the components of the $\varphi$-Cotton tensor of $(M,\widetilde{g})$ with respect to the coframe $\{\widetilde{\theta}^i\}$ are given by
		\begin{equation}\label{conformal change metric cotton tensor}
		e^{-\frac{3}{m-2}f}\widetilde{C}^{\varphi}_{ijk}=C^{\varphi}_{ijk}+W^{\varphi}_{tijk}f^t,
		\end{equation}
		where $C^{\varphi}_{ijk}$, $W^{\varphi}_{tijk}$ and $f_t$ are, respectively, the components of $C^{\varphi}$, $W^{\varphi}$ and $df$ in the coframe $\{\theta^i\}$. 
	\end{prop}
	\begin{proof}
		For simplicity of notation we set
		\begin{equation}\label{def di T in terms of A}
			T_{ij}:=A^{\varphi}_{ij}+f_{ij}+\frac{1}{m-2}\left(f_if_j-\frac{|\nabla f|^2}{2}\eta_{ij}\right),
		\end{equation}
		so that \eqref{conformal change metric schouten tensor local coordinates} reads
		\begin{equation}\label{phi schout with respect to T}
			\widetilde{A}^{\varphi}_{ij}=e^{\frac{2}{m-2}f}T_{ij}.
		\end{equation}
		We claim the validity of
		\begin{equation}\label{der cov A tilde in funz di T}
		e^{-\frac{3}{m-2}f}\widetilde{A}^{\varphi}_{ij,k}=\frac{2}{m-2}T_{ij}f_k+T_{ij,k}+\frac{1}{m-2}(T_{kj}f_i-T_{tj}f^t\eta_{ki}+T_{ik}f_j-T_{it}f^t\eta_{jk}).
		\end{equation}
		To prove the claim we use the definition of covariant derivative, \eqref{phi schout with respect to T} and \eqref{conn form change metric} (with $h$ given by \eqref{def f}), obtaining
		\begin{align*}
		\widetilde{A}^{\varphi}_{ij,k}\widetilde{\theta}^k=&d\widetilde{A}^{\varphi}_{ij}-\widetilde{A}^{\varphi}_{kj}\widetilde{\theta}^k_i-\widetilde{A}^{\varphi}_{ik}\widetilde{\theta}^k_j\\
		=&d(e^{\frac{2}{m-2}f}T_{ij})-e^{\frac{2}{m-2}f}T_{kj}\left(\theta^k_i-\frac{f_i}{m-2}\theta^k+\frac{f^k}{m-2}\theta_i\right)\\
		&-e^{\frac{2}{m-2}f}T_{ik}\left(\theta^k_j-\frac{f_j}{m-2}\theta^k+\frac{f^k}{m-2}\theta_j\right),
		\end{align*}
		that is,
		\begin{align*}
		e^{-\frac{2}{m-2}f}\widetilde{A}^{\varphi}_{ij,k}\widetilde{\theta}^k=&\frac{2}{m-2}T_{ij}df+(dT_{ij}-T_{kj}\theta^k_i-T_{ik}\theta^k_j)+\frac{1}{m-2}[T_{kj}(f_i\theta^k-f^k\theta_i)+T_{ik}(f_j\theta^k-f^k\theta_j)].
		\end{align*}
		Using \eqref{def tilde theta} and the definition of $T_{ij,k}$ the above yields
		\begin{align*}
		e^{-\frac{3}{m-2}f}\widetilde{A}^{\varphi}_{ij,k}\theta^k=&\left[\frac{2}{m-2}T_{ij}f_k+T_{ij,k}+\frac{1}{m-2}(T_{kj}f_i-T_{tj}f^t\eta_{ki}+T_{ik}f_j-T_{it}f^t\eta_{jk})\right]\theta^k,
		\end{align*}
		hence \eqref{der cov A tilde in funz di T} holds.
		
		Now, using the definition of the $\varphi$-Cotton tensor, \eqref{der cov A tilde in funz di T} twice and the symmetry of $T$ we get
		\begin{align*}
		e^{-\frac{3}{m-2}f}\widetilde{C}^{\varphi}_{ijk}=&e^{-\frac{3}{m-2}f}(\widetilde{A}^{\varphi}_{ij,k}-\widetilde{A}^{\varphi}_{ik,j})\\
		=&T_{ij,k}-T_{ik,j}+\frac{2}{m-2}(T_{ij}f_k-T_{ik}f_j)+\frac{1}{m-2}[T_{ik}f_j-T_{ij}f_k+(T_{tk}\eta_{ji}-T_{tj}\eta_{ki})f^t],
		\end{align*}
		that is,
		\begin{equation}\label{conf change cotton in terms of T}
		e^{-\frac{3}{m-2}f}\widetilde{C}^{\varphi}_{ijk}=T_{ij,k}-T_{ik,j}+\frac{1}{m-2}(T_{ij}\eta_{kt}-T_{ik}\eta_{jt}+T_{tk}\eta_{ji}-T_{tj}\eta_{ki})f^t.
		\end{equation}
		To express the right hand side of the above in terms of $C^{\varphi}$ we first observe that, from the definition \eqref{def di T in terms of A} of $T$,
		\begin{equation*}
		T_{ij,k}=A^{\varphi}_{ij,k}+f_{ijk}+\frac{1}{m-2}(f_{ik}f_j+f_if_{jk}-f_tf_{tk}\delta_{ij}),
		\end{equation*}
		so that, using the commutation rule (see \eqref{general commutation rule tensor field})
		\begin{equation*}
		f_{ijk}=f_{ikj}+R_{tijk}f^t,
		\end{equation*}
		and the definition of $C^{\varphi}_{ijk}$ we get
		\begin{align*}
		T_{ij,k}-T_{ik,j}=&A^{\varphi}_{ij,k}+f_{ijk}+\frac{1}{m-2}(f_{ik}f_j+f_if_{jk}-f^tf_{tk}\eta_{ij})\\
		&-\left[A^{\varphi}_{ik,j}+f_{ikj}+\frac{1}{m-2}(f_{ij}f_k+f_if_{kj}-f^tf_{tj}\eta_{ik})\right]\\
		=&C^{\varphi}_{ijk}+R^t_{ijk}f_t+\frac{1}{m-2}[f_{ik}f_j-f_{ij}f_k+f^t(f_{tj}\eta_{ik}-f_{tk}\eta_{ij})].
		\end{align*}
		Moreover an easy computation using \eqref{def di T in terms of A} shows that
		\begin{align*}
		(T_{ij}\eta_{kt}-T_{ik}\eta_{jt}+T_{tk}\eta_{ji}-T_{tj}\eta_{ki})f^t=&A^{\varphi}_{ij}f_k-A^{\varphi}_{ik}f_j+A^{\varphi}_{tk}f^t\eta_{ji}-A^{\varphi}_{tj}f^t\eta_{ki}\\
		&+f_{ij}f_k-f_{ik}f_j+f_{tk}f^t\eta_{ji}-f_{tj}f^t\eta_{ki}.
		\end{align*}
		Plugging the two relations above into \eqref{conf change cotton in terms of T} we finally conclude
		\begin{align*}
		e^{-\frac{3}{m-2}f}\widetilde{C}^{\varphi}_{ijk}=&C^{\varphi}_{ijk}+R_{tijk}f^t+\frac{1}{m-2}[f_{ik}f_j-f_{ij}f_k+f^t(f_{tj}\eta_{ik}-f_{tk}\eta_{ij})]\\
		&+\frac{1}{m-2}(A^{\varphi}_{ij}f_k-A^{\varphi}_{ik}f_j+A^{\varphi}_{tk}f^t\eta_{ji}-A^{\varphi}_{tj}f^t\eta_{ki})\\
		&+\frac{1}{m-2}(f_{ij}f_k-f_{ik}f_j+f_{tk}f^t\eta_{ji}-f_{tj}f^t\eta_{ki})\\
		=&C^{\varphi}_{ijk}+R^t_{ijk}f_t-\frac{1}{m-2}(A^{\varphi}_{tj}\eta_{ki}-A^{\varphi}_{tk}\eta_{ij}+A^{\varphi}_{ik}\eta_{tj}-A^{\varphi}_{ij}\eta_{tk})f^t.
		\end{align*}
		Thus follows \eqref{conformal change metric cotton tensor}, in view of the decomposition \eqref{def of phi Weyl}.
	\end{proof}
	
	The last transformation law we are going to illustrate is the one for the $\varphi$-Bach tensor $B^{\varphi}$ and is the hardest to obtain. In order to determine it we first need to evaluate the transformation law for the tensor
	\begin{equation}\label{def of V}
	V_{ij}:=\eta^{kt}C^{\varphi}_{ijk,t}-\alpha[(R^{\varphi})^k_j\varphi^a_k+\tau(\varphi)^a_{j}]\varphi^a_i,
	\end{equation}
	that is the content of
	\begin{lemma}
		In the above notations,
		\begin{equation}\label{conformal change metric divergence phi cotton}
		\begin{aligned}
		e^{-\frac{4}{m-2}f}\widetilde{V}_{ij}=&V_{ij}+f^{tk}W^{\varphi}_{tijk}-\frac{m-5}{m-2}f^tf^kW^{\varphi}_{tijk}+\frac{m-4}{m-2}(C^{\varphi}_{jki}+C^{\varphi}_{ikj})f^k+\alpha \varphi^a_{ij}\varphi^a_kf^k\\
		&+\frac{\alpha}{m-2}[(\varphi^a_kf^k-\tau(\varphi)^a)(\varphi^a_if_j+\varphi^a_jf_i)-\tau(\varphi)^a\varphi^a_kf^k\eta_{ij}-\Delta_ff\varphi^a_i\varphi^a_j].
		\end{aligned}
		\end{equation}
	\end{lemma}
	\begin{proof}
		We set
		\begin{equation*}
		T_{ijk}=C^{\varphi}_{ijk}+f^tW^{\varphi}_{tijk},
		\end{equation*}
		so that, from \eqref{conformal change metric cotton tensor},
		\begin{equation*}
			\widetilde{C}^{\varphi}_{ijk}=e^{\frac{3}{m-2}f}T_{ijk}.
		\end{equation*}		
		Proceeding exactly as in the proof of the Proposition above we get
		\begin{align*}
		e^{-\frac{4}{m-2}f}\widetilde{C}^{\varphi}_{ijk,s}=&T_{ijk,s}+\frac{3}{m-2}T_{ijk}f_s+\frac{1}{m-2}(f_iT_{sjk}+f_jT_{isk}+f_kT_{ijs})\\
		&-\frac{f^t}{m-2}(T_{tjk}\eta_{is}+T_{itk}\eta_{js}+T_{ijt}\eta_{ks}).
		\end{align*}
		Applying $\eta^{sk}$ to the above an easy calculation shows that
		\begin{equation}\label{in proof conformal change divergence of C}
		e^{-\frac{4}{m-2}f}\eta^{sk}\widetilde{C}^{\varphi}_{ijk,s}=\eta^{sk}T_{ijk,s}-\frac{m-4}{m-2}T_{ijk}f^k+\frac{1}{m-2}(\eta^{sk}T_{kjs}f_i+\eta^{sk}T_{iks}f_j)-\frac{f^k}{m-2}(T_{kji}+T_{ikj}).
		\end{equation}
		Using the definition of $T_{ijk}$ and by plugging \eqref{div di phi Weyl} into the above we infer
		\begin{equation}\label{in proof first item divergence cotton under conformal change metric}
		\begin{aligned}
		\eta^{sk}T_{ijk,s}=&\eta^{sk}C^{\varphi}_{ijk,s}+f^{tk}W^{\varphi}_{tijk}+\frac{m-3}{m-2}C^{\varphi}_{jki}f^k+\alpha(\varphi^a_{ji}\varphi^a_kf^k-\varphi^a_{jk}f^k\varphi^a_i)\\
		&+\frac{\alpha}{m-2}\tau(\varphi)^a(\varphi^a_if_j-\varphi^a_kf^k\eta_{ij}).
		\end{aligned}
		\end{equation}
		Clearly
		\begin{equation}\label{in proof second item divergence cotton under conformal change metric}
		T_{ijk}f^k=C^{\varphi}_{ijk}f^k+f^tf^kW^{\varphi}_{tijk}.
		\end{equation}
		The traces of $T$ are given by, using \eqref{traccia phi cotton}, \eqref{traccia phi weyl} and the symmetries of tensors involved,
		\begin{equation*}
		\eta^{ks}T_{kjs}=\eta^{ks}C^{\varphi}_{kjs}+f^t\eta^{ks}W^{\varphi}_{tkjs}=-\alpha\tau(\varphi)^a\varphi^a_j+\alpha\varphi^a_t\varphi^a_jf^t=\alpha(\varphi^a_kf^k-\tau(\varphi)^a)\varphi^a_j
		\end{equation*}
		and
		\begin{equation*}
		\eta^{ks}T_{iks}=0,
		\end{equation*}
		then we easily get
		\begin{equation}\label{in proof third item divergence cotton under conformal change metric}
		\eta^{ks}T_{kjs}f_i+\eta^{ks}T_{iks}f_j=\alpha(\varphi^a_kf^k-\tau(\varphi)^a)\varphi^a_jf_i.
		\end{equation}
		Using once again the definition of $T$, the skew symmetry in the first two indices of $W^{\varphi}$ and the identity \eqref{bianchi id phi cotton} for $C^{\varphi}$ we evaluate
		\begin{align*}
		f^kT_{kji}+f^kT_{ikj}=&f^k(C^{\varphi}_{kji}+W^{\varphi}_{tkji}f^t)+f^k(C^{\varphi}_{ikj}+W^{\varphi}_{tikj}f^t)\\
		=&f^k(C^{\varphi}_{kji}+C^{\varphi}_{ikj})+f^tf^kW^{\varphi}_{tikj}\\
		=&-f^kC^{\varphi}_{jik}+f^tf^kW^{\varphi}_{tikj}\\
		=&f^kC^{\varphi}_{jki}+f^tf^kW^{\varphi}_{tikj}.
		\end{align*}
		Plugging the above together with \eqref{in proof first item divergence cotton under conformal change metric}, \eqref{in proof second item divergence cotton under conformal change metric} and \eqref{in proof third item divergence cotton under conformal change metric} in \eqref{in proof conformal change divergence of C} we finally get
		\begin{align*}
		e^{-\frac{4}{m-2}f}\eta^{ks}\widetilde{C}^{\varphi}_{ijk,s}=&\eta^{ks}C^{\varphi}_{ijk,s}+f^{tk}W^{\varphi}_{tijk}-\frac{m-5}{m-2}f^tf^kW^{\varphi}_{tijk}\\
		&+\frac{m-4}{m-2}(C^{\varphi}_{jki}+C^{\varphi}_{ikj})f^k+\alpha(\varphi^a_{ij}\varphi^a_kf^k-\varphi^a_{jk}f^k\varphi^a_i)\\
		&+\frac{\alpha}{m-2}[\tau(\varphi)^a(\varphi^a_if_j-\varphi^a_jf_i)+\varphi^a_kf^k\varphi^a_jf_i-\tau(\varphi)^a\varphi^a_kf^k\eta_{ij}].
		\end{align*}
		To conclude the proof notice that, with the aid of \eqref{transf law ricci phi con f} and \eqref{tension phi conf change metric con f}, we get
		\begin{align*}
		e^{-\frac{4}{m-2}f}((\widetilde{R}^{\varphi})^k_j\widetilde{\varphi}^a_k\widetilde{\varphi}^a_i+\widetilde{\tau}(\varphi)^a_j\widetilde{\varphi}^a_i)=&(R^{\varphi})^k_{j}\varphi^a_k\varphi^a_i+\tau(\varphi)^a_{j}\varphi^a_i-\varphi^a_{kj}f^k\varphi^a_i\\
		&+\frac{1}{m-2}(\Delta_ff\varphi^a_i\varphi^a_j-\varphi^a_kf^k\varphi^a_if_j+2\tau(\varphi)^a\varphi^a_if_j).
		\end{align*}
		Inserting the relations obtained so far into the definition \eqref{def of V} we obtain the validity of \eqref{conformal change metric divergence phi cotton}.
	\end{proof}
	Now we are finally ready to prove
	\begin{thm}
		In the above notations, the components of the $\varphi$-Bach tensor of $(M,\widetilde{g})$ in the local coframe $\{\widetilde{\theta}^i\}$ are given by
		\begin{equation}\label{conformal change metric phi bach tensor}
		e^{-\frac{4}{m-2}f}(m-2)\widetilde{B}^{\varphi}_{ij}=(m-2)B^{\varphi}_{ij}-\frac{m-4}{m-2}f^k(C^{\varphi}_{ijk}+f^tW^{\varphi}_{tijk}-C^{\varphi}_{jki}).
		\end{equation}
	\end{thm}
	\begin{proof}
		From the definition of $\varphi$-Bach \eqref{def phi bach} and \eqref{def of V}
		\begin{equation}\label{form per phi bach usando V}
		(m-2)B^{\varphi}_{ij}=V_{ij}+W^{\varphi}_{tikj}(R^{\varphi})^{tk}+\alpha\tau(\varphi)^a\left(\varphi^a_{ij}-\frac{1}{m-2}\tau(\varphi)^a\eta_{ij}\right)
		\end{equation}
		Using \eqref{conformal change metric ricci in local coordinates}, the conformal invariance of $\varphi$-Weyl \eqref{conf invariance phi weyl} and \eqref{traccia phi weyl} we infer
		\begin{align*}
		e^{-\frac{4}{m-2}f}\widetilde{W}^{\varphi}_{tikj}(\widetilde{R}^{\varphi})^{tk}=&W^{\varphi}_{tikj}\left((R^{\varphi})^{tk}+f^{tk}+\frac{f^tf^k}{m-2}+\frac{\Delta_ff}{m-2}\eta^{tk}\right)\\
		=&W^{\varphi}_{tikj}(R^{\varphi})^{tk}+W^{\varphi}_{tikj}f^{tk}+\frac{1}{m-2}W^{\varphi}_{tikj}f^tf^k+\alpha\frac{\Delta_ff}{m-2}\varphi^a_i\varphi^a_j\\
		=&W^{\varphi}_{tikj}(R^{\varphi})^{tk}-W^{\varphi}_{tijk}f^{tk}-\frac{1}{m-2}W^{\varphi}_{tijk}f^tf^k+\alpha\frac{\Delta_ff}{m-2}\varphi^a_i\varphi^a_j.
		\end{align*}
		Using \eqref{coeff nabla d phi conf change metric} three times a computation yields
		\begin{align*}
		e^{-\frac{4}{m-2}f}\widetilde{\tau}(\varphi)^a\left(\widetilde{\varphi}^a_{ij}-\frac{1}{m-2}\widetilde{\tau}(\varphi)^a\eta_{ij}\right)=&\tau(\varphi)^a\left(\varphi^a_{ij}-\frac{1}{m-2}\tau(\varphi)^a\eta_{ij}\right)\\
		&+\frac{1}{m-2}\tau(\varphi)^a\varphi^a_kf^k\eta_{ij}-\varphi^a_kf^k\varphi^a_{ij}\\
		&+\frac{1}{m-2}(\tau(\varphi)^a-\varphi^a_kf^k)(\varphi^a_if_j+\varphi^a_jf_i).
		\end{align*}
		Combining these two relations with \eqref{conformal change metric divergence phi cotton}, from \eqref{form per phi bach usando V} we deduce the validity of \eqref{conformal change metric phi bach tensor}.
	\end{proof}
	As an immediate consequence of the transformation law for $\varphi$-Bach we generalize the conformal invariance in the four dimensional case of the Bach tensor in the following.
	\begin{cor}\label{cor phi bach invariante conforme per m uguale a 4}
		Let $(M,g)$ be a four dimensional pseudo-Riemannian manifold, $\varphi:M\to N$ a smooth map, where $(N,\eta)$ is a target Riemannian manifold, and $\alpha\in\mathbb{R}\setminus \{0\}$. Then the $\varphi$-Bach tensor is a conformal invariant tensor.
	\end{cor}

	\begin{dfn}
		Let $(M,g)$ be a semi-Riemannian manifold of dimension $m\geq 3$, $\varphi:M\to N$ a smooth map, where $(N,\eta)$ is a target Riemannian manifold, and $\alpha\in\mathbb{R}\setminus \{0\}$. We denote
		\begin{equation*}
		[g]:=\{v^2g: v\in\mathcal{C}^{\infty}(M), v>0 \mbox{ on } M\}
		\end{equation*}
		and we say that $(M,g)$ is {\em conformally harmonic-Einstein} (with respect to $\varphi$ and $\alpha$) if there exists $\widetilde{g}\in [g]$ such that $(M,\widetilde{g})$ is harmonic-Einstein (with respect to $\varphi$ and $\alpha$). We say that $(M,g)$ is {\em conformally $\varphi$-Ricci flat} (with respect to $\alpha$) if there exists $\widetilde{g}\in [g]$ such that $(M,\widetilde{g})$ is $\varphi$-Ricci flat (with respect to $\alpha$).
	\end{dfn}
	\begin{rmk}
		Four dimensional conformally harmonic-Einstein manifolds are $\varphi$-Bach flat, this is due to the conformal invariance of $\varphi$-Bach for four dimensional semi-Riemannian manifolds and the trivial fact that harmonic-Einstein manifolds are $\varphi$-Bach flat, see \hyperref[rmk phi curv of harm einst]{Remark \ref*{rmk phi curv of harm einst}}.
	\end{rmk}
	\begin{rmk}
		Let $(M,g)$ be a semi-Riemannian manifold of dimension $m\geq 3$, $\varphi:M\to N$ a smooth map, where $(N,\eta)$ is a target Riemannian manifold, and $\alpha\in\mathbb{R}\setminus \{0\}$. Using \eqref{transf law ricci phi con f}, \eqref{transf law phi scalar con f} and \eqref{tension phi conf change metric con f} it is immediate to realize that $(M,g)$ is conformally harmonic-Einstein (with respect to $\varphi$ and $\alpha$) if and only if
		\begin{equation*}
		\begin{dcases}
		\left(\mbox{Ric}^{\varphi}+\mbox{Hess}(f)+\frac{1}{m-2}df\otimes df\right)^{\circ}=0\\
		\tau(\varphi)=d\varphi(\nabla f),
		\end{dcases}
		\end{equation*}
		for some $f\in\mathcal{C}^{\infty}(M)$. Indeed the second equation of the above is clearly equivalent to $\widetilde{\tau}(\varphi)=0$, using \eqref{tension phi conf change metric con f}, while the first equation of can be rewritten as
		\begin{equation*}
		\mbox{Ric}^{\varphi}+\mbox{Hess}(f)+\frac{1}{m-2}df\otimes df=\frac{1}{m}\left(S^{\varphi}+\Delta f+\frac{1}{m-2}|\nabla f|^2\right)g,
		\end{equation*}
		or also as
		\begin{equation*}
		\mbox{Ric}^{\varphi}+\mbox{Hess}(f)+\frac{1}{m-2}(df\otimes df+\Delta_ff g)=\frac{1}{m}\left(S^{\varphi}+\frac{2(m-1)}{m-2}\Delta f-\frac{m-1}{m-2}|\nabla f|^2\right)g.
		\end{equation*}
		Now the equivalence between the above and
		\begin{equation*}
		\widetilde{\mbox{Ric}}^{\varphi}=\frac{\widetilde{S}^{\varphi}}{m}\widetilde{g}
		\end{equation*}
		is evident, using \eqref{transf law ricci phi con f}, \eqref{transf law phi scalar con f}.
	\end{rmk}
	
	\section{Harmonic-Einstein warped products}\label{section Warped products}
	
	Let $(M,g)$ and $(F,g_F)$ be two semi-Riemannian manifolds of dimension $m$ and $d$ respectively. Let $u\in\mathcal{C}^{\infty}(M)$, $u>0$ on $M$.
	\begin{dfn}
		We denote by $\bar{M}=M\times F$ the product manifold, by
		\begin{equation*}
		\overline{g}:=\pi_M^*g+(u\circ \pi_M)^2\pi_F^* g_F,
		\end{equation*}
		where $\pi_M:\bar{M}\to M$ and $\pi_F:\bar{M}\to F$ are the canonical projections, and by
		\begin{equation*}
		M\times_u F:=(\bar{M},\overline{g})
		\end{equation*}
		the {\em semi-Riemannian warped product} with {\em base} $(M,g)$, {\em fibre} $(F,g_F)$ and {\em warping function} $u$.
	\end{dfn}
	We are going to identify $T\bar{M}$ with $TM\oplus TF$. Via the identification
	\begin{equation*}
	\overline{g}\equiv g+u^2g_F.
	\end{equation*}
	We use the following indexes conventions
	\begin{equation*}
	1\leq i,j,\ldots \leq m, \quad 1\leq \alpha,\beta,\ldots \leq d, \quad 1\leq A,B,\ldots \leq m+d.
	\end{equation*}
	
	Let $\{e_i\}$, $\{\theta^i\}$, $\{\theta^i_j\}$, $\{\Theta^i_j\}$ be, respectively, a local $g$-orthonormal frame, the dual coframe, the relative connection and curvature forms on an open subset $\mathcal{U}$ of $M$ and let $\{\varepsilon_\alpha\}$, $\{\psi^{\alpha}\}$, $\{\psi^{\alpha}_{\beta}\}$, $\{\Psi^{\alpha}_{\beta}\}$ be the same quantities on an open subset $\mathcal{W}$ of $F$ with respect to $g_F$.
	
	In the next well known Proposition we determine the local $\bar{g}$-orthonormal frame, the dual coframe, the relative connection and curvature forms on $\overline{\mathcal{U}}:=\mathcal{U}\times \mathcal{W}$ induced by the choices above, that we denote by $\{\overline{e}_A\}$, $\{\overline{\theta}^A\}$, $\{\overline{\theta}^A_B\}$, $\{\overline{\Theta}^A_B\}$, respectively. We have
	\begin{equation*}
	g=\eta_{ij}\theta^i\otimes \theta^j, \quad g_F={}^F\eta_{\alpha \beta}\psi^{\alpha}\otimes \psi^{\beta},
	\end{equation*}
	where $\eta_{ij}$ and ${}^F\eta_{\alpha\beta}$ and defined as in \eqref{def eta i j}, according to the signature of $g$ and $g_F$, respectively.
	\begin{prop}
		In the notations above
		\begin{itemize}
			\item The local $g$-orthonormal frame $\{\overline{e}_A\}$ is given by
			\begin{equation}\label{orth frame warp prod}
			\overline{e}_i=\pi_M^*(e_i)\equiv e_i, \quad \overline{e}_{m+\alpha}=\frac{1}{u\circ \pi_M}\pi_F^*(\varepsilon_{\alpha})\equiv\frac{1}{u}\varepsilon_{\alpha}.
			\end{equation}
			\item The corresponding dual coframe $\{\overline{\theta}^A\}$ is given by
			\begin{equation}\label{orth coframe warp prod}
			\overline{\theta}^i=\pi_M^*(\theta^i)\equiv\theta^i, \quad \overline{\theta}^{m+\alpha}=u\circ \pi_M\cdot\pi_F^*\psi^{\alpha}\equiv u\psi^{\alpha}.
			\end{equation}
			Then
			\begin{equation*}
			\bar{g}=\bar{\eta}_{AB}\overline{\theta}^A\otimes \overline{\theta}^B,
			\end{equation*}
			where
			\begin{equation*}
			\bar{\eta}_{ij}=\eta_{ij}, \quad \bar{\eta}_{i\, m+\beta}=0=\bar{\eta}_{m+\alpha \,j}, \quad \bar{\eta}_{m+\alpha\, m+\beta}={}^F\eta_{\alpha \beta}.
			\end{equation*}
			\item The Levi-Civita connection forms $\{\overline{\theta}^A_B\}$ are given by
			\begin{equation}\label{conn form warp prod}
			\overline{\theta}^i_j=\theta^i_j, \quad \overline{\theta}^{m+\alpha}_{m+\beta}=\psi^{\alpha}_{\beta}, \quad \overline{\theta}^{m+\alpha}_i=u_i\psi^{\alpha}, \quad \overline{\theta}^i_{m+\alpha}=-u^i\psi_{\alpha}.
			\end{equation}
			\item The curvature forms $\{\overline{\Theta}^A_B\}$ are given by
			\begin{equation}\label{curv form warp prod}
			\overline{\Theta}^j_i=\Theta^j_i, \quad \overline{\Theta}^{m+\beta}_{m+\alpha}=\Psi^{\beta}_{\alpha}+|\nabla u|^2\psi_{\alpha}\wedge \psi^{\beta}, \quad \overline{\Theta}^{m+\alpha}_i=u_{ij}\theta^j\wedge\psi^{\alpha}, \quad \overline{\Theta}^i_{m+\alpha}=-u^{ij}\theta_j\land \psi_{\alpha}.
			\end{equation}
			Then the non-vanishing components of $\overline{\mbox{Riem}}$ are determined by
			\begin{equation}\label{comp riem warp prod}
			\begin{aligned}
			&\bar{R}_{ijkt}=R_{ijkt}, \quad \bar{R}_{i\, m+\alpha\, j\, m+\beta}=-\frac{u_{ij}}{u}{}^F\eta_{\alpha\beta},\\
			&\bar{R}_{m+\alpha\, m+\beta\, m+\gamma\, m+\delta}=\frac{1}{u^2}{}^FR_{\alpha\beta\gamma\delta}-\frac{|\nabla u|^2}{u^2}({}^F\eta_{\alpha \gamma}{}^F\eta_{\beta\delta}-{}^F\eta_{\alpha\delta}{}^F\eta_{\beta\gamma}),
			\end{aligned}
			\end{equation}
			where $R_{ijkt}$ and ${}^FR_{\alpha\beta\gamma\delta}$ are the components of the Riemann tensors of $(M,g)$ and $(F,g_F)$, respectively.
		\end{itemize}
	\end{prop}
	\begin{proof}
		It is clear that $\{\overline{e}_A\}$ defined as in \eqref{orth frame warp prod} is a local orthonormal frame, indeed
		\begin{equation*}
		\bar{g}(\overline{e}_i ,\overline{e}_j)=g(e_i,e_j)=\eta_{ij}, \quad \bar{g}(\overline{e}_i,\overline{e}_{m+\alpha})=0
		\end{equation*}
		and
		\begin{equation*}
		\bar{g}(\overline{e}_{m+\alpha},\overline{e}_{m+\beta})= u^2g_F\left(\frac{\varepsilon_{\alpha}}{u},\frac{\varepsilon_{\beta}}{u}\right)=g_F(\varepsilon_{\alpha},\varepsilon_{\beta})={}^F\eta_{\alpha \beta}.
		\end{equation*}
		The relations \eqref{orth coframe warp prod} follows immediately from \eqref{orth frame warp prod}.
		
		To show the validity of \eqref{conn form warp prod} recall that the first structure equation on $M\times_u F$ are given by
		\begin{equation*}
		d\overline{\theta}^A=-\overline{\theta}^A_B\wedge \overline{\theta}^B.
		\end{equation*}
		For $A=i$ we obtain, using \eqref{orth coframe warp prod},
		\begin{align*}
		d\overline{\theta}^i=&-\overline{\theta}^i_B\wedge \overline{\theta}^B\\
		=&-\overline{\theta}^i_j\wedge \overline{\theta}^j-\overline{\theta}^i_{m+\alpha}\wedge \overline{\theta}^{m+\alpha}\\
		=&-\overline{\theta}^i_j\wedge \theta^j-u\overline{\theta}^i_{m+\alpha}\wedge \psi^{\alpha}
		\end{align*}
		and since, from the first structure equation on $M$,
		\begin{align*}
		d\overline{\theta}^i=&d\theta^i=-\theta^i_j\wedge \theta^j,
		\end{align*}
		we conclude from the above
		\begin{equation}\label{conn form warp prod for a =i}
		(\overline{\theta}^i_j-\theta^i_j)\wedge \theta^j+u\overline{\theta}^i_{m+\alpha}\wedge \psi^{\alpha}=0.
		\end{equation}
		For $A=m+\alpha$ we obtain, using \eqref{orth coframe warp prod},
		\begin{align*}
		d\overline{\theta}^{m+\alpha}=&-\overline{\theta}^{m+\alpha}_B\wedge \overline{\theta}^B\\
		=&-\overline{\theta}^{m+\alpha}_i\wedge \overline{\theta}^i-\overline{\theta}^{m+\alpha}_{m+\beta}\wedge \overline{\theta}^{m+\beta}\\
		=&-\overline{\theta}^{m+\alpha}_i\wedge \theta^i-u\overline{\theta}^{m+\alpha}_{m+\beta}\wedge \psi^{\beta}
		\end{align*}
		and since, from the first structure equation for $F$,
		\begin{align*}
		d\overline{\theta}^{m+\alpha}=&d(u\psi^{\alpha})\\
		=&du\wedge \psi^{\alpha}+ud\psi^{\alpha}\\
		=&u_i\theta^i\wedge \psi^{\alpha}-u\psi^{\alpha}_{\beta}\wedge \psi^{\beta}\\
		=&-u_i\psi^{\alpha}\wedge \theta^i-u\psi^{\alpha}_{\beta}\wedge \psi^{\beta}
		\end{align*}
		we conclude from the above
		\begin{equation}\label{conn form warp prod for a =m + alpha}
		u(\overline{\theta}^{m+\alpha}_{m+\beta}-\psi^{\alpha}_{\beta})\wedge \psi^{\beta}+(\overline{\theta}^{m+\alpha}_i-u_i\psi^{\alpha})\wedge\theta^i=0.
		\end{equation}
		It is immediate to verify that $\{\overline{\theta}^A_B\}$ given by \eqref{conn form warp prod} satisfies \eqref{conn form warp prod for a =i}, \eqref{conn form warp prod for a =m + alpha} and
		\begin{equation*}
		\bar{\eta}_{CB}\overline{\theta}^C_A+\bar{\eta}_{AC}\overline{\theta}^C_B=0,
		\end{equation*}
		hence they are the connection forms associated to the coframe $\{\overline{\theta}^A\}$.
		
		The second structure equation reads as
		\begin{equation*}
			\overline{\Theta}^B_A=d\overline{\theta}^B_A+\overline{\theta}_{CA}\land \overline{\theta}^{CB}.
		\end{equation*}
		For $A=i$ and $B=j$, using \eqref{conn form warp prod}, $\psi_{\alpha}\land\psi^{\alpha}=0$ and the second structure equations of $(M,g)$ we obtain
		\begin{align*}
			\overline{\Theta}^j_i=&d\overline{\theta}^j_i+\overline{\theta}_{ki}\land \overline{\theta}^{kj}+\overline{\theta}_{m+\alpha i}\land \overline{\theta}^{m+\alpha j}\\
			=&d\theta^j_i+\theta_{ki}\land\theta^{kj}+u_iu^j\psi_{\alpha}\land \psi^{\alpha}\\
			=&\Theta^j_i.
		\end{align*}		
		For $A=m+\alpha$ and $B=m+\beta$, using \eqref{conn form warp prod} and the second structure equations of $(F,g_F)$ we get
		\begin{align*}
			\overline{\Theta}^{m+\beta}_{m+\alpha}=&d\overline{\theta}^{m+\beta}_{m+\alpha}+\overline{\theta}_{k m+\alpha}\land \overline{\theta}^{k m+\beta}+\overline{\theta}_{m+\gamma\,  m+\alpha}\land \overline{\theta}^{m+\gamma\, m+\beta}\\
			=&d\psi^{\beta}_{\alpha}+u_ku^k\psi_{\alpha}\land \psi^{\beta}+\psi_{\gamma \alpha}\land \psi^{\gamma \beta}\\
			=&\Psi^{\beta}_{\alpha}+|\nabla u|^2\psi_{\alpha}\land \psi^{\beta}
		\end{align*}
		For $A=i$ and $B=m+\alpha$, using \eqref{conn form warp prod}, the definition of $u_{ij}$ and the first structure equations of $(F,g_F)$ we have
		\begin{align*}
			\overline{\Theta}^{m+\alpha}_i=&d\overline{\theta}^{m+\alpha}_i+\overline{\theta}_{k i}\land \overline{\theta}^{k m+\alpha}+\overline{\theta}_{m+\gamma i}\land \overline{\theta}^{m+\gamma\, m+\alpha}\\
			=&d(u_i\psi^{\alpha})-u^k\theta_{ki}\land \psi^{\alpha}+u_i\psi_{\gamma}\land \psi^{\gamma \alpha}\\
			=&(du_i-u^k\theta_{ki})\land \psi^{\alpha}+u_i(d\psi^{\alpha}-\psi^{\gamma \alpha}\land \psi_{\gamma})\\
			=&u_{ik}\theta^k\land \psi^{\alpha}.
		\end{align*}
		Since
		\begin{equation*}
			\overline{\eta}_{CB}\overline{\Theta}^C_A+\overline{\eta}_{CA}\overline{\Theta}^C_B=0
		\end{equation*}
		we deduce that
		\begin{equation*}
			\overline{\Theta}^i_{m+\alpha}=-\eta^{ij}{}^F\eta_{\alpha\beta}\overline{\Theta}^{m+\beta}_j=-\eta^{ij}{}^F\eta_{\alpha\beta}u_{jk}\theta^k\land \psi^{\beta}=-u^{jk}\theta_k\land \psi_{\alpha}.
		\end{equation*}
			
		By definition of $\overline{\mbox{Riem}}$, the Riemann tensor of $M\times_u F$, we have
		\begin{equation*}
			\overline{\Theta}^A_B=\frac{1}{2}\bar{R}^A_{BCD}\overline{\theta}^C\land \overline{\theta}^D.
		\end{equation*}
		For $A=m+\alpha$ and $B=i$, with the aid of \eqref{orth coframe warp prod},
		\begin{equation*}
		\overline{\Theta}^{m+\alpha}_i=\frac{1}{2}\bar{R}^{m+\alpha}_{iCD}\overline{\theta}^C\land \overline{\theta}^D=\frac{1}{2}\left(\bar{R}^{m+\alpha}_{ikt}\theta^k\land\theta^t+u\bar{R}^{m+\alpha}_{ik\, m+\delta}\theta^k\land \psi^{\delta}+u\bar{R}^{m+\alpha}_{i\, m+\delta\, k}\psi^{\delta}\land \theta^k+u^2\bar{R}^{m+\alpha}_{i \, m+\gamma\, m+\delta}\psi^{\gamma}\land \psi^{\delta}\right).
		\end{equation*}
		Using \eqref{curv form warp prod} and the symmetries of $\overline{\mbox{Riem}}$ the above gives
		\begin{equation*}
		u_{ik}\theta^k\land \psi^{\alpha}=\overline{\Theta}^{m+\alpha}_i=\frac{1}{2}\bar{R}^{m+\alpha}_{ikt}\theta^k\land\theta^t+u\bar{R}^{m+\alpha}_{ik\, m+\delta}\theta^k\land \psi^{\delta}+\frac{1}{2}u^2\bar{R}^{m+\alpha}_{i \, m+\gamma\, m+\delta}\psi^{\gamma}\land \psi^{\delta},
		\end{equation*}
		hence we deduce
		\begin{equation}\label{first rel Riemann warp prod}
		\bar{R}_{m+\alpha \, ijk}=0, \quad \bar{R}_{m+\alpha\, i \, m+\beta \, m+\gamma}=0, \quad \bar{R}_{m+\alpha\, i j\, m+\beta}=\frac{u_{ij}}{u}{}^F\eta_{\alpha\beta}.
		\end{equation}
		For $A=i$ and $B=j$, using \eqref{orth coframe warp prod},
		\begin{equation*}
		\overline{\Theta}^i_j=\frac{1}{2}\bar{R}^i_{jCD}\overline{\theta}^C\land \overline{\theta}^D=\frac{1}{2}\left(\bar{R}^i_{jkt}\theta^k\land\theta^t+u\bar{R}^i_{j k m+\alpha}\theta^k\land \psi^{\alpha}+u\bar{R}^i_{j \, m+\alpha \, k}\psi^{\alpha}\land \theta^k+u^2\bar{R}^i_{j\, m+\alpha\, m+\beta}\psi^{\alpha}\land \psi^{\beta}\right).
		\end{equation*}
		Using the \eqref{curv form warp prod}, the symmetries of $\overline{\mbox{Riem}}$ and \eqref{first rel Riemann warp prod} the above gives
		\begin{equation*}
		\frac{1}{2}R^i_{jkt}\theta^k\land \theta^t=\Theta^i_j=\overline{\Theta}^i_j=\frac{1}{2}\bar{R}^i_{jkt}\theta^k\land\theta^t+\frac{1}{2}u^2\bar{R}^i_{j\, m+\alpha\, m+\beta}\psi^{\alpha}\land \psi^{\beta},
		\end{equation*}
		so that
		\begin{equation}\label{second rel Riemann warp prod}
		\bar{R}_{ijkt}=R_{ijkt}, \quad \bar{R}_{i j \, m+\alpha \, m+\beta}=0.
		\end{equation}
		For $A=m+\alpha$ and $B=m+\beta$,
		\begin{equation*}
		\overline{\Theta}^{m+\alpha}_{m+\beta}=\frac{1}{2}\left(\bar{R}^{m+\alpha}_{m+\beta \,ij}\theta^i\land \theta^j+u\bar{R}^{m+\alpha}_{m+\beta \, i\, m+\gamma}\theta^i\land \psi^{\gamma}+u\bar{R}^{m+\alpha}_{m+\beta\, m+\gamma\, i}\psi^{\gamma}\land \theta^i+u^2\bar{R}^{m+\alpha}_{m+\beta \, m+\gamma\, m+\delta}\psi^{\gamma}\land \psi^{\delta}\right).
		\end{equation*}
		Using the \eqref{curv form warp prod}, the symmetries of $\overline{\mbox{Riem}}$, \eqref{first rel Riemann warp prod} and \eqref{second rel Riemann warp prod} from the above we infer
		\begin{equation*}
		\frac{1}{2}{}^FR^{\alpha}_{\beta \gamma \delta}\psi^{\gamma}\land \psi^{\delta}+|\nabla u|^2\psi_{\beta}\land \psi^{\alpha}=\Psi^{\alpha}_{\beta}+|\nabla u|^2\psi_{\beta}\land \psi^{\alpha}=\overline{\Theta}^{m+\alpha}_{m+\beta}=\frac{1}{2}u^2\bar{R}^{m+\alpha}_{m+\beta \, m+\gamma\, m+\delta}\psi^{\gamma}\land \psi^{\delta}.
		\end{equation*}
		Skew-symmetrizing the above we obtain
		\begin{equation*}
		u^2\bar{R}_{m+\alpha\, m+\beta\, m+\gamma\, m+\delta}={}^FR_{\alpha\beta\gamma\delta}+|\nabla u|^2({}^F\eta_{\alpha\delta}{}^F\eta_{\beta\gamma}-{}^F\eta_{\alpha \gamma}{}^F\eta_{\beta\delta}),
		\end{equation*}
		that is,
		\begin{equation*}
		\bar{R}_{m+\alpha\,m+\beta\, m+\gamma\, m+\delta}=\frac{1}{u^2}{}^FR_{\alpha\beta\gamma\delta}-\frac{|\nabla u|^2}{u^2}({}^F\eta_{\alpha \gamma}{}^F\eta_{\beta\delta}-{}^F\eta_{\alpha\delta}{}^F\eta_{\beta\gamma}).
		\end{equation*}
		
		We are finally able to conclude the validity of \eqref{comp riem warp prod}.
	\end{proof}
	
	Since $u>0$ on $M$ there exists $f\in\mathcal{C}^{\infty}(M)$ such that
	\begin{equation}\label{f in funct of u warp prod ricci}
	u=e^{-\frac{f}{d}}.
	\end{equation}
	As an immediate consequence of the above Proposition we have
	\begin{cor}\label{cor Ricci for warped prod}
		In the notations above, the non-vanishing components of $\overline{\mbox{Ric}}$, the Ricci tensor of $(\bar{M},\overline{g})$, are given by
		\begin{equation}\label{ricci tensor warped prod in terms of u}
		\bar{R}_{ij}=R_{ij}-d\frac{u_{ij}}{u}, \quad \bar{R}_{m+\alpha \, m+\beta}=-\left(\frac{\Delta u}{u}+(d-1)\frac{|\nabla u|^2}{u^2}\right){}^F\eta_{\alpha \beta}+\frac{1}{u^2}{}^FR_{\alpha\beta},
		\end{equation}
		where $R_{ij}$ and ${}^FR_{\alpha\beta}$ are the components of the Ricci tensors of $(M,g)$ and $(F,g_F)$, respectively. Moreover
		\begin{equation*}
		\bar{S}=S+\frac{{}^FS}{u^2}-d\left[2\frac{\Delta u}{u}+(d-1)\frac{|\nabla u|^2}{u^2}\right].
		\end{equation*}
		Equivalently, in terms of $f$, where $f$ is defined by \eqref{f in funct of u warp prod ricci},
		\begin{equation}\label{ricci tensor warped prod}
		\bar{R}_{ij}=R_{ij}+f_{ij}-\frac{1}{d}f_if_j, \quad \bar{R}_{m+\alpha\, m+\beta}=\frac{\Delta_ff}{d}{}^F\eta_{\alpha\beta}+e^{\frac{2f}{d}}{}^FR_{\alpha \beta}
		\end{equation}
		and
		\begin{equation*}
		\bar{S}=S+e^{\frac{2f}{d}}{}^FS+2\Delta f-\frac{d+1}{d}|\nabla f|^2.
		\end{equation*}
	\end{cor}

	Let $\varphi:M\to N$ be a smooth map with source the base manifold and target a Riemannian manifold $(N,\eta)$ and denote
	\begin{equation}\label{dfn of Phi}
	\bar{\varphi}:=\varphi\circ \pi_M:\bar{M}\to N.
	\end{equation}
	We use the indexes convention
	\begin{equation*}
	1\leq a,b,\ldots \leq n,
	\end{equation*}
	where $n$ is the dimension of $N$. Let $\{E_a\}$, $\{\omega^a\}$, $\{\omega^a_b\}$, $\{\Omega^a_b\}$ be, respectively, a $\eta$-orthonormal frame, $\eta$-orthonormal coframe, connection forms and curvatures form on a open subset $\mathcal{V}$ of $N$ such that $\varphi^{-1}(\mathcal{V})\subseteq \overline{\mathcal{U}}$. We denote
	\begin{equation*}
		d\bar{\varphi}=\bar{\varphi}^a_A\overline{\theta}^A\otimes E_a,
	\end{equation*}
	so that
	\begin{equation*}
		\bar{\varphi}^*\eta=\bar{\varphi}^a_A\bar{\varphi}^a_B\overline{\theta}^A\otimes \overline{\theta}^B.
	\end{equation*}
	
	\begin{prop}\label{prop tension of Phi in function of tension of phi}
		In the assumptions and the notations above
		\begin{itemize}
			\item The components of $d\bar{\varphi}$ are given by
			\begin{equation}\label{comp of d bar Phi}
			\bar{\varphi}^a_i=\varphi^a_i, \quad \bar{\varphi}^a_{m+\alpha}=0,
			\end{equation}
			as a consequence
			\begin{equation}\label{density of energy of Phi compared to phi}
			\overline{|d\bar{\varphi}|}^2=|d\varphi|^2.
			\end{equation}
			\item The tension $\bar{\tau}(\bar{\varphi})$ of $\bar{\varphi}:M\times_u F\to (N,\eta)$ is given by
			\begin{equation}\label{tension field of Phi in funct of tension field of phi in terms of u}
			\bar{\tau}(\bar{\varphi})=\tau(\varphi)+d\frac{d\varphi(\nabla u)}{u},
			\end{equation}
			in terms of $f$ given by \eqref{f in funct of u warp prod ricci},
			\begin{equation}\label{tension field of Phi in funct of tension field of phi}
			\bar{\tau}(\bar{\varphi})=\tau(\varphi)-d\varphi(\nabla f).
			\end{equation}
		\end{itemize}
	\end{prop}
	\begin{proof}
		Using \eqref{orth coframe warp prod},
		\begin{equation*}
		d\bar{\varphi}=\bar{\varphi}^a_A\overline{\theta}^A\otimes E_a=\bar{\varphi}^a_i\theta^i\otimes E_a+u\bar{\varphi}^a_{m+\alpha}\psi^{\alpha}\otimes E_a.
		\end{equation*}
		From \eqref{dfn of Phi} we deduce $d\bar{\varphi}=\pi_M^*d\varphi\equiv d\varphi$, that is,
		\begin{equation*}
		d\bar{\varphi}=\pi_M^*d\varphi\equiv d\varphi=\varphi^a_i\theta^i\otimes E_a,
		\end{equation*}
		and thus \eqref{comp of d bar Phi} follows by comparison with the above. In particular
		\begin{equation*}
		\bar{\eta}^{AB}\bar{\varphi}^a_A\bar{\varphi}^a_B=\eta^{ij}\varphi^a_i\varphi^a_j,
		\end{equation*}
		hence \eqref{density of energy of Phi compared to phi} holds. Moreover
		\begin{equation*}
		\bar{\varphi}^a_{AB}\overline{\theta}^B=d\bar{\varphi}^a_A-\bar{\varphi}^a_B\overline{\theta}^B_A+\bar{\varphi}^b_A\omega^a_b,
		\end{equation*}
		that is, using \eqref{conn form warp prod},
		\begin{equation*}
		\bar{\varphi}^a_{Aj}\theta^j+u\bar{\varphi}^a_{A\, m+\alpha}\psi^{\alpha}=d\bar{\varphi}^a_A-\bar{\varphi}^a_j\overline{\theta}^j_A-\bar{\varphi}^a_{m+\alpha}\overline{\theta}^{m+\alpha}_A+\bar{\varphi}^b_A\omega^a_b,
		\end{equation*}
		For $A=i$ we obtain, using \eqref{comp of d bar Phi} and \eqref{conn form warp prod},
		\begin{equation*}
		\bar{\varphi}^a_{ij}\theta^j+u\bar{\varphi}^a_{i\, m+\alpha}\psi^{\alpha}=d\varphi^a_i-\varphi^a_j\theta^j_i+\varphi^b_i\omega^a_b=\varphi^a_{ij}\theta^j,
		\end{equation*}
		hence
		\begin{equation}\label{sec cov der of Phi first comp}
		\bar{\varphi}^a_{ij}=\varphi^a_{ij}, \quad \bar{\varphi}^a_{i\, m+\alpha}=0.
		\end{equation}
		For $A=m+\beta$ we obtain using \eqref{comp of d bar Phi} and \eqref{conn form warp prod},
		\begin{equation*}
		\bar{\varphi}^a_{m+\beta\, j}\theta^j+u\bar{\varphi}^a_{m+\beta\, m+\alpha}\psi^{\alpha}=u^j\varphi^a_j\psi_{\beta},
		\end{equation*}
		hence,
		\begin{equation}\label{sec cov der of Phi second comp}
		\bar{\varphi}^a_{m+\alpha\, j}=0, \quad \bar{\varphi}^a_{m+\alpha \, m+\beta}=\varphi^a_j\frac{u^j}{u}{}^F\eta_{\alpha \beta}.
		\end{equation}
		Then, using \eqref{sec cov der of Phi first comp} and \eqref{sec cov der of Phi second comp}, we infer
		\begin{equation*}
		\bar{\tau}(\bar{\varphi})^a=\eta^{AB}\bar{\varphi}^a_{AB}=\eta^{ij}\bar{\varphi}^a_{ij}+{}^F\eta^{\alpha\beta}\bar{\varphi}^a_{m+\alpha\, m+\beta}=\tau(\varphi)^a-d\varphi^a_j\frac{u^j}{u},
		\end{equation*}
		that is, \eqref{tension field of Phi in funct of tension field of phi in terms of u}. Finally \eqref{tension field of Phi in funct of tension field of phi}, where $f\in\mathcal{C}^{\infty}(M)$ given by \eqref{f in funct of u warp prod ricci}, follows from \eqref{tension field of Phi in funct of tension field of phi in terms of u}.
	\end{proof}

	Combining \hyperref[cor Ricci for warped prod]{Corollary \ref*{cor Ricci for warped prod}} with \hyperref[prop tension of Phi in function of tension of phi]{Proposition \ref*{prop tension of Phi in function of tension of phi} } we immediately get
	
	\begin{cor}\label{cor Phi Ricci for warped prod}
		In the notations above, the non-vanishing components of $\overline{\mbox{Ric}}^{\bar{\varphi}}$, the $\bar{\varphi}$-Ricci tensor of $M\times_uF$, are given by
		\begin{equation}\label{phi ricci tensor warped prod in terms of u}
		\bar{R}_{ij}^{\bar{\varphi}}=R^{\varphi}_{ij}-d\frac{u_{ij}}{u}, \quad \bar{R}^{\bar{\varphi}}_{m+\alpha \, m+\beta}=-\left(\frac{\Delta u}{u}+(d-1)\frac{|\nabla u|^2}{u^2}\right){}^F\eta_{\alpha \beta}+\frac{1}{u^2}{}^FR_{\alpha\beta},
		\end{equation}
		where $R^{\varphi}_{ij}$ and ${}^FR_{\alpha\beta}$ are the components of the $\varphi$-Ricci tensors of $(M,g)$ and of the Ricci tensor of $(F,g_F)$, respectively. Moreover
		\begin{equation*}
		\bar{S}^{\bar{\varphi}}=S^{\varphi}+\frac{{}^FS}{u^2}-d\left[2\frac{\Delta u}{u}+(d-1)\frac{|\nabla u|^2}{u^2}\right].
		\end{equation*}
		Equivalently, in terms of $f$, where $f$ is defined by \eqref{f in funct of u warp prod ricci},
		\begin{equation}\label{phi ricci tensor warped prod}
		\bar{R}^{\bar{\varphi}}_{ij}=R^{\varphi}_{ij}+f_{ij}-\frac{1}{d}f_if_j, \quad \bar{R}^{\bar{\varphi}}_{m+\alpha\, m+\beta}=\frac{\Delta_ff}{d}{}^F\eta_{\alpha\beta}+e^{\frac{2f}{d}}{}^FR_{\alpha \beta}
		\end{equation}
		and
		\begin{equation*}
		\bar{S}^{\bar{\varphi}}=S^{\varphi}+e^{\frac{2f}{d}}{}^FS+2\Delta f-\frac{d+1}{d}|\nabla f|^2.
		\end{equation*}
	\end{cor}
	
	\begin{thm}\label{thm harm einst warp prod}
		Let $(M,g)$ and $(F,g_F)$ be pseudo-Riemannian manifolds of dimension $m$ and $d$ respectively, with $m+d\geq 3$, and $(N,\eta)$ a Riemannian manifold. Let $f\in\mathcal{C}^{\infty}(M)$ and $\varphi:M\to N$ smooth. Set $u$ as in \eqref{f in funct of u warp prod ricci} and $\bar{\varphi}:\bar{M}\to N$ as in \eqref{dfn of Phi}. Then the following are equivalent
		\begin{itemize}
			\item $M\times_u F$ is harmonic-Einstein with respect to $\alpha\in\mathbb{R}\setminus\{0\}$ and $\bar{\varphi}$ and with $\bar{\varphi}$-scalar curvature $\lambda\in\mathbb{R}$.
			\item $(M,g)$ satisfies, for some $\alpha\in\mathbb{R}\setminus\{0\}$ and $\lambda\in\mathbb{R}$,
			\begin{equation}\label{grad einst type struct on M warp prod harm einst}
			\begin{dcases}
			\mbox{Ric}^{\varphi}+\mbox{Hess}(f)-\frac{1}{d}df\otimes df=\frac{\lambda}{m+d} g\\
			\tau(\varphi)=d\varphi(\nabla f),
			\end{dcases}
			\end{equation}
			$(F,g_F)$ is Einstein with scalar curvature $\Lambda\in\mathbb{R}$ and the following equation holds
			\begin{equation}\label{Lambda in warp prod second factor einst consnt}
			\Delta_ff-\frac{d}{m+d}\lambda+\Lambda e^{\frac{2f}{d}}=0.
			\end{equation}
		\end{itemize}
	\end{thm}
	\begin{proof}
		The warped product $M\times_u F$ is harmonic-Einstein with $\bar{\varphi}$-scalar curvature $\lambda$ if and only if
		\begin{equation}\label{warp prod harm einst}
		\begin{dcases}
		\overline{\mbox{Ric}}^{\bar{\varphi}}=\frac{\lambda}{m+d} \overline{g}\\
		\bar{\tau}(\bar{\varphi})=0.
		\end{dcases}
		\end{equation}
		The first equation above gives, via \eqref{phi ricci tensor warped prod},
		\begin{equation}\label{grad einst type structure with mu equal 1 over d in warp prod harm einst}
		R_{ij}^{\varphi}+f_{ij}-\frac{1}{d}f_if_j=\frac{\lambda}{m+d}\delta_{ij},
		\end{equation}
		and
		\begin{equation}\label{F einst man in warp prod second factor}
		{}^FR_{\alpha\beta}=\left(\frac{\lambda}{m+d}-\frac{\Delta_ff}{d}\right)e^{-\frac{2f}{d}}{}^F\eta_{\alpha\beta},
		\end{equation}
		while the second, using \eqref{tension field of Phi in funct of tension field of phi}, implies $\tau(\varphi)=d\varphi(\nabla f)$. Hence \eqref{grad einst type struct on M warp prod harm einst} holds. It is possible to prove that, since $M$ is connected and even thought $g$ is a pseudo-Riemannian metric, the Hamilton-type identity holds for the system \eqref{grad einst type struct on M warp prod harm einst} holds, that is, \eqref{Lambda in warp prod second factor einst consnt} holds. From \eqref{F einst man in warp prod second factor} we infer
		\begin{equation*}
			{}^FS=\left(\frac{d}{m+d}\lambda-\Delta_ff\right)e^{-\frac{2f}{d}}=\Lambda.
		\end{equation*}
		The converse implication is trivial.
	\end{proof}
	The following is immediate from the above.
	\begin{cor}\label{cor warp prod harm einst con fibra R}
		Let $(M,g)$ be a semi-Riemannian manifold of dimension $m\geq 2$ and $(N,\eta)$ a Riemannian manifold. Let $f\in\mathcal{C}^{\infty}(M)$ and $\varphi:M\to N$ smooth. Set $u$ as
		\begin{equation*}
			u=e^{-f}
		\end{equation*}
		and $\bar{\varphi}:=\varphi\circ \pi_M:\bar{M}\to N$. Let $\alpha\in\mathbb{R}\setminus\{0\}$ and $\lambda\in\mathbb{R}$. Then the following are equivalent
		\begin{itemize}
			\item $\bar{M}:=M\times \mathbb{R}$ is harmonic-Einstein with respect to $\alpha$ and $\bar{\varphi}$, with $\bar{\varphi}$-scalar curvature $\lambda$ and with respect to the metric
			\begin{equation*}
				\bar{g}=g\pm u^2dt\otimes dt,
			\end{equation*}
			where $t$ is the coordinate of $\mathbb{R}$.
			\item $(M,g)$ satisfies
			\begin{equation}\label{grad einst type struct on M warp prod harm einst con fibra R}
			\begin{dcases}
			\mbox{Ric}^{\varphi}+\mbox{Hess}(f)-df\otimes df=\frac{\lambda}{m+1} g\\
			\tau(\varphi)=d\varphi(\nabla f)
			\end{dcases}
			\end{equation}
			and the following equation holds
			\begin{equation}\label{f laplacian for phi static metric}
			\Delta_ff=\frac{\lambda}{m+1},
			\end{equation}
			or equivalently, in terms of $u$,
			\begin{equation*}
				-\Delta u=\frac{\lambda}{m+1}u,
			\end{equation*}
			that is, $u$ is an eigenfunction of $-\Delta$ for the eigenvalue $\frac{\lambda}{m+1}$.
		\end{itemize}
	\end{cor}

	\begin{dfn}\label{def phi static}
		We say that a semi-Riemannian manifold $(M,g)$ of dimension $m\geq 2$ is {\em $\varphi$-static harmonic-Einstein} with respect to a smooth map $\varphi:M\to N$, where $(N,\eta)$ is a target Riemannian manifold, and $\alpha\in\mathbb{R}\setminus\{0\}$ if it satisfies one the equivalent conditions of the Proposition above for some $f\in\mathcal{C}^{\infty}(M)$ and $\lambda\in\mathbb{R}$.
	\end{dfn}

	\begin{rmk}
		If $(M,g)$ is $\varphi$-static harmonic-Einstein then it has constant $\varphi$-scalar curvature given by
		\begin{equation*}
			S^{\varphi}=\frac{m-1}{m+1}\lambda.
		\end{equation*}
		The validity of the above follows easily taking the trace of the first equation in \eqref{grad einst type struct on M warp prod harm einst con fibra R} and plugging by \eqref{f laplacian for phi static metric} into it.
	\end{rmk}

		Recall that a four dimensional Lorentzian manifold $(\bar{M},\bar{g})$ of dimension is called {\em static spacetime} if it admits a timelike and irrotational Killing vector field $K$. Locally any static spacetime is isometric to $M\times \mathbb{R}$ endowed with the metric $g-e^{-2f}dt\otimes dt$, where $(M,g)$ is a three-dimensional Riemannian manifold, $f\in\mathcal{C}^{\infty}(M)$ and $t$ is the coordinate in $\mathbb{R}$ where $K$ is given by $\partial/\partial t$ (see \cite{W} for details). 
		
		Furthermore, let $(N,\eta)$ be a Riemannian manifold and $\bar{\varphi}:(\bar{M},\bar{g})\to (N,\eta)$ a smooth map.
		\begin{dfn}
			We say that $(\bar{M},\bar{g})$ is a {\em $\bar{\varphi}$-static spacetime} if it admits a timelike and irrotational Killing vector field $K$ such that $d\bar{\varphi}(K)=0$.
		\end{dfn} 
		\begin{rmk}
			It is clear that in a (connected) neighborhood of each point any static spacetime is isometric to $M\times \mathbb{R}$ endowed with the metric $g-e^{-2f}dt\otimes dt$, where $(M,g)$ is a three-dimensional Riemannian manifold, $f\in\mathcal{C}^{\infty}(M)$ and $t$ is the coordinate in $\mathbb{R}$, but it can also be proved that in that neighborhood $\bar{\varphi}$ is given by the lifing of a time independent map $\varphi:M\to N$, i.e., $\bar{\varphi}=\varphi\circ \pi_M$ where $\pi_M:M\times \mathbb{R}\to M$ is the canonical projection.
		\end{rmk}
	\begin{rmk}
		The name $\varphi$-static harmonic-Einstein in \hyperref[def phi static]{Definition \ref*{def phi static}} is justified by the fact that, if $(M,g)$ is a three dimensional $\varphi$-static Riemannian manifold with respect to $\alpha$ given by \eqref{alpha per eq campo einstein}, then the four dimensional manifold $\bar{M}:=M\times \mathbb{R}$ endowed with the Lorentzian metric $\bar{g}:=g-{e^{-2f}}dt\otimes dt$, where $t\in\mathbb{R}$ represent the time, is a $\bar{\varphi}$-static spacetime, where $\bar{\varphi}=\varphi\circ \pi_M$ and $\pi_M:\bar{M}\to M$ is the canonical projection, that solves the Einstein field equations with source the wave map $\bar{\varphi}$. This fact follows easily from \hyperref[rmk sol eq Einstein]{Remark \ref*{rmk sol eq Einstein}}.
	\end{rmk}

	\section{Variational characterizations}\label{section variations}
	
	Let $M$ be a closed, i.e. a compact smooth manifold without boundary, and orientable manifold of dimension $m\geq 2$. We denote by $\mathcal{M}$ the set of all the Riemannian metrics of $M$.
	\begin{rmk}
		All the results of this Section can be extended to non-compact manifolds by considering compactly supported variations, but for simplicity we deal with closed and orientable manifolds in this work.
	\end{rmk}
	
	Let $g\in \mathcal{M}$, the volume form associated to $g$ is denoted by $\mu_g$ and is locally given by
	\begin{equation*}
		\mu_g=\sqrt{\mbox{det}(g_{ij})_{ij}}dx^1\wedge\ldots\wedge dx^m,
	\end{equation*}
	where $(x^1,\ldots ,x^m)$ are local coordinates on an open subset $\mathcal{U}$ of $M$ and $g$ is, in turn, locally given by
	\begin{equation*}
		g=g_{ij}dx^i\otimes dx^j.
	\end{equation*}
	We set
	\begin{equation*}
		\mbox{vol}_g(M):=\int_M\mu_g.
	\end{equation*}
	
	Recall that for every $g\in\mathcal{M}$ the tangent space $T_g\mathcal{M}$ of $\mathcal{M}$ at $g$ can be identified with $S^2(M)$, the set of two-times covariant tensor fields on $M$. The identification is the following: for $g\in\mathcal{M}$ and $h\in S^2(M)$ we define
	\begin{equation}\label{def di g t}
		g_t:=g+th \quad \mbox{ for } t\in(-\varepsilon,\varepsilon),
	\end{equation}
	where $\varepsilon>0$ is sufficient small so that $g_t$ is positive definite for every $t\in(-\varepsilon,\varepsilon)$. Then $g_0=g$ and
	\begin{equation*}
		\left.\frac{d}{dt}\right|_{t=0}g_t=h.
	\end{equation*}
	In a local chart the components of $g_t$ are given by
	\begin{equation*}
	g_{ij}(t)=g_{ij}+th_{ij}, \quad g_t=g_{ij}(t)dx^i\otimes dx^j,
	\end{equation*}
	where
	\begin{equation*}
		h=h_{ij}dx^i\otimes dx^j.
	\end{equation*}
	Clearly
	\begin{equation}\label{der di g t e h}
	\dot{g}=\left.\frac{d}{dt}\right|_{t=0}g_t=h,
	\end{equation}
	locally,
	\begin{equation*}
	\dot{g}=\dot{g}_{ij}dx^i\otimes dx^j, \quad \dot{g}_{ij}=h_{ij}.
	\end{equation*}
	
	We denote by $\mathcal{F}$ the set of all the smooth maps $\varphi:M\to (N,\eta)$, where the target $(N,\eta)$ is a fixed Riemannian manifold and we fix $\alpha\in\mathbb{R}\setminus\{0\}$. The results on the following Proposition are well known or follows easily from well known results. We sketch their proof for the sake of completeness.
	\begin{prop}
		Let $g\in\mathcal{M}$, $\varphi\in\mathcal{F}$ and $\alpha\in\mathbb{R}\setminus\{0\}$. Let $h\in S^2(M)$ and $g_t$ as in \eqref{def di g t}.
		\begin{itemize}
			\item Let $(g^{ij}(t))_{ij}$ be the inverse matrix of $(g_{ij}(t))_{ij}$. Then
			\begin{equation}\label{var inversa metrica g}
			\dot{g}^{ij}=\left.\frac{d}{dt}\right|_{t=0}g^{ij}(t)=-h^{ij},
			\end{equation}
			where the indexes of $h$ are raised with the aid of the metric $g$.
			\item The variation of the Riemannian volume element in the direction $h$ is given by
			\begin{equation}\label{var vol form}
			\dot{\mu}_{g}:=\left.\frac{d}{dt}\right|_{t=0}\mu_{g_t}=\frac{1}{2}\mbox{tr}_{g}(h)\mu_{g},
			\end{equation}
			\item The variation of the volume in the direction $h$ is given by
			\begin{equation}\label{variazione volume}
			\dot{vol}_{g}(M):=\left.\frac{d}{dt}\right|_{t=0}\mbox{vol}_{g_t}(M)=\frac{1}{2}\int_M\mbox{tr}_{g}(h)\mu_{g}.
			\end{equation}
			\item The variation of the $\varphi$-Ricci tensor is given by
			\begin{equation}\label{variation phi ricci}
				\dot{R}^{\varphi}_{ij}=\left.\frac{d}{dt}\right|_{t=0}R^{\varphi}_{ij}(t)=\frac{1}{2}g^{pq}(\nabla_q\nabla_jh_{ip}-\nabla_i\nabla_jh_{pq}+\nabla_q\nabla_ih_{jp}-\nabla_q\nabla_ph_{ij}).
			\end{equation}
			where
			\begin{equation*}
				\mbox{Ric}^{\varphi}_{g_t}=R^{\varphi}_{ij}(t)dx^i\otimes dx^j, \quad \nabla^{g}(\nabla^{g} h)=\nabla_k\nabla_th_{ij}dx^k\otimes dx^t\otimes dx^i\otimes dx^j.
			\end{equation*}
			\item The variation of the density of energy of $\varphi$ in the direction $h$ is given by
			\begin{equation}\label{variation of energy density}
			\dot{e^{g}(\varphi)}:=\left.\frac{d}{dt}\right|_{t=0}e^{g_t}(\varphi)=-\frac{1}{2}\langle h,\varphi^*\eta\rangle_{g}.
			\end{equation}
			\item The variation of the $\varphi$-scalar curvature in the direction $h$ is given by
			\begin{equation}\label{var phi scalar curvature}
			\dot{S}^{\varphi}_{g}:=\left.\frac{d}{dt}\right|_{t=0}S^{\varphi}_{g_t}=-\Delta_{g}(\mbox{tr}_{g}(h))+\mbox{div}_{g}(\mbox{div}_{g}(h))-\langle h,\mbox{Ric}_{g}^{\varphi}\rangle_{g}.
			\end{equation}
		\end{itemize}
	\end{prop}
	\begin{proof}
		The validity of \eqref{var inversa metrica g} is trivial. For the proof of \eqref{var vol form} and \eqref{variazione volume} see Proposition 1.186 of \cite{B}. Recall 1.1774 (d) of \cite{B}:
		\begin{equation*}
		\dot{R}_{ij}=\left.\frac{d}{dt}\right|_{t=0}R_{ij}(t)=\frac{1}{2}g^{pq}(\nabla_q\nabla_jh_{ip}-\nabla_i\nabla_jh_{pq}+\nabla_q\nabla_ih_{jp}-\nabla_q\nabla_ph_{ij}),
		\end{equation*}
		where locally
		\begin{equation*}
			\mbox{Ric}_{g_t}=R_{ij}(t)dx^i\otimes dx^j.
		\end{equation*}
		Since $\alpha\varphi^*\eta$ does not depend on the metric $g$ on $M$, the above gives the validity of \eqref{variation phi ricci}.
		
		The above equation gives
		\begin{equation}\label{var scalar curvature}
		\dot{S}_{g}:=\left.\frac{d}{dt}\right|_{t=0}S_{g_t}=-\Delta_{g}(\mbox{tr}_{g}(h))+\mbox{div}_{g}(\mbox{div}_{g}(h))-\langle h,\mbox{Ric}_{g}\rangle_{g},
		\end{equation}
		where we denoted by $\langle \,,\,\rangle_{g}$ the extension of $g$ to the bundle of the two-times covariant symmetric tensors on $M$, that is locally given by
		\begin{equation*}
		\langle h,\mbox{Ric}_{g}\rangle_{g}=g^{ij}h_{ij}R_{ij}.
		\end{equation*}
		
		Notice that, if we choose local coordinates $y^1,\ldots ,y^n$ on a open subset $\mathcal{V}$ of $N$ such that $\varphi(\mathcal{U})\subset \mathcal{V}$, then $|d\varphi|^2_{g_t}$ is locally given by
		\begin{equation}\label{local form d phi quadro}
		|d\varphi|^2_{g_t}=g^{ij}(t)\eta_{ab}\varphi^a_i\varphi^b_j,
		\end{equation}
		where $\eta$ is locally given by
		\begin{equation*}
		\eta_{ab}dy^a\otimes dy^b.
		\end{equation*}
		Using \eqref{local form d phi quadro} and \eqref{var inversa metrica g} we easily get \eqref{variation of energy density}. Combining \eqref{var scalar curvature} and \eqref{variation of energy density} and recalling the definition of the $\varphi$-Ricci tensor we obtain the variation formula \eqref{var phi scalar curvature} for the $\varphi$-scalar curvature.
	\end{proof}
	
	On the other hand we may vary the map $\varphi\in\mathcal{F}$. The tangent space $T_{\varphi}\mathcal{F}$ of $\mathcal{F}$ at $\varphi$ can be identified with $\Gamma(\varphi^{-1}TN)$, the set of smooth sections of $\varphi^{-1}TN$. The identification is the following: let $v$ be a smooth section of $\varphi^{-1}TN$. We define
	\begin{equation*}
		\Phi:M\times (-\varepsilon,\varepsilon)\to N, \quad \Phi(x,t)=\exp^N_{\varphi(x)}(tv_x),
	\end{equation*}
	where $\exp^N_y:T_yN\to N$ denotes the exponential map of $(N,\eta)$ at $y\in N$ and $\varepsilon>0$ is sufficiently small. Then, by setting
	\begin{equation}\label{def di d phi t}
		\varphi_t:=\Phi(\cdot,t)
	\end{equation}
	for every $t\in (-\varepsilon,\varepsilon)$, we have $\varphi_0=\varphi$ and
	\begin{equation*}
	\left.\frac{d}{dt}\right|_{t=0}\varphi_t=v.
	\end{equation*}
	
	Recall that the total energy of $\varphi$ is given by
	\begin{equation}\label{def energia tot}
		E^g(\varphi):=\int_Me^g(\varphi)\mu_g=\frac{1}{2}\int_M|d\varphi|^2_g\mu_g,
	\end{equation}
	where $e^g(\varphi)$ is the density of energy of $\varphi$, defined in \eqref{def dens en} while the total bi-energy of $\varphi$ is given by
	\begin{equation}\label{def bi energia totale}
		E_2^g(\varphi):=\int_Me_2^g(\varphi)\mu_g=\frac{1}{2}\int_M|\tau^g(\varphi)|^2\mu_g,
	\end{equation}
	where $e_2^g(\varphi)$ is the density of bi-energy of $\varphi$, defined in \eqref{def dens bi-ener}.
	
	The results on the following Proposition are well known. We sketch their proof for completeness and to show how the method of the moving frame makes computation easier.
	\begin{prop}\label{prop energia tot e bienergia punti critici}
		Let $g\in\mathcal{M}$, $\varphi\in\mathcal{F}$ and $\alpha\in\mathbb{R}\setminus\{0\}$. Let $h\in S^2(M)$ and $g_t$ as in \eqref{def di g t}.
		\begin{itemize}
			\item The variation of the energy of $\varphi$ in the direction $h$ is given by
			\begin{equation}\label{var energ rispetto metrica}
			\dot{E^{g}(\varphi)}=\left.\frac{d}{dt}\right|_{t=0}E^{g_t}(\varphi)=-\frac{1}{2}\int_M\langle T^{g}, h\rangle_{g}\mu_{g},
			\end{equation}
			where $T^g$ is the energy stress tensor \eqref{stress energy tensor} of the map $\varphi:(M,g)\to (N,\eta)$.
			\item The variation of the bi-energy of $\varphi$ in the direction $h$ is given by
			\begin{equation}\label{var bi energ rispetto metrica}
			\dot{E_2^{g}(\varphi)}=\left.\frac{d}{dt}\right|_{t=0}E_2^{g_t}(\varphi)=\frac{1}{2}\int_M\langle T_2^{g}, h\rangle_{g}\mu_{g},
			\end{equation}
			where, in a local orthonormal coframe for $g$,
			\begin{equation*}
			(T_2^g)_{ij}=\tau(\varphi)^a_i\varphi^a_j+\tau(\varphi)^a_j\varphi^a_i-\left(e_2(\varphi)+\tau(\varphi)^a_k\varphi^a_k\right)\delta_{ij}.
			\end{equation*}
		\end{itemize}
		Let $v\in \Gamma(\varphi^{-1}TN)$ and $\varphi_t$ such that \eqref{def di d phi t} holds.
		\begin{itemize}
			\item The variation of total energy of $\varphi$ in the direction $v$ is given by
			\begin{equation}\label{var energia se cambio mappa}
			\left.\frac{d}{dt}\right|_{t=0}E^{g}(\varphi_t)=-\int_M (\tau^g(\varphi),v)\mu_{g},
			\end{equation}
			where $(\,,\,)$ denotes the inner product on $\varphi^{-1}TN$ and $\tau^g(\varphi)$ is the tension field \eqref{def tension field} of the map $\varphi:(M,g)\to (N,\eta)$.
			\item The variation of total bi-energy of $\varphi$ in the direction $v$ is given by
			\begin{equation}\label{var bi energ rispetto mappa}
			\left.\frac{d}{dt}\right|_{t=0}E_2^{g}(\varphi_t)=\int_M (\tau^g_2(\varphi),v)\mu_{g},
			\end{equation}
			where $\tau_2^g(\varphi)$ is the bi-tension field of the map $\varphi:(M,g)\to (N,\eta)$, given by \eqref{def bi tension}.
		\end{itemize}
	\end{prop}
	\begin{proof}
		The validity of \eqref{var energ rispetto metrica} follows from the definition of energy of $\varphi$ and the formulas \eqref{variation of energy density} and \eqref{var vol form}, since
		\begin{equation*}
			\left.\frac{d}{dt}\right|_{t=0}\int_Me^{g_t}(\varphi)\mu_{gt}=\int\left(\left.\frac{d}{dt}\right|_{t=0}e^{g_t}(\varphi)\right)\mu_g+\int_Me^{g}(\varphi)\left.\frac{d}{dt}\right|_{t=0}\mu_{g_t}.
		\end{equation*}
		
		Recall that
		\begin{equation*}
		\nabla_{\frac{\partial}{\partial x^i}}\frac{\partial}{\partial x^j}=\Gamma^k_{ij}\frac{\partial}{\partial x^k}
		\end{equation*}
		and
		\begin{equation*}
		\Gamma^k_{ij}=\frac{1}{2}g^{tk}\left(\frac{\partial g_{tj}}{\partial x^i}+\frac{\partial g_{ti}}{\partial x^j}-\frac{\partial g_{ij}}{\partial x^t}\right),
		\end{equation*}
		Notice that $\Gamma^k_{ij}$ does not define a tensor while $\dot{\Gamma}^k_{ij}$ does, where
		\begin{equation}\label{var christoffel symbols}
		\dot{\Gamma}^k_{ij}=\frac{1}{2}g^{tk}(\nabla_ih_{tj}+\nabla_jh_{it}-\nabla_th_{ij}).
		\end{equation}
		Recall that the components of the generalized second fundamental form of $\varphi$ are given by
		\begin{equation}
		\nabla d\varphi^a_{ij}=\frac{\partial^2\varphi^a}{\partial x^i\partial x^j}-\Gamma^k_{ij}\frac{\partial \varphi^a}{\partial x^k}+{}^N\Gamma^a_{bc}\frac{\partial \varphi^b}{\partial x^i}\frac{\partial \varphi^c}{\partial x^j}.
		\end{equation}
		Using \eqref{var christoffel symbols} we get
		\begin{equation}\label{var gen sec fund form}
		\dot{\nabla d\varphi}^a_{ij}=-\dot{\Gamma}^k_{ij}\frac{\partial \varphi^a}{\partial x^k}=-\frac{1}{2}g^{tk}(\nabla_ih_{tj}+\nabla_jh_{it}-\nabla_th_{ij})\frac{\partial \varphi^a}{\partial x^k}.
		\end{equation}
		The components of the tension field of $\varphi$ are given by
		\begin{equation*}
		\tau^g(\varphi)^a=g^{ij}\nabla d\varphi^a_{ij}.
		\end{equation*}
		Using \eqref{var inversa metrica g} and \eqref{var gen sec fund form}
		\begin{equation*}
		\dot{\tau^g(\varphi)}^a=\dot{g}^{ij}\nabla d\varphi^a_{ij}+g^{ij}\dot{\nabla d\varphi}^a_{ij}=-h^{ij}\nabla d\varphi^a_{ij}-\frac{1}{2}g^{tk}g^{ij}(\nabla_ih_{tj}+\nabla_jh_{it}-\nabla_th_{ij})\frac{\partial \varphi^a}{\partial x^k}
		\end{equation*}
		In a local orthonormal coframe for $g$ the above equation reads
		\begin{equation*}
		\dot{\tau^g(\varphi)}^a=-h_{ij}\varphi^a_{ij}-\frac{1}{2}(2h_{ki,i}-h_{ii,k})\varphi^a_k
		\end{equation*}	
		Hence
		\begin{equation*}
		\dot{e_2^g(\tau)}=\tau^g(\varphi)^a\dot{\tau^g(\varphi)}^a=-h_{ij}\varphi^a_{ij}\tau(\varphi)^a-h_{ki,i}\tau(\varphi)^a\varphi^a_k+\frac{1}{2}h_{ii,k}\tau(\varphi)^a\varphi^a_k.
		\end{equation*}
		Then, using \eqref{var vol form},
		\begin{equation}
		\dot{E^g_2(\varphi)}=\int_M\dot{e^g_2(\tau)}\mu_g+\int_Me^g_2(\tau)\dot{\mu_g}=\int_M\left(\dot{e^g_2(\tau)}+\frac{1}{2}e^g_2(\tau)\mbox{tr}_g(h)\right)\mu_g.
		\end{equation}
		Notice that
		\begin{align*}
		\dot{e^g_2(\tau)}+\frac{1}{2}e^g_2(\tau)\langle h, g\rangle_g=&-h_{ij}\varphi^a_{ij}\tau^g(\varphi)^a-h_{ki,i}\tau^g(\varphi)^a\varphi^a_k+\frac{1}{2}h_{ii,k}\tau^g(\varphi)^a\varphi^a_k+\frac{1}{2}e^g_2(\varphi)h_{ij}\delta_{ij}\\
		=&-h_{ij}\varphi^a_{ij}\tau^g(\varphi)^a+h_{ki}(\tau^g(\varphi)^a\varphi^a_k)_i-\frac{1}{2}h_{ii}(\tau^g(\varphi)^a\varphi^a_k)_k+\frac{1}{2}e^g_2(\varphi)h_{ij}\delta_{ij}+\ldots\\
		=&h_{ij}\left(\tau^g(\varphi)^a_i\varphi^a_j+\frac{1}{2}(e^g_2(\varphi)-\tau^g(\varphi)^a_k\varphi^a_k-|\tau^g(\varphi)|^2)\delta_{ij}\right)+\ldots\\
		=&h_{ij}\left(\tau^g(\varphi)^a_i\varphi^a_j-\frac{1}{2}(e^g_2(\varphi)+\tau^g(\varphi)^a_k\varphi^a_k)\delta_{ij}\right)+\ldots
		\end{align*}
		where with the lower dots we denote divergences terms. Then, integrating by parts we obtain \eqref{var bi energ rispetto metrica}.
		
		Now we deal with variations of $\varphi$. Clearly
		\begin{equation}
			\left.\frac{d}{dt}\right|_{t=0}e^g(\varphi_t)=\varphi^a_i\left.\frac{d}{dt}\right|_{t=0}(\varphi_t)^a_i,
		\end{equation}
		hence exchanging the covariant derivatives of $\Phi:(-\varepsilon,\varepsilon)\times M\to N$, where $\Phi(t,x)=\varphi_t(x)$, we have
		\begin{equation*}
			\left.\frac{d}{dt}\right|_{t=0}e^g(\varphi_t)=\varphi^a_iv^a_i=(\varphi^a_iv^a)_i-\varphi^a_{ii}v^a.
		\end{equation*}
		Integrating the above, using the divergence theorem, we conclude the validity of \eqref{var energia se cambio mappa}.
		
		Finally
		\begin{equation}
			\left.\frac{d}{dt}\right|_{t=0}e_2^g(\varphi_t)=\tau(\varphi)^a\left.\frac{d}{dt}\right|_{t=0}\tau(\varphi_t)^a.
		\end{equation}
		Using \eqref{comm rule der terza phi} for the map $\Phi$, since the components $\bar{R}_{\alpha \beta \gamma \delta}$, for $1\leq \alpha,\beta,\ldots \leq m+1$, of the Riemann tensor of $\bar{M}=(-\varepsilon,\varepsilon)\times M$ satisfies (it can be seen using \eqref{comp riem warp prod} with $u\equiv 1$)
		\begin{equation*}
			\bar{R}_{kj \, m+1\, t}=0, \quad \bar{R}_{m+1\, j\, m+1\, t}=0, \quad \bar{R}_{kjst}=R_{kjst},
		\end{equation*}
		we obtain
		\begin{align*}
			\frac{d}{dt}(\varphi_t)^a_{ij}=&\Phi^a_{ij\, m+1}=\Phi^a_{i\, m+1\, j}-\bar{R}^{\alpha}_{i \, m+1\, j}\Phi^a_{\alpha}+{}^NR^a_{bcd}\Phi^b_i\Phi^c_{m+1}\Phi^d_j\\
			=&\frac{d}{dt}(\varphi_t)^a_{ij}+{}^NR^a_{bcd}(\varphi_t)^b_i\frac{d}{dt}(\varphi_t)^c(\varphi_t)^d_j,
		\end{align*}
		hence, evaluating at $t=0$,
		\begin{equation*}
			\left.\frac{d}{dt}\right|_{t=0}(\varphi_t)^a_{ij}=v^a_{ij}+{}^NR^a_{bcd}\varphi^b_iv^c\varphi^d_j
		\end{equation*}
		Tracing the above we conclude
		\begin{equation*}
			\left.\frac{d}{dt}\right|_{t=0}\tau^g(\varphi_t)^a=v^a_{ii}+{}^NR^a_{bcd}\varphi^b_i\varphi^d_jv^c.
		\end{equation*}
		Then we infer
		\begin{equation*}
			\frac{d}{dt}E_2^{g}(\varphi_t)=\int_M(v^a_{ii}+{}^NR^a_{bcd}(\varphi_t)^b_i(\varphi_t)^d_jv^c)\tau(\varphi_t)^a\mu_g,
		\end{equation*}
		that integrating by parts twice gives \eqref{var bi energ rispetto mappa}.
	\end{proof}

	\begin{rmk}\label{rmk stress energy tensor har e biharm map}
		It is well known, see for instance Proposition 1.1.17 of \cite{A} for a proof, that
		\begin{equation}\label{div stress energy tensor}
			\mbox{div}(T^g)_i=\tau^g(\varphi)^a\varphi^a_i.
		\end{equation}
		In particular harmonic maps are conservative. The analogous happens also for the bi-energy, that is,
		\begin{equation}\label{div energy stress per bi energy}
			\mbox{div}(T^g_2)_i=\tau^g_2(\varphi)^a\varphi^a_i.
		\end{equation}
		To prove \eqref{div energy stress per bi energy} notice that
		\begin{align*}
		(T^g_2)_{ij,j}=&\tau^g(\varphi)^a_{ij}\varphi^a_j+\tau^g(\varphi)^a_i\tau^g(\varphi)^a+\tau^g(\varphi)^a_{jj}\varphi^a_i+\tau^g(\varphi)^a_j\varphi^a_{ij}-\left(e^g_2(\varphi)+\tau^g(\varphi)^a_k\varphi^a_k\right)_i,
		\end{align*}
		that is,
		\begin{equation}\label{scritt div T2 in corso}
		(T^g_2)_{ij,j}=(\tau^g(\varphi)^a_{ij}-\tau^g(\varphi)^a_{ji})\varphi^a_j+\tau^g(\varphi)^a_{jj}\varphi^a_i,
		\end{equation}
		
		Notice that \eqref{comm rule der quarta phi} implies
		\begin{equation*}
		\tau^g(\varphi)^a_{ij}=\tau^g(\varphi)^a_{ji}+2R^s_{kij}\varphi^a_{ks}-{}^NR^a_{bcd}\varphi^b_{kk}\varphi^c_i\varphi^d_j,
		\end{equation*}
		that is, using the symmetries of $\mbox{Riem}$ and of $\nabla d\varphi$,
		\begin{equation*}
		\tau^g(\varphi)^a_{ij}-\tau^g(\varphi)^a_{ji}=-{}^NR^a_{bcd}\varphi^b_{kk}\varphi^c_i\varphi^d_j.
		\end{equation*}
		Plugging the above into \eqref{scritt div T2 in corso} we get
		\begin{equation*}
		\mbox{div}(T^g_2)_i=\left(\tau^g(\varphi)^a_{jj}-{}^NR^a_{bcd}\varphi^b_j\varphi^c_j\tau^g(\varphi)^d\right)\varphi^a_i,
		\end{equation*}
		that is \eqref{div energy stress per bi energy}, recalling \eqref{def bi tension}.
		
		For a detailed study of the stress-energy tensor $T_2^g$ for biharmonic maps we refer \cite{LMO}. Notice that the stress-energy tensor $T^g$ for harmonic maps play a special role for bidimensional manifolds: indeed it is well known that $T^g=0$ if and only if $\varphi:(M,g)\to (N,\eta)$ is weakly conformal and $m=2$, where $m$ is the dimension of $M$. We expect that the the stress-energy tensor for biharmonic maps could play a special role for four dimensional manifolds and we will investigate it in future works.
	\end{rmk}
	
	\subsection{The linearization of the $\varphi$-scalar curvature map}\label{section The linearization of the}
		We can consider the $\varphi$-scalar curvature as a map
		\begin{equation*}
			\mathbb{S}:\mathcal{M}\times \mathcal{F}\to \mathcal{C}^{\infty}(M), \quad (g,\varphi)\mapsto \mathbb{S}(g,\varphi)\equiv \mathbb{S}^{\varphi}(g)\equiv \mathbb{S}^g(\varphi)=S^{\varphi}_g.
		\end{equation*}
		Then, for every $(g,\varphi)\in \mathcal{M}\times \mathcal{F}$, the linearization of the $\varphi$-scalar curvature map $\mathbb{S}$ at $(g,\varphi)$
		\begin{equation*}
			d_{(g,\varphi)}\mathbb{S}:T_{(g,\varphi)}(\mathcal{M}\times\mathcal{F})\to \mathcal{C}^{\infty}(M)
		\end{equation*}
		is given by, using the identification $T_{g,\varphi}(\mathcal{M}\times\mathcal{F})\equiv S^2(M)\oplus \Gamma(\varphi^{-1}TN)$,
		\begin{equation*}
			(d_{(g,\varphi)}\mathbb{S})(h,v)=(d_g\mathbb{S}^{\varphi})(h)+(d_{\varphi}\mathbb{S}^g)(v) \quad \mbox{ for every } (h,v)\in S^2(M)\oplus \Gamma(\varphi^{-1}TN).
		\end{equation*}
		Using \eqref{var phi scalar curvature} and \eqref{var S phi rispetto phi} (that will be proved below),
		\begin{equation}
			(d_{(g,\varphi)}\mathbb{S})(h,\varphi)=-\Delta_g(\mbox{tr}_g(h))+\mbox{div}_g(\mbox{div}_g(h))-\langle h,\mbox{Ric}_g^{\varphi}\rangle_g-2\alpha \varphi^a_iv^a_i.
		\end{equation}
			It is easy to see that the adjoint $(d_{(g,\varphi)}\mathbb{S})^*:\mathcal{C}^{\infty}(M)\to S^2(M)\times \Gamma(\varphi^{-1}TN)$ of $d_{(g,\varphi)}\mathbb{S}$ is given by
			\begin{equation}\label{form adjoint phi scalar curv map}
			(d_{(g,\varphi)}\mathbb{S})^*(u)=[(d_{(g,\varphi)}\mathbb{S})^*(u)_1,(d_{(g,\varphi)}\mathbb{S})^*(u)_2],
			\end{equation}
			where
			\begin{equation*}
				\begin{dcases}
				(d_{(g,\varphi)}\mathbb{S})^*(u)_1:=\mbox{Hess}_g(u)-u\mbox{Ric}_g^{\varphi}-\Delta_g u g\in S^2(M)\\
				(d_{(g,\varphi)}\mathbb{S})^*(u)_2:=2\alpha[u\tau^g(\varphi)+d\varphi(\nabla_gu)]\in \Gamma(\varphi^{-1}TN).
				\end{dcases}
			\end{equation*}
			To prove \eqref{form adjoint phi scalar curv map} it is sufficient to prove that for every $u\in\mathcal{C}^{\infty}(M)$
			\begin{align*}
			\int_Mu\left(-\Delta_g(\mbox{tr}_g(h))+\mbox{div}_g(\mbox{div}_g(h))-\langle h,\mbox{Ric}^{\varphi}_g\rangle_g-2\alpha \varphi^a_iv^a_i\right)\mu_g\\
			=\int_M\left[\langle h,\mbox{Hess}_g(u)-u\mbox{Ric}^{\varphi}_g-\Delta_gu g \rangle_g+2\alpha(v,u\tau^g(\varphi)+d\varphi(\nabla_gu))\right]\mu_g,
			\end{align*}
			that follows easily from the divergence theorem.
			
			In the following Proposition we show that $d_{(g,\varphi)}\mathbb{S}$ is surjective (or equivalently, $(d_{(g,\varphi)}\mathbb{S})^*$ is injective), unless some particular conditions on $(g,\varphi)\in\mathcal{M}\times \mathcal{F}$ are satisfied.
		\begin{prop}\label{prop aggiunto iniettivo}
			Let $M$ be a compact manifold and let $(g,\varphi)\in \mathcal{M}\times \mathcal{F}$. If $(d_{(g,\varphi)}\mathbb{S})^*$ is not injective, then one of the following hold.
			\begin{itemize}
				\item $(M,g)$ is $\varphi$-Ricci flat with respect to $\alpha$ and
				\begin{equation}\label{nucleo aggiunto formale banale}
					\mbox{ker}(d_{(g,\varphi)}\mathbb{S})^*=\mathbb{R}.
				\end{equation}
				\item There exists a non-constant function $u\in\mathcal{C}^{\infty}(M)$ such that $\Sigma:=u^{-1}(\{0\})$ is a total umbilical hypersurface of $(M,g)$ and if $\mathcal{U}$ is a connected component of $M\setminus \Sigma$, then $\bar{\mathcal{U}}:=\mathcal{U}\times \mathbb{R}$ endowed with the metric $\bar{g}=g\pm u^2dt\otimes dt$, where $t$ is the coordinate on $\mathbb{R}$, is harmonic-Einstein with respect to $\alpha$ and $\bar{\varphi}:=\varphi\circ \pi_M$, where $\pi_M:\bar{M}\to M$ is the canonical projection.
			\end{itemize}
		\end{prop}
		\begin{proof}	
			From \eqref{form adjoint phi scalar curv map}, $u\in \mbox{ker}(d_{(g,\varphi)}\mathbb{S})^*$ if and only if
			\begin{equation}\label{ker form adj scalar curv map}
			\begin{dcases}
			\mbox{Hess}_g(u)-u\mbox{Ric}_g^{\varphi}-\Delta_g u g=0\\
			u\tau^g(\varphi)+d\varphi(\nabla_gu)=0.
			\end{dcases}
			\end{equation}
			
			Notice that a non-zero constant $u$ belongs to $\mbox{ker}(d_{(g,\varphi)}S)^*$ if and only if $(M,g)$ is $\varphi$-Ricci flat and, if this is the case, then \eqref{nucleo aggiunto formale banale} holds. The equivalence follows immediately from \eqref{ker form adj scalar curv map} and if $(M,g)$ is $\varphi$-Ricci flat \eqref{ker form adj scalar curv map} reduces to
			\begin{equation}\label{hess u prop a delta u nel provare corvino}
			\begin{dcases}
			\mbox{Hess}_g(u)=\Delta_g u g\\
			d\varphi(\nabla_gu)=0.
			\end{dcases}
			\end{equation}
			Tracing the first equation above we conclude that $u$ is harmonic and, since $M$ is compact, is constant.
			
			Assume that $u$ is non-constant. Taking the trace of the first equation of \eqref{ker form adj scalar curv map} we get
			\begin{equation}\label{eq aut per u nel caso compatto}
			-\Delta u=\frac{\lambda}{m+1} u,
			\end{equation}
			where
			\begin{equation*}
			\lambda:=\frac{m+1}{m-1}S^{\varphi}.
			\end{equation*}
			Then $u$ satisfies a unique continuation property and thus, since it is not identically zero, it cannot vanish on an open subset of $M$.
			
			Now we show that $S^{\varphi}$ is constant. Taking the divergence of the first equation of \eqref{hess u prop a delta u nel provare corvino} we get
			\begin{equation}
			\eta^{ik}u_{ijk}-u^iR^{\varphi}_{ij}-u\eta^{ik}R^{\varphi}_{ij,k}-(\Delta u)_j=0.
			\end{equation}
			Using \eqref{div of phi Ricci} and commutating the indexes we obtain
			\begin{equation}
			R_{tj}u^t-u^iR^{\varphi}_{ij}-u\frac{1}{2}S^{\varphi}_j+\alpha u\tau(\varphi)^a\varphi^a_j=0.
			\end{equation}
			Using the second equation of \eqref{ker form adj scalar curv map} and the definition of $\varphi$-Ricci we conclude
			\begin{equation*}
			u dS^{\varphi}=0.
			\end{equation*}
			Then $S^{\varphi}$ is constant on $M$.
			
			It can be easily proved that $\Sigma:=u^{-1}(\{0\})$, if it is not empty, is a total umbilical hypersurface of $(M,g)$ (it follows from the fact that $du\neq 0$ on $\Sigma$, see the proof of Proposition 2.3 of \cite{C}). Notice that, since $u$ is non-constant and $M$ is compact, from \eqref{eq aut per u nel caso compatto} we deduce that $S^{\varphi}$ must be a positive constant and $u$ must change sign, hence $\Sigma\neq \varnothing$.
			
			By setting, on a fixed connected component $\mathcal{U}$ of $M\setminus \Sigma$,
			\begin{equation*}
			u:=\pm e^{-f},
			\end{equation*}
			according to the sign of $u$ on $\mathcal{U}$, then the validity of \eqref{ker form adj scalar curv map} gives, on $\mathcal{U}$,
			\begin{equation*}
			\begin{dcases}
			\mbox{Ric}^{\varphi}_g+\mbox{Hess}(f)-df\otimes df=(\Delta_g f-|\nabla f|^2_g) g\\
			\tau^g(\varphi)=d\varphi(\nabla_g f).
			\end{dcases}
			\end{equation*}
			Notice that, taking the trace of the first equation of the above,
			\begin{equation*}
			S^{\varphi}=(m-1)\Delta_f f,
			\end{equation*}
			hence the above can be rewritten as
			\begin{equation*}
			\begin{dcases}
			\mbox{Ric}^{\varphi}_g+\mbox{Hess}(f)-df\otimes df=\frac{\lambda}{m+1}g\\
			\tau^g(\varphi)=d\varphi(\nabla_g f).
			\end{dcases}
			\end{equation*}
			
			Then, using \hyperref[cor warp prod harm einst con fibra R]{Corollary \ref*{cor warp prod harm einst con fibra R}}, we have that $\bar{\mathcal{U}}:=\mathcal{U}\times \mathbb{R}$ endowed with the metric $\bar{g}=g\pm u^2dt\otimes dt$, where $t$ is the coordinate on $\mathbb{R}$, is harmonic-Einstein with respect to $\alpha$ and $\bar{\varphi}:=\varphi\circ \pi_M$, where $\pi_M:\bar{M}\to M$ is the canonical projection.
		\end{proof}
		
		The following Corollary follows automatically from the above Proposition
		\begin{cor}\label{cor estensione corvino}
			Let $(M,g)$ be a compact Riemannian manifold, $\alpha\in\mathbb{R}\setminus \{0\}$ and $u\in\mathcal{C}^{\infty}(M)$ a non-constant function. Then $u\in \mbox{Ker}(d_{(g,\varphi)}\mathbb{S})^*$ if and only if $\Sigma:=u^{-1}(\{0\})$ is a total umbilical hypersurface of $(M,g)$ and the (possibly disconnected) Riemannian manifold $M\setminus \Sigma$ is $\varphi$-static harmonic-Einstein, in the sense of \hyperref[def phi static]{Definition \ref{def phi static}}, with respect to $\alpha$ and $f$, where $f=-\log|u|$ on $M\setminus \Sigma$. In other words, $\bar{M}:=M\times \mathbb{R}$ endowed with the metric $\bar{g}=g- u^2dt\otimes dt$, where $t$ is the coordinate on $\mathbb{R}$, is harmonic-Einstein with respect to $\alpha$ and $\bar{\varphi}:=\varphi\circ \pi_M$, where $\pi_M:\bar{M}\to M$ is the canonical projection, outside of $\Sigma$ (where $\bar{g}$ degenerates).
		\end{cor}
	
		\begin{rmk}
			The Corollary above shows that compact Riemannian manifolds admitting a non-constant smooth function in $u\in\mbox{Ker}(d_{(g,\varphi)}\mathbb{S})^*$ are (possibly disconnected) $\varphi$-static harmonic-Einstein manifolds endowed with a \textquotedblleft horizon \textquotedblright given by the zero-locus of $u$. Notice that \hyperref[cor estensione corvino]{Corollary \ref*{cor estensione corvino}} is an extension, in the compact case, of Proposition 2.7 of \cite{C}. It is possible also to deal with the non-compact case and to study in more detail the image of the $\varphi$-scalar curvature in the compact case, as done in \cite{FM}. Those tasks will be addressed, possibly, in some future works.
		\end{rmk}

	\subsection{Variational derivation of the harmonic-Einstein equations}\label{section Variational derivation of the harmonic-Einstein equations}
	
	\begin{dfn}
		The functional of {\em total $\varphi$-scalar curvature}, for every $(g,\varphi)\in\mathcal{M}\times \mathcal{F}$ and $\alpha\in\mathbb{R}\setminus\{0\}$, is given by
		\begin{equation}\label{def di phi R operatore}
		\mathcal{S}(g,\varphi)\equiv \mathcal{S}^{\varphi}(g)\equiv \mathcal{S}^g(\varphi):=\int_MS^{\varphi}_g\mu_g.
		\end{equation}
	\end{dfn}
	\begin{rmk}
		Denoting by $\mathcal{S}(g)$ the total scalar curvature of $(M,g)$, from the relation $S^{\varphi}=S-\alpha|d\varphi|^2$ and of total energy of $\varphi$ we immediately deduce
		\begin{equation}\label{oper S g phi relazionato a S e E}
		\mathcal{S}(g,\varphi)=\mathcal{S}(g)-2\alpha E^g(\varphi).
		\end{equation}
	\end{rmk}

	\begin{rmk}\label{remark per superfici}
		For $m=2$ the total $\varphi$-scalar curvature is given by
		\begin{equation*}
			\mathcal{S}(g,\varphi)=4\pi \chi(M)-2\alpha  E^g(\varphi),
		\end{equation*}
		where $\chi(M)$ is the Euler characteristic of $M$. Indeed, from Gauss-Bonnet formula,
		\begin{equation*}
			\frac{1}{2}\int_MS_g\mu_g=2\pi \chi(M),
		\end{equation*}
		hence the above follows easily from \eqref{oper S g phi relazionato a S e E}.
		
		As a consequence, $(g,\varphi)\in\mathcal{S}\times \mathcal{F}$ is a critical point for $\mathcal{S}$ if and only if it is a critical point for the total energy of $\varphi$. We have characterized critical point to the total energy of $\varphi$ in \hyperref[prop energia tot e bienergia punti critici]{Proposition \ref*{prop energia tot e bienergia punti critici}}: they satisfy
		\begin{equation*}
			\begin{dcases}
			\tau^g(\varphi)=0\\
			T^g=0.
			\end{dcases}
		\end{equation*}
		As mentioned in \hyperref[rmk stress energy tensor har e biharm map]{Remark \ref*{rmk stress energy tensor har e biharm map}}, $T^g=0$ if and only if $m=2$ and $\varphi$ is weakly conformal. Recall that $\varphi:(M,g)\to (N,\eta)$, where $M$ is a surface, is called branched minimal immersion (see \cite{BW}, Section 3.5) if it is weakly-conformal and harmonic. In conclusion, critical points of $\mathcal{S}$ for $m=2$ are given by $(g,\varphi)\in\mathcal{S}\times \mathcal{F}$ such that $\varphi:(M,g)\to (N,\eta)$ is a branched minimal immersion.
		
		Notice that the fact of being a branched minimal immersion does not depend only on the Riemannian metric $g$ but depends on the conformal class $[g]$ of $g$. Indeed, using \eqref{conf change energy of phi} and \eqref{conf change metric vol element}, we get that
		\begin{equation*}
		E^{\widetilde{g}}(\varphi)=E^g(\varphi).
		\end{equation*}
		This can be seen also directly: from \eqref{tension phi conf change metric} we have
		\begin{equation}\label{tension phi conf change metric surface}
		\tau^{\widetilde{g}}(\varphi)=e^{2h}\tau^g(\varphi)
		\end{equation}
		and clearly $\varphi$ is weakly conformal with respect to $\widetilde{g}$ if and only if it is with respect to $g$.
		
		Finally \eqref{conformal change metric scalar curvature} and \eqref{conf change metric vol element} give
		\begin{equation}\label{conformal change metric scalar curvature surface}
		S^{\varphi}_{\widetilde{g}}\mu_{\widetilde{g}}=S^{\varphi}_g\mu_g+2\Delta_g h\mu_g,
		\end{equation}
		and this in another equivalent way to see that $\mathcal{S}^{\varphi}(g)=\mathcal{S}^{\varphi}(\widetilde{g})$.
	\end{rmk}
	
	From now on assume that $m\geq 3$.
	\begin{rmk}\label{rmk phi scalar curv homog}
		Notice that $\mathcal{S}^{\varphi}$ is not scale invariant, it is homogeneous of degree $\frac{m-2}{2}$, that is, for every $\lambda>0$,
		\begin{equation}\label{tot scal curv hom di grado}
		\mathcal{S}^{\varphi}(\lambda g)=\lambda^{\frac{m-2}{2}} \mathcal{S}^{\varphi}( g).
		\end{equation}
		To prove \eqref{tot scal curv hom di grado} we set $\widetilde{g}:=\lambda g$ and we use \eqref{conformal change metric scalar curvature} and \eqref{conf change metric vol element} with $h\in\mathbb{R}$ such that $\lambda=e^{-2h}$, that are $\lambda S^{\varphi}_{\widetilde{g}}=S^{\varphi}_g$ and $\mu_{\widetilde{g}}=\lambda^{\frac{m}{2}}\mu_g$ to get
		\begin{equation*}
		\int_MS^{\varphi}_{\widetilde{g}}\mu_{\widetilde{g}}=\lambda^{\frac{m}{2}-1}\int_MS^{\varphi}_g\mu_g.
		\end{equation*}
	\end{rmk}

	To overcome this issue we will study also another functional.
	\begin{dfn}
		We set the {\em rescaled total $\varphi$-scalar curvature of $(M,g)$} as
		\begin{equation}\label{rescaled total scalar curvature}
		\bar{\mathcal{S}}(g,\varphi)=\bar{\mathcal{S}}^{\varphi}(g)=\bar{\mathcal{S}}^g(\varphi):=\mbox{vol}_{g}(M)^{-\frac{m-2}{m}}\mathcal{S}^{\varphi}(g).
		\end{equation}
	\end{dfn}
	\begin{rmk}
		It is easy to see that, proceeding as in \hyperref[rmk phi scalar curv homog]{Remark \ref*{rmk phi scalar curv homog}}, $\mbox{vol}_{\lambda g}(M)=\lambda^{\frac{m}{2}}\mbox{vol}_g(M)$ for every $g\in\mathcal{M}$ and $\lambda>0$. Combining it with \eqref{tot scal curv hom di grado} we immediately get, for every $g\in\mathcal{M}$ and $\lambda>0$
		\begin{equation*}
			\bar{\mathcal{S}}^{\varphi}(\lambda g)=\bar{\mathcal{S}}^{\varphi}(g),
		\end{equation*}
		that is, $\bar{\mathcal{S}}^{\varphi}$ is scale invariant.
	\end{rmk}
	
	\begin{prop}
		Let $M$ be a compact manifold of dimension $m\geq 3$ and let $(N,\eta)$ be a Riemannian manifold.
		\begin{itemize}
			\item The pair $(g,\varphi)\in \mathcal{M}\times \mathcal{F}$ is a critical point of the functional $\mathcal{S}$ on $\mathcal{M}\times \mathcal{F}$ if and only if
			\begin{equation*}
			\begin{dcases}
			\mbox{Ric}_g^{\varphi}=0\\
			\tau^g(\varphi)=0,
			\end{dcases}
			\end{equation*}
			that is, if and only if  $(M,g)$ is $\varphi$-Ricci flat with respect to $\alpha$. 
			\item The pair $(g,\varphi)\in \mathcal{M}\times \mathcal{F}$ is a critical point of the functional $\bar{\mathcal{S}}$ on $\mathcal{M}\times \mathcal{F}$ if and only if
			\begin{equation*}
			\begin{dcases}
			\mathring{\mbox{Ric}}^{\varphi}=0\\
			\tau(\varphi)=0,
			\end{dcases}
			\end{equation*}
			that is, if and only if  $(M,g)$ is harmonic-Einstein with respect to $\varphi$ and $\alpha$.
		\end{itemize}
	\end{prop}
	\begin{proof}
		Clearly
		\begin{equation*}
		\left.\frac{d}{dt}\right|_{t=0}\mathcal{S}^{\varphi}(g_t)=\int_M\left(\left.\frac{d}{dt}\right|_{t=0}S^{\varphi}_{g_t}\right)\mu_{g}+\int_MS^{\varphi}_{g}\left.\frac{d}{dt}\right|_{t=0}\mu_{g_t},
		\end{equation*}
		so that, using \eqref{var phi scalar curvature} and \eqref{var vol form},
		\begin{equation*}
		\left.\frac{d}{dt}\right|_{t=0}\mathcal{S}^{\varphi}(g_t)=\int_M[-\Delta_{g}(\mbox{tr}_{g}(h))+\mbox{div}_{g}(\mbox{div}_{g}(h))-\langle h,\mbox{Ric}^{\varphi}_{g}\rangle_{g}]\mu_{g}+\frac{1}{2}\int_MS^{\varphi}_{g}\mbox{tr}_{g}(h)\mu_{g}.
		\end{equation*}
		Using the divergence theorem, from the above we get
		\begin{equation}\label{variation R g(t) senza normalizzazione}
		\left.\frac{d}{dt}\right|_{t=0}\mathcal{S}^{\varphi}(g_t)=\int_M\left\langle h,\frac{S^{\varphi}_{g}}{2}g-\mbox{Ric}^{\varphi}_{g}\right\rangle_{g}\mu_{g}.
		\end{equation}
		In particular, $g$ is a critical point for $\mathcal{S}^{\varphi}$ if and only if
		\begin{equation*}
		\mbox{Ric}_g^{\varphi}=\frac{S^{\varphi}_g}{2}g.
		\end{equation*}
		Since $m\geq 3$ the above is equivalent to
		\begin{equation*}
		\mbox{Ric}^{\varphi}=0.
		\end{equation*} 
		
		Moreover, using the definition \eqref{rescaled total scalar curvature},
		\begin{equation}\label{form che serve per variazione R phi}
		\left.\frac{d}{dt}\right|_{t=0}\bar{\mathcal{S}}^{\varphi}(g_t)=-\frac{m-2}{m}\mbox{vol}_{g}(M)^{-\frac{2(m-1)}{m}}\mathcal{S}^{\varphi}(g)\left.\frac{d}{dt}\right|_{t=0}vol_{g_t}(M)+\mbox{vol}_{g}(M)^{-\frac{m-2}{m}}\left.\frac{d}{dt}\right|_{t=0}\mathcal{S}^{\varphi}_{g_t}.
		\end{equation}
		Then, by plugging \eqref{variazione volume} and \eqref{variation R g(t) senza normalizzazione} into \eqref{form che serve per variazione R phi} we obtain
		\begin{equation}\label{variation rescaled total phi scalar curvature}
		\left.\frac{d}{dt}\right|_{t=0}\bar{\mathcal{S}}^{\varphi}(g_t)=\int_M\left\langle vol_{g}(M)^{-\frac{m-2}{m}}\left[\left(\frac{S^{\varphi}_{g}}{2}-\frac{m-2}{2m}\frac{\mathcal{S}^{\varphi}(g)}{\mbox{vol}_{g}(M)}\right)g-\mbox{Ric}^{\varphi}_{g}\right],h\right\rangle_{g}\mu_g,
		\end{equation}
		and thus $g$ is critical for $\bar{\mathcal{S}}^{\varphi}$ if and only if
		\begin{equation*}
		\mbox{Ric}^{\varphi}_g=\left(\frac{S^{\varphi}_{g}}{2}-\frac{m-2}{2m}\frac{\mathcal{S}^{\varphi}(g)}{\mbox{vol}_{g}(M)}\right)g.
		\end{equation*}
		The above gives
		\begin{equation*}
		\mathring{\mbox{Ric}}^{\varphi}_g=0
		\end{equation*}
		and
		\begin{equation*}
		\mathcal{S}^{\varphi}(g)=S^{\varphi}_g\mbox{vol}_{g}(M).
		\end{equation*}
		
		From \eqref{oper S g phi relazionato a S e E} it is easy to see that $\varphi$ is a critical point of $\mathcal{S}^g$ or $\bar{\mathcal{S}}^g$ if and only if it is a critical point for $E^g$, that is, if and only if $\varphi:(M,g)\to (N,\eta)$ is harmonic. Combining with the results obtained above we conclude the proof.
	\end{proof}

	\begin{rmk}
		We denote by $\mathcal{M}_1$ the subset of $\mathcal{M}$ determined by the Riemannian metrics $g\in\mathcal{M}$ such that $\mbox{vol}_g(M)=1$. We claim that $g\in\mathcal{M}_1$ is critical for $\mathcal{S}^{\varphi}$ in $\mathcal{M}_1$ if and only if is critical for $\bar{\mathcal{S}}^{\varphi}$ in $\mathcal{M}$. Indeed, $T_g\mathcal{M}_1$ can be identified with $S^2_0(M,g)$, the set of traceless two times covariant tensor fields on $(M,g)$ (see \cite{B} at page 118). Hence, proceeding as in the proof of the Proposition above we get, for $g\in\mathcal{M}_1$ and $h\in S^2_0(M,g)$,
		\begin{equation*}
			\left.\frac{d}{dt}\right|_{t=0}\mathcal{S}^{\varphi}(g+th)=\int_M\left(\left.\frac{d}{dt}\right|_{t=0}S^{\varphi}_{g+th}\right)\mu_g=-\int_M\langle h,\mbox{Ric}^{\varphi}_g\rangle_g\mu_g,
		\end{equation*}
		where we integrated by parts and we used that
		\begin{equation*}
			\left.\frac{d}{dt}\right|_{t=0}\mu_{g+th}=\frac{1}{2}\mbox{tr}_g(h)\mu_g=0.
		\end{equation*}
		Then $g$ is critical in $\mathcal{M}_1$ if and only if $\mathring{\mbox{Ric}}^{\varphi}=0$, hence the claim.
	\end{rmk}
	
	\subsection{The total $\varphi$-scalar curvature restricted to a conformal class of metrics}\label{section scalar curvature restricted to conformal classes of metrics}
	
	The following Proposition shows that the problem of finding a conformal metric with constant $\varphi$-scalar curvature on a compact Riemannian manifold admit a variational characterization.
	\begin{prop}\label{prop phi curv scalar const punto critico in classe confome}
		Let $(M,g)$ be a compact Riemannian manifold of dimension $m\geq 3$, $\varphi:M\to N$ a smooth map, where $(N,\eta)$ is a target Riemannian manifold and $\alpha\in\mathbb{R}\setminus\{0\}$. Then the following are equivalent:
		\begin{itemize}
			\item The $\varphi$-scalar curvature is constant;
			\item The metric $g$ is a critical point of the rescaled total $\varphi$-scalar curvature $\bar{\mathcal{S}}^{\varphi}$ restricted to conformal class $[g]\subseteq \mathcal{M}$ of $g$.
		\end{itemize}
	\end{prop}
	\begin{proof}
		We have to characterize the critical points of $\bar{\mathcal{S}}^{\varphi}$ restricted to $[g]$. Let $\widetilde{g}\in [g]$, that is,
		\begin{equation*}
		\widetilde{g}=f^2g
		\end{equation*}
		for some positive function $f$ on $M$. We set $\eta\in\mathcal{C}^{\infty}(M)$ and we define, for $t$ sufficiently small,
		\begin{equation*}
			f_t^2:=f^2+t\eta>0 \quad \mbox{ on } M.
		\end{equation*}
		By setting
		\begin{equation*}
			\widetilde{g}_t=f^2_tg
		\end{equation*}
		we have
		\begin{equation}
			\widetilde{g}_t=\widetilde{g}+th,
		\end{equation}
		where $h:=\eta g$. In particular $\widetilde{g}_t$ is a variation of $\widetilde{g}$ that lies in $[g]$ and all such variations are of this form. Then $\widetilde{g}$ is critical for $\bar{\mathcal{S}}^{\varphi}$ on $[g]$ if and only if, for every $\eta\in\mathcal{C}^{\infty}(M)$,
		\begin{equation*}
			\left.\frac{d}{dt}\right|_{t=0}\bar{\mathcal{S}}^{\varphi}(\widetilde{g}_t)=0.
		\end{equation*}
		Using \eqref{variation rescaled total phi scalar curvature}, since $h=\eta g$ we immediately get
		\begin{align*}
		\left.\frac{d}{dt}\right|_{t=0}\bar{\mathcal{S}}^{\varphi}(\widetilde{g}_t)=&\int_M\mbox{vol}_{\widetilde{g}}(M)^{-\frac{m-2}{m}}\eta \mbox{tr}_{\widetilde{g}}\left[\left(\frac{S^{\varphi}_{\widetilde{g}}}{2}-\frac{m-2}{2m}\frac{\mathcal{S}^{\varphi}(\widetilde{g})}{\mbox{vol}_{\widetilde{g}}(M)}\right)\widetilde{g}-\mbox{Ric}^{\varphi}_{\widetilde{g}}\right]\mu_{\widetilde{g}}\\
		=&\frac{m-2}{2}\mbox{vol}_{\widetilde{g}}(M)^{-\frac{m-2}{m}}\int_M\left(S^{\varphi}_{\widetilde{g}}-\frac{\mathcal{S}^{\varphi}(\widetilde{g})}{\mbox{vol}_{\widetilde{g}}(M)}\right)\eta\mu_{\widetilde{g}}.
		\end{align*}
		Then $\widetilde{g}$ is critical for $\bar{\mathcal{S}}^{\varphi}$ on $[g]$ if and only if
		\begin{equation}\label{g tilde staz rispetto var conf}
		\mathcal{S}^{\varphi}(\widetilde{g})=S^{\varphi}_{\widetilde{g}}vol_{\widetilde{g}}(M),
		\end{equation}
		that is, if and only if $S^{\varphi}_{\widetilde{g}}$ is constant.
	\end{proof}
	
	Let $(M,g)$ be a compact Riemannian manifold of dimension $m\geq 3$, $\varphi:M\to N$ a smooth map, where $(N,\eta)$ is a target Riemannian manifold and $\alpha\in\mathbb{R}\setminus\{0\}$.
	\begin{dfn}
		The {\em $\varphi$-Yamabe invariant of $(M,g)$} as
		\begin{equation}\label{def phi yamabe invariant}
		Y^{\varphi}(g):=\inf_{\widetilde{g}\in [g]}\overline{\mathcal{S}}^{\varphi}.
		\end{equation}
	\end{dfn}
	We are going to show that the Definition above makes sense.
	\begin{dfn}
		For every $u\in\mathcal{C}^{\infty}(M)$, the {\em $\varphi$-conformal Laplacian} is given by
		\begin{equation*}
		L^{\varphi}_g(u):=-\frac{4(m-1)}{m-2}\Delta_g u+S^{\varphi}_g u.
		\end{equation*}
	\end{dfn}
	We denote by $\lambda_1(L^{\varphi}_g)$ the first eigenvalue of $L^{\varphi}_g$. By the variational characterization of $\lambda_1(L^{\varphi}_g)$
	\begin{equation*}
		\lambda_1(L^{\varphi}_g)=\inf_{u\in\mathcal{C}^{\infty}(M), u\not \equiv 0}\frac{\int_M\left(\frac{4(m-1)}{m-2}|\nabla_g u|^2_g+S^{\varphi}_g u^2\right)\mu_g}{\int_M u^2\mu_g}
	\end{equation*}
	it is immediate to get
	\begin{equation}\label{stima grezza primo autovalore lapl conf}
		\lambda_1(L^{\varphi}_g)\geq \inf_MS^{\varphi}>-\infty.
	\end{equation}
	\begin{prop}
		Let $(M,g)$ be a compact Riemannian manifold of dimension $m\geq 3$, $\varphi:M\to N$ a smooth map, where $(N,\eta)$ is a target Riemannian manifold and $\alpha\in\mathbb{R}\setminus\{0\}$. For every $\widetilde{g}\in [g]$ we have
		\begin{equation}\label{tot scalar curv limitata}
		\overline{\mathcal{S}}^{\varphi}(\widetilde{g})\geq \min\{0,\lambda_1(L^{\varphi}_g)\}.
		\end{equation}
		In particular, the $\varphi$-Yamabe invariant of $(M,g)$ is well defined.

	\end{prop}
	\begin{proof}
		Recall the validity of \eqref{transf law phi scalar con f}. By setting $u=e^{-\frac{f}{2}}$, so that
		\begin{equation*}
		\widetilde{g}:=u^{\frac{4}{m-2}}g
		\end{equation*}
		the validity of \eqref{transf law phi scalar con f} gives 
		\begin{equation}\label{eq phi yamabe}
		\frac{4(m-1)}{m-2}\Delta u-S^{\varphi}u+\widetilde{S}^{\varphi}u^{\frac{m+2}{m-2}}=0.
		\end{equation}
		
		We set
		\begin{equation*}
		(L^{\varphi}_g(u),u):=\int_M L^{\varphi}_g(u)u\mu=-\frac{4(m-1)}{m-2}\int_Mu\Delta_g u\mu_g+\int_MS^{\varphi}_g u^2\mu_g.
		\end{equation*}
		Notice that, since $M$ is compact,
		\begin{equation*}
		(L^{\varphi}_g(u),u)=\int_M\left(\frac{4(m-1)}{m-2}|\nabla_g u|^2_g+S^{\varphi}_g u^2\right)\mu_g.
		\end{equation*}
		Clearly,
		\begin{equation}\label{stima su L phi u applicato a u}
		(L^{\varphi}_g(u),u)\geq \lambda_1(L^{\varphi}_g)\|u\|^2_{L^2(M,g)}.
		\end{equation}
		
		Recalling the validity of \eqref{conf change metric vol element}, that in terms of $u$ is given by
		\begin{equation*}
		\widetilde{\mu}=u^{\frac{2m}{m-2}}\mu,
		\end{equation*}
		we immediately get
		\begin{equation*}
		\mbox{vol}_{\widetilde{g}}(M)=\int_Mu^{\frac{2m}{m-2}}\mu.
		\end{equation*}
		Combining the above and \eqref{eq phi yamabe} with the definition of $\bar{\mathcal{S}}^{\varphi}$ we have
		\begin{equation*}
		\overline{\mathcal{S}}^{\varphi}(\widetilde{g})=\left(\int_Mu^{\frac{2m}{m-2}}\mu\right)^{-\frac{m-2}{m}}(L^{\varphi}u,u).
		\end{equation*}
		Then, using \eqref{stima su L phi u applicato a u}, from the above we deduce
		\begin{equation}\label{stima su S bar phi con lamba 1}
		\overline{\mathcal{S}}^{\varphi}(\widetilde{g})\geq\lambda_1(L^{\varphi}_g)\left(\int_Mu^{\frac{2m}{m-2}}\mu\right)^{-\frac{m-2}{m}}\int_Mu^2\mu.
		\end{equation}
		
		From \eqref{stima su S bar phi con lamba 1} we deduce the validity of \eqref{tot scalar curv limitata}. Indeed, if $\lambda_1(L^{\varphi}_g)\geq 0$, from \eqref{stima su S bar phi con lamba 1} we immediately get
		\begin{equation*}
		\overline{\mathcal{S}}^{\varphi}(\widetilde{g})\geq 0.
		\end{equation*}
		
		Notice that, from Jensen's inequality applied to the convex function $t^p$, for $p>1$, we have
		\begin{equation}
		\left(\int u^2\right)^p\leq \int u^{2p},
		\end{equation} 
		that is,
		\begin{equation}\label{Jensen}
		\left(\int u^{2p}\right)^{\frac{1}{p}}\geq \int u^2.
		\end{equation}
		
		Then, if $\lambda_1(L^{\varphi}_g)<0$, using \eqref{stima su S bar phi con lamba 1} and \eqref{Jensen} for
		\begin{equation*}
		p=\frac{m}{m-2}>1,
		\end{equation*}
		we obtain
		\begin{equation*}
		\overline{\mathcal{S}}^{\varphi}(\widetilde{g})\geq \lambda_1(L^{\varphi}_g).
		\end{equation*}
		In conclusion, \eqref{tot scalar curv limitata} holds.
		
		The validity of \eqref{tot scalar curv limitata} shows that
		\begin{equation*}
		\inf_{u\in\mathcal{C}^{\infty}(M), u>0}\overline{\mathcal{S}}^{\varphi}(u^{\frac{4}{m-2}}g)>-\infty.
		\end{equation*}
		To conclude notice that
		\begin{equation*}
		Y^{\varphi}(g)=\inf_{u\in\mathcal{C}^{\infty}(M), u>0}\overline{\mathcal{S}}^{\varphi}(u^{\frac{4}{m-2}}g). \qedhere
		\end{equation*}
	\end{proof}

	\begin{rmk}
		We have just given the definition of the $\varphi$-Yamabe invariant in the compact case. The study of his property (also in the complete non-compact case) will be the subject of some future works.
	\end{rmk}
	
	\subsection{Variational characterization of four dimensional $\varphi$-Bach flat manifolds}\label{section Variational characterization of four dimensional}
	
	Let $M$ be a closed smooth manifold of dimension $m\geq 3$, $\varphi:M\to N$ a smooth map with target a Riemannian manifold $(N,\eta)$ and $\alpha\in\mathbb{R}\setminus\{0\}$.
	\begin{dfn}
		We define the {\em Bach operator} $\mathcal{B}:\mathcal{M}\times \mathcal{F}\to \mathbb{R}$ as
		\begin{equation*}
		\mathcal{B}(g,\varphi)=\mathcal{B}^{\varphi}(g):=\int_M(S_2(A^{\varphi}_g)-\alpha e_2(\varphi))\mu_g=\mathcal{S}_2(g,\varphi)-\alpha E^g_2(\varphi),
		\end{equation*}
		where
		\begin{equation*}
		\mathcal{S}_2(g,\varphi)\equiv \mathcal{S}_2^{\varphi}(g):=\int_MS_2^{\varphi}\mu_g
		\end{equation*}
		and $S_2(A^{\varphi}_g)$ is the second elementary symmetric polynomial in the eigenvalues of the $\varphi$-Schouten tensor $A^{\varphi}_g$ of $(M,g)$.
	\end{dfn}
	\begin{rmk}\label{remark per S 2 schouten}
		The Bach operator can be equivalently written as
		\begin{equation}\label{scrittura comoda per Bach operator}
		\mathcal{B}(g,\varphi)=\int_M\left(\frac{m}{8(m-1)}(S_g^{\varphi})^2-\frac{1}{2}|\mbox{Ric}_g^{\varphi}|_g^2-\frac{\alpha}{2}|\tau^g(\varphi)|^2\right)\mu_g.
		\end{equation}
		
		Indeed, recall that if $A$ is a two times covariant symmetric tensor with eigenvalues $\lambda_1\leq \ldots \leq \lambda_m$, then
		\begin{equation*}
		\mbox{tr}(A)=\sum_i\lambda_i, \quad S_2(A)=\sum_{i<j}\lambda_i\lambda_j.
		\end{equation*}
		Clearly
		\begin{equation*}
		\left(\sum_i\lambda_i\right)^2=\sum_i\lambda_i^2+2\sum_{i<j}\lambda_i\lambda_j,
		\end{equation*}
		hence we have the validity of
		\begin{equation*}
		\mbox{tr}(A)^2=|A|^2+2S_2(A).
		\end{equation*}
		
		Applying the above formula to the $\varphi$-Schouten tensor $A^{\varphi}_g$ of $(M,g)$ we get
		\begin{equation*}
		S_2(A^{\varphi}_g)=\frac{1}{2}[\mbox{tr}(A^{\varphi}_g)^2-|A^{\varphi}_g|_g^2].
		\end{equation*}
		Using the definition of the $\varphi$-Schouten tensor we infer
		\begin{equation*}
		\mbox{tr}(A^{\varphi}_g)=\frac{m-2}{2(m-1)}S_g^{\varphi}, \quad |A^{\varphi}|_g^2=|\mbox{Ric}^{\varphi}|_g^2-\frac{3m-4}{4(m-1)^2}(S^{\varphi}_g)^2,
		\end{equation*}
		then, by plugging those relations into the above we finally conclude
		\begin{equation}\label{valore di S 2 A phi in termini di phi ricci e phi scalar curv}
		S_2(A^{\varphi}_g)=\frac{m}{8(m-1)}(S^{\varphi}_g)^2-\frac{1}{2}|\mbox{Ric}^{\varphi}|_g^2,
		\end{equation}
		and thus \eqref{scrittura comoda per Bach operator} follows immediately, using also the definition \eqref{def bi energia totale} of total bi-energy of $\varphi$.
	\end{rmk}
	
	We are ready to state our main Theorem.
	\begin{thm}\label{theom phi bach flat punti critici bach}
		Let $M$ be a closed orientable four dimensional smooth manifold. Then $(g,\varphi)\in\mathcal{M}\times \mathcal{F}$ is a critical point for $\mathcal{B}$ on $\mathcal{M}\times \mathcal{F}$ if and only if $B^{\varphi}=0$ and $J=0$. Notice that, if $\varphi$ is a submersion a.e. then $(g,\varphi)$ is critical if and only if $(M,g)$ is $\varphi$-Bach flat.
	\end{thm}

	From now on, when is clear from the context, we omit the dependence on $g$. To prove \hyperref[theom phi bach flat punti critici bach]{Theorem \ref*{theom phi bach flat punti critici bach}} we need to evaluate variations of $\mathcal{B}$ both with respect to $\varphi$ and to $g$. In the next Lemma we compute the variations of the Bach functional with respect to variations of the map.
	\begin{lemma}\label{lemma variazione bach operator rispetto mappa}
		Let $(g,\varphi)\in\mathcal{S}\times \mathcal{F}$. Let $v\in\Gamma(\varphi^{-1}TN)$ and $\varphi_t$ as in \eqref{def di d phi t}. Then, in the notations above,
		\begin{equation}\label{var B rispetto mappa}
		\begin{aligned}
		\left.\frac{d}{dt}\right|_{t=0}\mathcal{B}^{\varphi_t}(g)=&\alpha \int_M\left(\frac{m}{2(m-1)}S^{\varphi}\tau(\varphi)^a-\frac{m-2}{2(m-1)}S^{\varphi}_i\varphi^a_i-2R^{\varphi}_{ij}\varphi^a_{ij}-\tau_2(\varphi)\right)v^a\mu_g\\
		&+2\alpha^2\int_M\tau(\varphi)^b\varphi^b_i\varphi^a_iv^a\mu_g.
		\end{aligned}
		\end{equation}
		In particular, if $m=4$
		\begin{equation}\label{var B rispetto mappa per m=4}
		\left.\frac{d}{dt}\right|_{t=0}\mathcal{B}^{\varphi_t}(g)=\alpha\int_M(J,v)\mu,
		\end{equation}
		where $J=J_4$ is defined as in \eqref{def of J 4}.
	\end{lemma}
	\begin{proof}
		Using the relation between $S^{\varphi}=S-\alpha|d\varphi|^2$ we get
		\begin{equation}\label{var S phi rispetto phi}
		\left.\frac{d}{dt}\right|_{t=0}S^{\varphi_t}=-2\alpha \varphi^a_iv^a_i,
		\end{equation}
		hence
		\begin{equation*}
		\left.\frac{d}{dt}\right|_{t=0}(S^{\varphi_t})^2=2S^{\varphi}\left.\frac{d}{dt}\right|_{t=0}S^{\varphi_t}=-4\alpha S^{\varphi}\varphi^a_iv^a_i.
		\end{equation*}
		Then, using the divergence theorem
		\begin{equation}\label{var s phi quadro risp phi}
		\left.\frac{d}{dt}\right|_{t=0}\int_M(S^{\varphi_t})^2\mu=4\alpha\int_M(S^{\varphi}_i\varphi^a_i+S^{\varphi}\varphi^a_{ii})v^a.
		\end{equation}
		
		Using the relation between $\mbox{Ric}^{\varphi}$ and $\mbox{Ric}$ we get
		\begin{equation*}
		\left.\frac{d}{dt}\right|_{t=0}R^{\varphi}_{ij}=-\alpha(\varphi^a_iv^a_j+\varphi^a_jv^a_i),
		\end{equation*}
		hence
		\begin{equation*}
		\left.\frac{d}{dt}\right|_{t=0}|\mbox{Ric}^{\varphi}|^2=2R^{\varphi}_{ij}\left.\frac{d}{dt}\right|_{t=0}R^{\varphi}_{ij}=-4\alpha R^{\varphi}_{ij}\varphi^a_jv^a_i.
		\end{equation*}
		Then, using the divergence theorem
		\begin{equation*}
		\left.\frac{d}{dt}\right|_{t=0}\int_M|\mbox{Ric}^{\varphi}|^2\mu=4\alpha \int_M(R^{\varphi}_{ij,i}\varphi^a_j+R^{\varphi}_{ij}\varphi^a_{ij})v^a.
		\end{equation*}
		Using \eqref{div of phi Ricci} the above yields
		\begin{equation}\label{var ric phi quadro risp phi}
		\left.\frac{d}{dt}\right|_{t=0}\int_M|\mbox{Ric}^{\varphi}|^2\mu=2\alpha \int_M(S^{\varphi}_i\varphi^a_i-2\alpha \tau(\varphi)^b\varphi^b_i\varphi^a_i+2R^{\varphi}_{ij}\varphi^a_{ij})v^a.
		\end{equation}
		
		Using \eqref{var s phi quadro risp phi}, \eqref{var ric phi quadro risp phi} and the definition of bi-tension we get \eqref{var B rispetto mappa}.
	\end{proof}
	Now we deal with variations of the metric.
	\begin{lemma}\label{lemma variazione bach operator rispetto metrica}
		Let $(g,\varphi)\in\mathcal{S}\times \mathcal{F}$. Let $h\in S^2(M)$ and $g_t$ as in \eqref{def di g t}, then
		\begin{equation}\label{var B rispetto metrica}
			\begin{aligned}
			\left.\frac{d}{dt}\right|_{t=0}\mathcal{B}^{\varphi}(g_t)=&\int_M\left[R^t_{ikj}R^{\varphi}_{tk}+\frac{1}{2}\Delta R^{\varphi}_{ij}-\frac{m-2}{4(m-1)}S^{\varphi}_{ij}-\frac{m}{4(m-1)}S^{\varphi}R^{\varphi}_{ij}\right]h_{ij}\mu\\
			&+\alpha\int_M\left(\varphi^a_{kk}\varphi^a_{ij}-R^{\varphi}_{ik}\varphi^a_k\varphi^a_j\right)h_{ij}\mu\\
			&+\int_M\left(\frac{m}{16(m-1)}(S^{\varphi})^2-\frac{1}{4(m-1)}\Delta S^{\varphi}-\frac{1}{4}|\mbox{Ric}^{\varphi}|^2-\frac{\alpha}{4}|\tau(\varphi)|^2\right)\delta_{ij}h_{ij}\mu.
			\end{aligned}
		\end{equation}
		In particular, if $m=4$,
		\begin{equation}\label{var B rispetto metrica m =4}
		\left.\frac{d}{dt}\right|_{t=0}\mathcal{B}^{\varphi}(g_t)=\int_M\langle B^{\varphi},h\rangle\mu.
		\end{equation}
	\end{lemma}
	\begin{proof}
		By definition
		\begin{equation*}
		|\mbox{Ric}^{\varphi}_{g_t}|^2_{g_t}=R^{\varphi}_{ij}(t)R^{\varphi}_{tk}(t)g^{it}(t)g^{jk}(t),
		\end{equation*}
		hence
		\begin{equation*}
		\left.\frac{d}{dt}\right|_{t=0}|\mbox{Ric}^{\varphi}_{g_t}|^2_{g_t}=2(\dot{R}^{\varphi}_{ij}g^{it}+R^{\varphi}_{ij}\dot{g}^{it})R^{\varphi}_{tk}g^{jk}.
		\end{equation*}
		In a local orthonormal coframe, using \eqref{var inversa metrica g} and \eqref{variation phi ricci}, we get
		\begin{equation}\label{variation norm phi ricci}
		\left.\frac{d}{dt}\right|_{t=0}|\mbox{Ric}^{\varphi}_{g_t}|^2_{g_t}=2h_{ik,jk}R^{\varphi}_{ij}-\mbox{tr}(h)_{ij}R^{\varphi}_{ij}-(\Delta h)_{ij}R^{\varphi}_{ij}-2(R^{\varphi})^2_{ij}h_{ij}
		\end{equation}
		
		We have
		\begin{equation*}
		\left.\frac{d}{dt}\right|_{t=0}\int_M|\mbox{Ric}^{\varphi}_{g_t}|^2_{g_t}\mu_{g_t}=\int_M\left(\left.\frac{d}{dt}\right|_{t=0}|\mbox{Ric}^{\varphi}_{g_t}|^2_{g_t}\right)\mu+\int_M|\mbox{Ric}^{\varphi}|^2\left.\frac{d}{dt}\right|_{t=0}\mu_{g_t},
		\end{equation*}
		so that, using \eqref{variation norm phi ricci} and \eqref{var vol form},
		\begin{equation}\label{form che serve per var integrale norma phi ricci}
		\begin{aligned}
		\left.\frac{d}{dt}\right|_{t=0}\int_M|\mbox{Ric}^{\varphi}_{g_t}|^2_{g_t}\mu_{g_t}=&\int_M(2h_{ik,jk}R^{\varphi}_{ij}-\mbox{tr}(h)_{ij}R^{\varphi}_{ij}-(\Delta h)_{ij}R^{\varphi}_{ij}-2(R^{\varphi})^2_{ij}h_{ij})\mu\\
		&+\frac{1}{2}\int_M|\mbox{Ric}^{\varphi}|^2\mbox{tr}(h)\mu,
		\end{aligned}
		\end{equation}
		that is, using the divergence theorem,
		\begin{equation}\label{form che serve per var int norma phi ricci quadro con integrale}
		\left.\frac{d}{dt}\right|_{t=0}\int_M|\mbox{Ric}^{\varphi}_{g_t}|^2_{g_t}\mu_{g_t}=\int_M\left(2R^{\varphi}_{ik,jk}-\delta_{ij}R^{\varphi}_{tk,tk}-R^{\varphi}_{ij,kk}-2(R^{\varphi})^2_{ij}+\frac{1}{2}|\mbox{Ric}^{\varphi}|^2\delta_{ij}\right)h_{ij}\mu.
		\end{equation}
		
		The following commutation relation holds (see \eqref{general commutation rule tensor field})
		\begin{equation*}
		R^{\varphi}_{ik,jk}=R^{\varphi}_{ik,kj}+R^t_{ijk}R^{\varphi}_{tk}+R^t_{kjk}R^{\varphi}_{it}.
		\end{equation*}
		Using the generalized Schur's identity and the definition of $\varphi$-Ricci the above reads
		\begin{equation*}
		R^{\varphi}_{ik,jk}=\frac{1}{2}S^{\varphi}_{ij}-\alpha(\varphi^a_{kk}\varphi^a_i)_j+R^t_{ijk}R^{\varphi}_{tk}+R^{\varphi}_{tj}R^{\varphi}_{it}+\alpha\varphi^a_t\varphi^a_jR^{\varphi}_{it},
		\end{equation*}
		that is,
		\begin{equation}\label{form che serve per var int norma phi ricci quadro}
		R^{\varphi}_{ik,jk}=\frac{1}{2}S^{\varphi}_{ij}-R^t_{ikj}R^{\varphi}_{tk}+(R^{\varphi})^2_{ij}+\alpha R^{\varphi}_{ik}\varphi^a_k\varphi^a_j-\alpha\varphi^a_{kkj}\varphi^a_i-\alpha\varphi^a_{kk}\varphi^a_{ij}.
		\end{equation}
		Moreover, taking the trace of the above we infer
		\begin{equation}\label{form che serve per var int norma phi ricci quadro traccia}
		R^{\varphi}_{tk,tk}=\frac{1}{2}\Delta S^{\varphi}-\alpha\varphi^a_{ttk}\varphi^a_k-\alpha|\tau(\varphi)|^2.
		\end{equation}
		By plugging \eqref{form che serve per var int norma phi ricci quadro} and \eqref{form che serve per var int norma phi ricci quadro traccia} into \eqref{form che serve per var int norma phi ricci quadro con integrale} we obtain
		\begin{align*}
		\left.\frac{d}{dt}\right|_{t=0}\int_M|\mbox{Ric}^{\varphi}_{g_t}|^2_{g_t}\mu_{g_t}=&\int_M\left(S^{\varphi}_{ij}-2R^t_{ikj}R^{\varphi}_{tk}+2\alpha R^{\varphi}_{ik}\varphi^a_k\varphi^a_j-2\alpha\varphi^a_{kkj}\varphi^a_i-2\alpha\varphi^a_{kk}\varphi^a_{ij}-\Delta R^{\varphi}_{ij}\right)h_{ij}\mu\\
		&-\int_M\left(\frac{1}{2}\Delta S^{\varphi}-\alpha\varphi^a_{ttk}\varphi^a_k-\alpha|\tau(\varphi)|^2-\frac{1}{2}|\mbox{Ric}^{\varphi}|^2\right)\delta_{ij}h_{ij}\mu.
		\end{align*}
		
		Using \eqref{var phi scalar curvature} and \eqref{var vol form} we have 
		\begin{equation}\label{variation of integral of S phi quadro}
		\left.\frac{d}{dt}\right|_{t=0}\int_M(S^{\varphi})^2_{g_t}\mu_{g_t}=\int_M\left\langle 2\mbox{Hess}(S^{\varphi})-2S^{\varphi}\mbox{Ric}^{\varphi}+\left(\frac{(S^{\varphi})^2}{2}-2\Delta S^{\varphi}\right)g,h\right\rangle\mu.
		\end{equation}
		
		In conclusion, since
		\begin{equation*}
		\left.\frac{d}{dt}\right|_{t=0}\mathcal{B}^{\varphi}(g_t)=\frac{m}{8(m-1)}\left.\frac{d}{dt}\right|_{t=0}\int_M(S_{g_t}^{\varphi})^2\mu_{g_t}-\frac{1}{2}\left.\frac{d}{dt}\right|_{t=0}\int_M|\mbox{Ric}_{g_t}^{\varphi}|_{g_t}^2\mu_{g_t}-\frac{\alpha}{2}\left.\frac{d}{dt}\right|_{t=0}\int_M|\tau^{g_t}(\varphi)|^2\mu_{g_t},
		\end{equation*}
		by plugging the two relations above and using \eqref{var bi energ rispetto metrica} we obtain the validity of \eqref{var B rispetto metrica}.
		
		For $m=4$ the above gives
		\begin{align*}
		\left.\frac{d}{dt}\right|_{t=0}\mathcal{B}^{\varphi}(g_t)=\int_M\left[R^t_{ikj}R^{\varphi}_{tk}+\frac{1}{2}\Delta R^{\varphi}_{ij}-\frac{1}{6}S^{\varphi}_{ij}-\frac{1}{3}S^{\varphi}R^{\varphi}_{ij}+\alpha\left(\varphi^a_{kk}\varphi^a_{ij}-R^{\varphi}_{ik}\varphi^a_k\varphi^a_j\right)\right]h_{ij}\mu\\
		+\int_M\left(\frac{1}{12}(S^{\varphi})^2-\frac{1}{12}\Delta S^{\varphi}-\frac{1}{4}|\mbox{Ric}^{\varphi}|^2-\frac{\alpha}{4}|\tau(\varphi)|^2\right)\delta_{ij}h_{ij}\mu
		\end{align*}
		Recalling \eqref{scritt alt phi bach m=4}, that is,
		\begin{align*}
		B^{\varphi}_{ij}=&\frac{1}{2}R^{\varphi}_{ij,kk}+R^t_{ikj}R^{\varphi}_{tk}-\frac{1}{6}S^{\varphi}_{ij}-\frac{1}{3}S^{\varphi}R^{\varphi}_{ij}+\left(\frac{(S^{\varphi})^2}{12}-\frac{\Delta S^{\varphi}}{12}-\frac{1}{4}|\mbox{Ric}^{\varphi}|^2\right)\delta_{ij}\\
		&+\alpha\left(\varphi^a_{ij}\varphi^a_{kk}-\frac{1}{2}R^{\varphi}_{ki}\varphi^a_k\varphi^a_j- \frac{1}{2}R^{\varphi}_{kj}\varphi^a_k\varphi^a_i-\frac{1}{4}|\tau(\varphi)|^2\delta_{ij}\right),
		\end{align*}
		we get
		\begin{align*}
		\langle B^{\varphi},h\rangle=&\left[\frac{1}{2}\Delta R^{\varphi}_{ij}+R^t_{ikj}R^{\varphi}_{tk}-\frac{1}{6}S^{\varphi}_{ij}-\frac{1}{3}S^{\varphi}R^{\varphi}_{ij}+\left(\frac{(S^{\varphi})^2}{12}-\frac{\Delta S^{\varphi}}{12}-\frac{1}{4}|\mbox{Ric}^{\varphi}|^2\right)\delta_{ij}\right]h_{ij}\\
		&+\alpha\left(\varphi^a_{ij}\varphi^a_{kk}-R^{\varphi}_{ki}\varphi^a_k\varphi^a_j-\frac{1}{4}|\tau(\varphi)|^2\delta_{ij}\right)h_{ij}.
		\end{align*}
		Then we finally obtain the validity of \eqref{var B rispetto metrica m =4}.
	\end{proof}
	We are finally ready to give the
	\begin{proof}[Proof (of Theorem \ref*{theom phi bach flat punti critici bach}).]
		The proof follows immediately from \eqref{var B rispetto mappa per m=4} and \eqref{var B rispetto metrica m =4}. It remains only to observe that, if $\varphi$ is a submersion a.e. and $\varphi$-Bach vanishes on $M$, then $\varphi$-Bach is is divergence free and from \hyperref[rmk phi bach privo di div per m uguale a 4]{Remark \ref*{rmk phi bach privo di div per m uguale a 4}} we automatically get that $J=0$.
	\end{proof}
	
	Recall that \hyperref[cor phi bach invariante conforme per m uguale a 4]{Corollary \ref*{cor phi bach invariante conforme per m uguale a 4}} tell us that $\varphi$-Bach is a conformal invariant tensor for $m=4$. Now we provide an alternative proof of the conformal invariance of $\varphi$-Bach, at least when $(M,g)$ is a closed orientable four dimensional Riemannian manifold.
	\begin{prop}
		For $m=4$ the functional $\mathcal{B}$ is conformal invariant, that is,
		\begin{equation}\label{conf invariance of functional B}
		\mathcal{B}^{\varphi}(\widetilde{g})=\mathcal{B}^{\varphi}(g),
		\end{equation}
		where
		\begin{equation}
		\widetilde{g}=e^{-f}g,
		\end{equation}
		for some smooth function $f$ on $M$. As a consequence its gradient, that is $\varphi$-Bach, is a conformal invariant tensor.
	\end{prop}
	\begin{proof}
		For $m=4$ the functional $\mathcal{B}$ is given by
		\begin{equation*}
		\mathcal{B}^{\varphi}(g)=\frac{1}{2}\int_M\left(\frac{1}{3}(S_g^{\varphi})^2-|\mbox{Ric}_g^{\varphi}|_g^2-\alpha|\tau^g(\varphi)|^2\right)\mu_g.
		\end{equation*}
		
		To prove \eqref{conf invariance of functional B} it is sufficient to show the validity of
		\begin{equation*}
		Q^{\varphi}_{\widetilde{g}}=Q^{\varphi}_g+\mbox{div}_g[P_g(f)]\mu_g,
		\end{equation*}
		where the $4$-form $Q^{\varphi}_g$ is given by, for every Riemannian metric $g$ on $M$,
		\begin{equation}\label{def di Q phi}
		Q^{\varphi}_g:=\left(\frac{1}{3}(S_g^{\varphi})^2-|\mbox{Ric}_g^{\varphi}|_g^2-\alpha|\tau^g(\varphi)|^2\right)\mu_g
		\end{equation}
		and the vector field $P_g(f)$ is defined as
		\begin{equation*}
		P_g(f)=\left(S^{\varphi}+\Delta f-\frac{1}{2}|\nabla f|^2\right)\nabla f-(2\mbox{Ric}^{\varphi}+\mbox{Hess}(f))(\nabla f,\cdot)^{\sharp}.
		\end{equation*}
		\begin{itemize}
			\item For $m=4$ \eqref{transf law phi scalar con f} reads
			\begin{equation*}
			e^{-f}S^{\varphi}_{\widetilde{g}}=S^{\varphi}+3\Delta f-\frac{3}{2}|\nabla f|^2,
			\end{equation*}
			hence
			\begin{equation}\label{conf change S phi quadro}
			e^{-2f}\frac{1}{3}(S^{\varphi}_{\widetilde{g}})^2=\frac{1}{3}(S^{\varphi})^2+S^{\varphi}(2\Delta f-|\nabla f|^2)+3(\Delta f)^2+\frac{3}{4}|\nabla f|^4-3|\nabla f|^2\Delta f.
			\end{equation}
			\item For $m=4$ \eqref{transf law ricci phi con f} reads
			\begin{equation*}
			\mbox{Ric}^{\varphi}_{\widetilde{g}}=\mbox{Ric}^{\varphi}+\mbox{Hess}(f)+\frac{1}{2}df\otimes df+\frac{1}{2}(\Delta f-|\nabla f|^2) g.
			\end{equation*}
			hence
			\begin{align*}
			e^{-2f}|\mbox{Ric}^{\varphi}_{\widetilde{g}}|_{\widetilde{g}}^2=&|\mbox{Ric}^{\varphi}|^2+|\mbox{Hess}(f)|^2+\frac{1}{4}|\nabla f|^4+(\Delta f-|\nabla f|^2)^2+2R^{\varphi}_{ij}f_{ij}+R^{\varphi}_{ij}f_if_j\\
			&+S^{\varphi}(\Delta f-|\nabla f|^2)+f_{ij}f_if_j+(\Delta f-|\nabla f|^2)\Delta f+\frac{1}{2}(\Delta f-|\nabla f|^2)|\nabla f|^2,
			\end{align*}
			that is, using the definition of $\varphi$-Ricci
			\begin{align*}
			e^{-2f}|\mbox{Ric}^{\varphi}_{\widetilde{g}}|_{\widetilde{g}}^2=&|\mbox{Ric}^{\varphi}|^2+S^{\varphi}(\Delta f-|\nabla f|^2)+2R^{\varphi}_{ij}f_{ij}\\
			&+|\mbox{Hess}(f)|^2+R_{ij}f_if_j+\frac{3}{4}|\nabla f|^4+2(\Delta f)^2-\frac{5}{2}|\nabla f|^2\Delta f+f_{ij}f_if_j-\alpha |d\varphi(\nabla f)|^2.
			\end{align*}
			Using Bochner formula (see, for instance, $(1.176)$ of \cite{AMR})
			\begin{equation*}
			\frac{1}{2}\Delta|\nabla f|^2=|\mbox{Hess}(f)|^2+R_{ij}f_if_j+f_{iij}f_j,
			\end{equation*}
			from the above we conclude
			\begin{equation}\label{norma quadro phi ricci conf change}
			\begin{aligned}
			e^{-2f}|\mbox{Ric}^{\varphi}_{\widetilde{g}}|_{\widetilde{g}}^2=&|\mbox{Ric}^{\varphi}|^2+S^{\varphi}(\Delta f-|\nabla f|^2)+2R^{\varphi}_{ij}f_{ij}+\frac{1}{2}\Delta|\nabla f|^2-f_{iij}f_j\\
			&+\frac{3}{4}|\nabla f|^4+2(\Delta f)^2-\frac{5}{2}|\nabla f|^2\Delta f+f_{ij}f_if_j-\alpha |d\varphi(\nabla f)|^2.
			\end{aligned}
			\end{equation}
			\item For $m=4$ \eqref{tension phi conf change metric con f} reads
			\begin{equation*}
			e^{-f}\tau^{\widetilde{g}}(\varphi)=\tau^g(\varphi)-d\varphi(\nabla f),
			\end{equation*}
			hence
			\begin{equation}\label{norma tau phi quadro conf change}
			e^{-2f}|\tau^{\widetilde{g}}(\varphi)|^2=|\tau(\varphi)|^2+|d\varphi(\nabla f)|^2-2\tau(\varphi)^a\varphi^a_if_i.
			\end{equation}
			\item For $m=4$ \eqref{conf change metric vol element} with $f=2h$ gives
			\begin{equation}\label{conf change metric vol element per m 4 e f}
			e^{2f}\mu_{\widetilde{g}}=\mu.
			\end{equation}
		\end{itemize}
		Using \eqref{conf change S phi quadro}, \eqref{norma quadro phi ricci conf change}, \eqref{norma tau phi quadro conf change} and \eqref{conf change metric vol element per m 4 e f} and also the definition \eqref{def di Q phi} we get
		\begin{align*}
		Q^{\varphi}_{\widetilde{g}}=Q^{\varphi}_g+\left(S^{\varphi}\Delta f+(\Delta f)^2-\frac{1}{2}\Delta|\nabla f|^2-\frac{1}{2}|\nabla f|^2\Delta f-2R^{\varphi}_{ij}f_{ij}+f_{iij}f_j-f_{ij}f_if_j+2\alpha\tau(\varphi)^a\varphi^a_if_i\right)\mu.
		\end{align*}
		To conclude notice that, using \eqref{div of phi Ricci},
		\begin{equation*}
		R^{\varphi}_{ij}f_{ij}=(R^{\varphi}_{ij}f_{i})_j-R^{\varphi}_{ij,j}f_{i}=(R^{\varphi}_{ij}f_{i})_j-\frac{1}{2}S^{\varphi}_jf_j+\alpha \varphi^a_{jj}\varphi^a_if_i,
		\end{equation*}
		\begin{equation*}
		S^{\varphi}\Delta f=(S^{\varphi}f_j)_j-S^{\varphi}_jf_j,
		\end{equation*}
		\begin{equation*}
		f_{iij}f_j=(\Delta f f_j)_j-(\Delta f)^2,
		\end{equation*}
		and, finally,
		\begin{equation}
		f_{ij}f_if_j=\frac{1}{2}|\nabla f|^2_jf_j=\left(\frac{1}{2}|\nabla f|^2f_j\right)_j-\frac{1}{2}|\nabla f|^2\Delta f
		\end{equation}
		hold and thus
		\begin{equation*}
		Q^{\varphi}_{\widetilde{g}}=Q^{\varphi}_g+\left(S^{\varphi}f_j-\frac{1}{2}|\nabla f|^2_j-2R^{\varphi}_{ij}f_{i}+\Delta f f_j-\frac{1}{2}|\nabla f|^2f_j\right)_j\mu.
		\end{equation*}
		This concludes the proof.
	\end{proof}
	\begin{rmk}
		We choose the notation $Q^{\varphi}_g$ in the proof above because when $\varphi$ is constant we recover the $Q$-curvature introduced by Branson, see \cite{Br}.
	\end{rmk}

	\bigskip
	\centerline{\textbf{Aknowledgement}}

	The author wishes to thank Prof. Niel Martin M\o ller for introducing him to the theory of $Q$-curvature and for the valuable discussion.


\begin{thebibliography}{}
		\bibitem[AMR]{AMR} L. J. Alias, P. Mastrolia, M. Rigoli - {\em Maximum principles and geometric applications}, Springer Monographs in Mathematics. Springer, Cham, 2016. xvii+570 pp. ISBN: 978-3-319-24335-1; 978-3-319-24337-5.
		\bibitem[A]{A} A. Anselli - {\em Phi-curvatures, harmonic-Einstein manifolds and Einstein-type structures}, PhD thesis, tutor: M. Rigoli; coordinatore: V. Mastropietro - Milano : Universit\'{a} degli studi di Milano. Dipartimento di matematica "Federigo Enriques", 2020 Jan 28. (32. ciclo, Anno Accademico 2019).
		\bibitem[Ba]{Ba} R. Bach, {\em Zur Weylschen Relativit\"{a}tstheorie und der Weylschen Erweiterung des Kr\"{u}mmungstensorbegriffs}, Mathematische Zeitschrift, 9 (1921), 110-135.
		\bibitem[BaE]{BaE} P. Baird, J. Eells - {\em A conservation law for harmonic maps}, Geometry Symposium, Utrecht 1980 (Utrecht, 1980), pp. 1-25, Lecture Notes in Math., 894, Springer, Berlin-New York, 1981.
		\bibitem[BW]{BW} P. Baird, J. C. Wood - {\em Harmonic morphisms between Riemannian manifolds}, London Mathematical Society Monographs. New Series, 29. The Clarendon Press, Oxford University Press, Oxford, 2003. xvi+520 pp. ISBN: 0-19-850362-8.
		\bibitem[B]{B} A. L. Besse - {\em Einstein manifolds}, Reprint of the 1987 edition. Classics in Mathematics. Springer-Verlag, Berlin, 2008. xii+516 pp. ISBN: 978-3-540-74120-6.
		\bibitem[Br]{Br} T. Branson - {\em Differential operators canonically associated to a conformal structure}, Math. Scand. 57 (1985), no. 2, 293–345.
		\bibitem[C-B]{C-B} Y. Choquet-Bruhat - {\em General relativity and the Einstein equations}, Oxford Mathematical Monographs. Oxford University Press, Oxford, 2009. xxvi+785 pp. ISBN: 978-0-19-923072-3.
		\bibitem[C]{C} J. Corvino - {\em Scalar curvature deformation and a gluing construction for the Einstein constraint equations}, Comm. Math. Phys. 214 (2000), no. 1, 137-189. 
		\bibitem[CCL]{CCL} S. S. Chern, W. H. Chen, K. S. Lam - {\em Lectures on differential geometry}, Series on University Mathematics, 1. World Scientific Publishing Co., Inc., River Edge, NJ, 1999. x+356 pp. ISBN: 981-02-4182-8.
		\bibitem[FM]{FM} A. E. Fischer, J. E. Marsden - {\em Deformations of the scalar curvature}, Duke Math. J. 42 (1975), no. 3, 519-547.
		\bibitem[LP]{LP} J. M. Lee, T. H. Parker - {\em The Yamabe problem}, Bull. Amer. Math. Soc. (N.S.) 17 (1987), no. 1, 37–91.
		\bibitem[LMO]{LMO} E. Loubeau, S. Montaldo, C. Oniciuc - {\em The stress-energy tensor for biharmonic maps}, Math. Z. 259 (2008), no. 3, 503-524.
		\bibitem[M]{M} R. M\"uller - {\em Ricci flow coupled with harmonic map flow}, Ann. Sci. \'Ec. Norm. Sup\'er. (4) 45 (2012), no. 1, 101-142.
		\bibitem[S]{S} R. M. Schoen - {\em Variational theory for the total scalar curvature functional for Riemannian metrics and related topics}, Topics in calculus of variations (Montecatini Terme, 1987), 120-154, Lecture Notes in Math., 1365, Springer, Berlin, 1989.
		\bibitem[W]{W} R. M. Wald - {\em General relativity}, University of Chicago Press, Chicago, IL, 1984. xiii+491 pp. ISBN: 0-226-87032-4; 0-226-87033-2.
	\end{thebibliography}
	\end{document}